\documentclass[11pt,twoside, reqno]{amsart}

\usepackage{amssymb}
\usepackage{amsmath}
\usepackage{mathrsfs}
\usepackage{amsthm}
\usepackage{amsfonts}
\usepackage{txfonts}
\usepackage{enumerate}
\usepackage{hyperref}
\hypersetup{hidelinks}
\usepackage{graphicx}
\usepackage{subfigure}
\usepackage{float}
\usepackage{latexsym,amssymb}
\usepackage{color}
\usepackage{graphicx}
\usepackage{pgf,tikz}
\usepackage{mathrsfs}
\usepackage{fancyhdr}
\usepackage[title]{appendix}
\usetikzlibrary{arrows}

\allowdisplaybreaks
\usepackage{indentfirst}
\usepackage{setspace}
\usepackage{geometry}
\newtheorem{theorem}{Theorem}[section]
\newtheorem{proposition}[theorem]{Proposition}
\newtheorem{corollary}[theorem]{Corollary}
\newtheorem{lemma}[theorem]{Lemma}
\newtheorem{remark}[theorem]{Remark}

\numberwithin{equation}{section} %??????????%

\numberwithin{equation}{section}

\newcommand{\R}{\mathbb{R}}

\newcommand{\C}{\mathbb{C}}
\newcommand{\N}{\mathbb{N}}

\newcommand{\ben}{\begin{eqnarray*}}
\newcommand{\enn}{\end{eqnarray*}}
\newcommand{\pa}{\partial}

\newcommand{\g}{\gamma}

\newcommand{\ve}{\varepsilon}

\newcommand{\al}{\alpha}
\newcommand{\la}{\lambda}
\newcommand{\La}{\Lambda}

\newcommand{\ol}{\overline}

\newcommand{\half}{\frac{1}{2}}

\newcommand{\na}{\nabla}
\newcommand{\be}{\begin{equation}}
\newcommand{\ee}{\end{equation}}
\newcommand{\ba}{\begin{aligned}}
\newcommand{\ea}{\end{aligned}}

\newcommand{\wt}{\tilde{w}}

\def\9{{\infty}}
\def\a{{\alpha}}

\def\g{{\gamma}}

\def\lbb{{\lambda}}
\def\t{{\theta}}

\def\calo{{\mathcal{O}}}
\def\calp{{\mathcal{P}}}

\def\bbr{{\mathbb{R}}}
\def\ve{{\varepsilon}}

\def\wt{\widetilde}

\def\ol{\overline}
\def\({\left(}
\def\){\right)}
\def\<{\langle}
\def\>{\rangle}

\begin{document}

\title
[Multi-bubble blow-up solutions to the half-wave equation]{Construction of multi-bubble blow-up solutions
to the $L^2$-critical half-wave equation}

\author{Daomin Cao}
\address{Institute of Applied Mathematics, Chinese Academy of Sciences, Beijing 100190, and University of Chinese Academy of Sciences, Beijing 100049, P.R. China}
\email{dmcao@amt.ac.cn}
\thanks{}

\author{Yiming Su}
\address{Department of mathematics,
Zhejiang University of Technology, 310014 Zhejiang, China}
\email{yimingsu@zjut.edu.cn}
\thanks{}

\author{Deng Zhang}
\address{School of mathematical sciences, CMA-Shanghai,
Shanghai Jiao Tong University, 200240, Shanghai, China}
\email{dzhang@sjtu.edu.cn}
\thanks{}


\date{}

%\dedicatory{}

\begin{abstract}
This paper concerns the bubbling phenomena for the $L^2$-critical half-wave equation in dimension one.
Given arbitrarily finitely many distinct singularities,
we construct  blow-up solutions concentrating exactly at these singularities.
This provides the first examples of multi-bubble solutions
for the half-wave equation.
In particular, the solutions exhibit the mass quantization property.
Our proof strategy draws upon the modulation method
in \cite{K-L-R} for the single-bubble case,
and explores the localization techniques in \cite{CSZ21,RSZ21} for bubbling solutions to
nonlinear Schr\"odinger equations (NLS).
However, unlike the single-bubble or NLS cases,
different bubbles exhibit the strongest interactions in dimension one.
In order to  get sharp estimates to control strong interactions,
as well as nonlocal effects on localization functions,
we utilize the Carlder\'on estimate
and the integration representation formula of the half-wave operator,
and find that there exists a narrow room
between the orders $|t|^{2+}$ and $|t|^{3-}$
for the  remainder in the geometrical decomposition.
Based on this, a novel bootstrap scheme is introduced
to address the multi-bubble non-local structure.
\end{abstract}

\maketitle

\section{Introduction and main results} \label{Sec-Intro}

\subsection{Introduction}
We are concerned with the $L^2$-critical half-wave equation in dimension one
\begin{align}  \label{equa}
 \left\{\begin{aligned}
 &i\partial_tu=D u-|u|^2u,  \\
 &u(t_0,x)=u_0(x)\in H^{\half}(\R), \quad u: I\times\R\rightarrow\C.
 \end{aligned}\right.
\end{align}
Here, $I\subseteq\R$ is an interval containing the initial time $t_0<0$,
and the square root of Laplacian operator $D=(-\Delta)^{\half}$
is defined via the Fourier transform by
\ben
\widehat{Df}(\xi)=|\xi|\hat{f}(\xi).
\enn

The evolution problem (1.1)
was initially introduced by Laskin \cite{La} in the context of quantum mechanics, generalizing  the Feynman path integral from the Brownian-like to L\'{e}vy-like
quantum mechanical paths.
Additionally, it finds applications in a range of
physical contexts,
such as turbulence phenomena, wave propagation, continuum limits of lattice systems and models for gravitational
collapse in astrophysics, see e.g. \cite{CMMT,ES,FL,KLS,MMT} and references therein.

Apart from its applications in physics, equation (1.1) is of considerable interest from the PDE point of view.
It is a canonical model for an $L^2$-critical PDE with nonlocal dispersion given by a fractional power of the Laplacian. On one hand, equation \eqref{equa} shares similar structures with the classic $L^2$-critical focusing nonlinear Schr\"odinger equation (NLS for short)
\be\label{NLS}
 i\partial_tu+\Delta u+|u|^\frac{4}{d}u=0,\quad (t,x)\in \R\times\R^d.
\ee
Hence, it exhibits rich nondispersive dynamics, such as
singularity formation, solitary waves
and turbulence phenomena.
On the other hand,
because of the nonlocal nature of the operator,
several properties of the NLS \eqref{NLS}
important in the quantitative descriptions of long time dynamics  no longer hold for \eqref{equa},
among which are the Galilean invariance, pseudo-conformal invariance and the exponential decay property of the ground state.
These facts make it rather difficult to study long time dynamics
of solutions to the $L^2$-critical half-wave equation.

The progress on the construction of critical mass blow-up solutions
was achieved by Krieger, Lenzmann and Rapha\"{e}l \cite{K-L-R}.
The cases in dimensions two and three were studied by
Georgiev-Li \cite{G-L-2,G-L-3}.

In the large mass regime,
the formation of singularity is much more complicated.
The famed mass quantization conjecture,
posed by Bourgain \cite{Bou00} and Merle-Rapha\"el \cite{MR05.2},
predicts that
blow-up solutions to $L^2$-critical NLS
decompose into a singular part and
an $L^2$ residual,
and the singular part expands asymptotically as
multiple bubbles, each of which concentrates a mass of
no less than $\|Q\|_{L^2}^2$ at the blow-up point.
Thus, an important step towards the understanding of
the mass quantization conjecture
is to construct multi-bubble blow-up solutions.

One typical example is the pure multi-bubble blow-up solutions without an $L^2$ residual
initially constructed by Merle \cite{M90}.
Another typical blow-up solution is the Bourgain-Wang blow-up solution \cite{BW97},
including a blow-up bubble and an residual.
Very recently,
multi-bubble Bourgain-Wang type blow-up solutions,
predicted by the mass quantization conjecture,
have been constructed by R\"ockner and the second and third authors \cite{RSZ21}.
This class of solutions also relates closely to non-pure multi-solitons (including dispersive part)
as predicted by the soliton resolution conjecture, see \cite{RSZ21}.
We also refer to \cite{CSZ21,CF21} for the uniqueness of multi-bubble
blow-ups and multi-solitons.

In contrast to the extensive studies of multi-bubble solutions
to NLS, to the best of our knowledge,
there seem few examples of multi-bubble blow-up solutions to
the half-wave equation \eqref{equa} in the large mass regime.

One significant challenge arises from the strong interaction between  different bubbles,
owing to the very slow algebraic decay property of the corresponding ground state.
The decay rate indeed reaches the lowest order in dimension one,
as indicated in \eqref{Decay} and \eqref{Decay-s} below.
This is in stark contrast to the rapid exponential decay rate in the NLS case (see \cite[Lemma 6.1]{S-Z}),
as well as the two and three dimensional cases in \cite{G-L-2,G-L-3}.

An additional obstruction, especially in the nonlocal multi-bubble setting, is
the nonlocal effect on the localization functions introduced in the multi-bubble construction in \cite{CSZ21,RSZ21,S-Z}.
In the context of NLS,
the localization functions cooperate well with the local derivative operators,
and exhibit favorable algebraic identities
and fast decay orders.
These properties are important when deriving the key monotonicity property of the generalized energy functional.
However, because of the non-local effect of the half-wave operator,
one has to control extra errors
particularly arising from the commutator estimates and
the large spatial regime,
which indeed require sharp estimates for obtaining
required decay orders essential to close bootstrap estimates.

The aim of this paper is to study the multi-bubble phenomena for the half-wave equation \eqref{equa} in the nonlocal setting.
We construct multi-bubble blow-up solutions
concentrating exactly at any given distinct singularities,
each bubble behaves asymptotically like the self-similar critical mass solutions constructed in \cite{K-L-R}.
To the best of our knowledge,
this provides the first examples of multi-bubble solutions to \eqref{equa}
in the large mass regime.
In particular,
the constructed solutions exhibit the mass quantization phenomena,
concentrating exactly the ground state mass at each singularity.

Our proof draws upon the modulation method in \cite{K-L-R}
for the construction of single-bubble blow-up solutions in the nonlocal setting,
and explores the localization techniques in \cite{CSZ21,RSZ21,S-Z}
to construct multi-bubble solutions to NLS in the large mass regime.
Much efforts has been dedicated to overcoming
the aforementioned obstructions.
In order to exploit sufficient decay orders and get sharp estimates,
we utilize the Carlder\'on estimate \cite{C65}
and the integration representation formula of fractional operators.
Quite delicately,
we find that there exists a narrow room
between the orders $|t|^{2+}$ and $|t|^{3-}$
for the  remainders in the geometrical decomposition.
This allows to develop a novel bootstrap scheme
tailored to address the
present multi-bubble non-local situation.

\subsection{Existing results}
It is well known that the half wave equation \eqref{equa} is locally well-posed in the energy space $H^\half(\R)$, see e.g. \cite{K-L-R}.
For every initial datum $u_0\in H^\half(\R)$,
there exists a unique solution $u\in C([t_0,T);H^\half(\R))$ to \eqref{equa}
with $t_0<T\leq +\9$ being the maximal forward time of existence,
and the following blow-up alternative property holds:
\ben
T < +\infty\quad implies\quad
\lim_{t\rightarrow T^-}\|u(t)\|_{H^\frac{1}{2}}=\infty.
\enn
Moreover, for $s > \half$, the $H^s$ regularity on the initial data can be propagated by the flow.

Equation \eqref{equa} also admits a number of symmetries and conservation laws.
It is invariant under the translation, scaling, phase rotation,
i.e.,
if $u$ solves (\ref{equa}),
then so is
\begin{align*}
v(t,x)=\lambda_0^{\frac{1}{2}}
u (\lambda_0 t+t_0,\lbb_0 x+x_0)
e^{i\g_0},
\end{align*}
with $v(0,x)=\lambda_0^{\frac{1}{2}} u_0 (\lbb_0 x+x_0) e^{i\g_0}$,
where
$(\lambda_0, t_0, x_0, \g_0) \in \bbr^+ \times \bbr \times \bbr\times \bbr$.
In particular,
the $L^2$-norm of solutions is preserved under the scaling symmetry,
and thus
\eqref{equa} is  called the {\it $L^2$-critical} equation.
The conservation laws related to \eqref{equa} contain
\begin{align*}
   & {\rm Mass}:\ \ M(u)(t):=\int_{\R}|u(t,x)|^2dx=M(u_0).    \\
   & {\rm Energy}:\ \ E(u)(t) := \half\int_{\R}|D^\half u(t,x)|^2dx-\frac{1}{4}\int_{\R}|u(t,x)|^{4}dx=E(u_0).  \\
   & {\rm Momentum}: \ \ P(u)(t):={\rm Im} \int_{\R}\nabla u(t,x)\bar{u}(t,x)dx
       =  P(u_0).
\end{align*}

One important role in the description of the long time dynamics of solutions to \eqref{equa}
is played by the \emph{ground state} $Q$, which  is the unique positive even solution to
\be\label{Ground}
DQ+Q-Q^3=0,\quad Q\in H^\half(\R).
\ee
The existence of $Q$ follows from standard variational arguments.
However, the uniqueness of $Q$ is an intricate problem  since  shooting arguments and ODE techniques are not applicable to nonlocal operators,
see \cite{F-L, F-L-S}.
An intriguing feature is
that the ground state only exhibits a slow algebraic decay (\cite{F-L, KMR}), i.e.,
\be\label{Decay}
Q(x)\sim\frac{1}{1+|x|^2}.
\ee
Moreover, by the approach of Weinstein \cite{W83},
the ground state is related to the best constant of a sharp Gagliardo-Nirenberg type inequality
\ben
\|f\|_{L^4}^4\leq \frac{\|f\|_{L^2}^2\|D^\half f\|_{L^2}^2}{2\|Q\|_{L^2}^2},\quad \forall f\in H^\half(\R).
\enn
Thus, a standard argument shows that, if the initial datum has a \emph{subcritical mass}, i.e., $\|u_0\|^2_{L^2}<\|Q\|^2_{L^2}$, then the corresponding solution exists globally in time.
However, unlike the NLS scattering below the ground state (\cite{D15}),
the half-wave equation has non-scattering traveling waves with even arbitrarily small mass, see e.g. \cite{BGLV, K-L-R}.

When the initial datum has a \emph{critical mass}, i.e., $\|u_0\|^2_{L^2}=\|Q\|^2_{L^2}$, the corresponding solution may form singularities in finite time.
However, differently from the $L^2$-critical NLS, the existence of finite time blow-up solutions
with minimal mass is a non-trivial result,
due to the absence of the pseudo-conformal invariance for \eqref{equa}.
The construction of critical mass blow-up solutions to \eqref{equa}
was initialed by Krieger, Lenzemann and Rapha\"{e}l \cite{K-L-R},
in which the authors developed a robust dynamical approach
for the construction of critical mass blow-up elements in the setting of nonlocal dispersion.
Recently, the construction of critical mass blow-up solutions
has been generalized by  Georgiev and Li \cite{G-L-2,G-L-3} to dimensions two and three,
with an additional radial assumption.
Note that, these minimal mass blow-up solutions admit the  blow-up speed
$\|D^\half u(t)\|_{L^2}\sim (T-t)^{-1}$ as $t\rightarrow T^-$.

When the mass of initial data is slightly above the critical mass,
a different blow-up scenario has been observed by Lan \cite{Lan}
for the general $L^2$-critical fractional Schr\"odinger equation
\begin{align}  \label{equa-gen}
 &i\partial_tu=(-\Delta)^{s} u-|u|^{\frac{4s}{d}}u=0,  \quad (t,x)\in \R\times\R^d, \; s\in[\half, 1).
\end{align}
See, e.g., \cite{BGV,CP,FGO,GG,GLPR,GH,HW,HS} and references therein
for the relevant well-poseness results.
For $s$ close to 1, Lan \cite{Lan} constructed blow-up solutions
with the rate $\|D^{s} u(t)\|_{L^2}\leq C \sqrt{\frac{|log(T-t)|^{\frac18}}{T-t}}$
as $t\rightarrow T^-$,
where $T$ is the blow-up time.
Unlike the critical mass blow-up solutions,
this class of solutions is stable with respect to the perturbation of initial data.
In some sense,  they resemble the stable {\it log-log} blow-up solutions
to the $L^2$-critical NLS,
which have been extensively studied by Merle and Rapha\"{e}l \cite{MR03,MR05,MR06}.
However, for the half-wave equation \eqref{equa},
it remains unknown whether this class of blow-up solutions exist.
See also \cite{BHL} and \cite{D19}
for the general blow-up criteria for radial and non-radial initial data, respectively.

In the large mass regime,
multi-bubble blow-up solutions have attracted significant interest in the literature.
In the context of $L^2$-critical NLS,
we refer to \cite{M90} for the initial construction of multi-bubble solutions,
\cite{F17,FM23+} for the construction of multi-bubble log-log blow-up solutions,
and \cite{MP18} for the multi-bubbles concentrating at the same point.
Very recently, multi-bubble Bourgain-Wang type solutions
predicted by the mass quantization conjecture
have been constructed in \cite{RSZ21}.
These solutions also relate, via the pseudo-conformal transformation,
to non-pure multi-solitons predicted by the soliton resolution conjecture.
Bubbling phenomena also exhibit for stochastic nonlinear Schr\"odinger equations
driven by Wiener process, see \cite{S-Z}.
Additionally, we would like to refer to \cite{CM18,JM20,MRT15,S-Z} for multi-bubble solutions
in the context of various other models.

\subsection{Main result}
The main result of this paper is formulated in Theorem \ref{thm-main} below,
concerning the multi-bubble blow-up solutions to the half-wave equation \eqref{equa}.

\begin{theorem} [Multi-bubble blow-up solutions]  \label{thm-main}
Let $\{x_k\}_{k=1}^K$ be arbitrary K distinct points in $\R$, $K\in \mathbb{N}^+$,
$\{\t_k\}_{k=1}^K\subseteq\R^K$and $\omega>0$.
Then, there exist $t_0<0$ and a solution $u\in C([t_0,0);H^{\half}(\R))$ to \eqref{equa}
which blows up at time $T=0$ and has the asymptotic behavior
\ben
u(t,x)-\sum_{k=1}^{K}\frac{1}{\lbb_k^{\half}(t)}Q\left(\frac{x-\al_k(t)}{\lbb_k(t)}\right)
e^{i\g_k(t)}\rightarrow 0\quad in \quad L^2(\R), \quad as \quad t\rightarrow 0^-,
\enn
where the parameters satisfy
\ben
\lbb_k(t)=\omega t^2+\calo(|t|^{4-2\delta}),\quad \al_k(t)=x_k+\calo(|t|^{3-\delta}),\quad \g_k(t)=\frac{1}{\omega|t|}+\t_k+\calo(|t|^{1-2\delta}), \quad 1\leq k\leq K.
\enn
for $\delta$ being any small positive number.
\end{theorem}

As a consequence,
we have the mass quantization property for the constructed solutions to \eqref{equa}.

\begin{corollary} (Mass quantization)
Let $u$ be the solution to \eqref{equa} constructed in Theorem \ref{thm-main}.
Then,
we have that for $t\to 0^-$,
\begin{align*}
        |u(t)|^2 \rightharpoonup \sum\limits_{k=1}^K \|Q\|_{L^2}^2 \delta_{x=x_k}
\end{align*}
and
\begin{align*}
        u(t) \to 0\ \ in\ L^2(\bbr - \bigcup\limits_{k=1}^K B(x_k, r)),
\end{align*}
where $B(x_k, r)$ denotes the open interval centered at $x_k$ with
any given small radius $r>0$, $1\leq k\leq K$.
In particular,
\ben
\|u\|_{L^2}^2=K\|Q\|_{L^2}^2.
\enn
\end{corollary}

\begin{remark}
To our best knowledge,
Theorem \ref{thm-main} provides the first examples of multi-bubble solutions
and  the mass quantization phenomenon for the half-wave equation
in the non-local setting.

Moreover,
Theorem \ref{thm-main} mainly treats the case where the singularities are arbitrary,
and the asymptotic frequencies in the scaling parameters $\lbb_k$ are the same.
With slight modifications,
our arguments also apply to the case where the asymptotic frequencies are different,
i.e., $\lbb_k(t)\sim\omega_k t^2$ with $\omega_k$ of small variations as in \cite{RSZ21,SZ20}.
We refer the reader to {\rm Case (II)} in \cite{S-Z}
for detailed arguments.

We would also  expect the method developed in this work
is applicable to the multi-bubble problem for  the generalized problem \eqref{equa-gen}
in high dimensions, possibly with additional effort to deal with
the nonsmoothness of the nonlinearity.
\end{remark}

As mentioned above,  our proof  utilizes the strategy developed in \cite{CSZ21,K-L-R,RSZ21,S-Z}.
Let us present some comments on the differences
between the multi-bubble/single-bubble
and non-local/local cases.

\paragraph{\bf $(i)$ Distinction between multi-bubble and single-bubble cases}
One major difference between two cases lies in the
strong interaction between difference bubbles,
which does not occur in the single-bubble case in \cite{K-L-R}.

As a matter of fact,
different bubbles exhibit the strongest interactions for the half-wave equation.
More precisely,
the bubbling interactions are determined by the decay rate of the ground state,
which solves the
fractional elliptic problem
\ben
(-\Delta)^{s}Q+Q-|Q|^{\frac{4s}{d}}Q=0.
\enn
One has
\be\label{Decay-s}
 Q(x)\sim\frac{1}{1+|x|^{d+2s}}, \ \ |x|\gg 1.
\ee
Note that for the half-wave equation where $s=1/2$,
the decay rate of the ground state reaches the lowest order when $d=1$.
This is different from the high dimension cases in \cite{G-L-2,G-L-3}.

Because of the slow decay of the ground state \eqref{Decay},
the interactions between different bubbles to \eqref{equa}
have the decay rate $\calo(|t|^4)$,
see the decouple estimates in Lemma \ref{lem-dep} below.
This restricts the decay rate of several functionals.
One typical restriction occurs for the localized mass
$Re\<U_k, R\>$,
where $U_k$ is the $k$-th localized blow-up profile given by \eqref{U-Uj} below, $1\leq k\leq K$.
In the present multi-bubble case,
one has
\begin{align*}
   |Re\<U_k, R\>| \lesssim  \int_0^t \|R\|_{L^2} ds + |t|^4.
\end{align*}
While in the single-bubble case,
due to the mass conservation law,
much better decay rate can be gained
\begin{align*}
   |Re\<U_k, R\>| \lesssim \|R\|^2 _{L^2}.
\end{align*}

The strong interaction effects can be seen more clearly
from the generalized energy  $\mathfrak{I}$ defined in \eqref{def-I} below.
On one hand,
in the derivation of the corresponding monotonicity property,
one roughly has
\begin{align*}
   \frac{d \mathfrak{I}}{dt} \geq  C\frac{\|R(t)\|_{L^2}^2}{|t|^{3}}
   +\calo\( \frac{\|R(t)\|_{L^2}^2}{|t|^{3-\delta}} + |t|^{3-\delta}
   + \frac{\int_t^0 \|R(s)\|_{L^2} ds}{|t|^3}\|R(t)\|_{L^2}\),
\end{align*}
where $\delta>0$ is a very small constant.
Thus, in order to preserve the main order
one needs $\|R(t)\|_{L^2}=\calo(t^2)$. Moreover, in order to close the following  bootstrap estimate for $\g_k(t)$, $1\leq k\leq K$,
\ben
 |\g_{k}(t) + \frac{4}{\omega^2 t} - \theta_k|
 \lesssim\int_{t}^{0} \frac{Mod(s)}{s^4}ds
 \lesssim\int_{t}^{0}s^{-4}\int_{s}^{0}\|R(\tau)\|_{L^2}d\tau ds,
\enn
one needs the upper bound of the remainder $\|R(t)\|_{L^2} \lesssim |t|^{2+}$.
On the other hand,
in order to close the following  bootstrap estimate for $\|R(t)\|_{L^2}$
\begin{align*}
    \frac{\|R\|_{L^2}^2}{|t|^2}
   \lesssim \mathfrak{I}(t)
   \lesssim \int_t^0 |s|^{3-\delta} ds \lesssim |t|^{4-\delta},
\end{align*}
one  has the lower bound of the remainder $\|R(t)\|_{L^2} \gtrsim |t|^{3-}$. Thus, the above heuristical arguments lead to
\begin{align*}
  |t|^{3-}  \lesssim \|R\|_{L^2} \lesssim |t|^{2+}.
\end{align*}
This narrow gap between orders $\calo(|t|^{2+})$ and $\calo(|t|^{3-})$
provides opportunity to develop a new bootstrap scheme
for the present multi-bubble non-local situation.

It might be also worth noting that
our modified blow-up profile $Q_k$ has the expansion
\ben
-\frac{i}{2}b_k^2\partial_{b_k}Q_k-ib_kv_k\partial_{v_k}Q_k+ib_k\Lambda Q_k-iv_k Q^\prime_k-DQ_k-Q_k+|Q_k|^2Q_k=\calo(|t|^4),
\ \ 1\leq k\leq K.
\enn
It is different from the NLS case,
where the modification procedure is unnecessary
owing to the rapid exponential decay rate of the ground state.
Moreover,
though the $\calo(|t|^4)$ decay rate is weaker than the $\calo(|t|^5)$ rate of the more refined profile
constructed in the single-bubble case \cite{K-L-R},
it is sufficient to close the bootstrap arguments,
as the main order is determined by the bubble interactions.

\paragraph{\bf $(ii)$ Distinction between non-local and local cases}
In order to address the above-mentioned heuristic arguments,
a main challenge is to deal with the non-local effect of the half-wave operator.

Actually, when constructing bubbling solutions to NLS,
we make use of the localization functions in \cite{S-Z}.
These functions permit to construct the previously mentioned localized mass and the generalized energy
adapted to the multi-bubble case.
They also cooperate well with the local derivative operators,
enabling several algebraic cancellations
and fast decay orders for deriving the key monotonicity property of the generalized energy.

However, in the current non-local context,
it is rather difficult to decouple different localization functions.
Moreover, new errors arise from the commutator estimates and in the large spatial regime.
In order to exploit sufficient decay orders,
we apply the integration presentation formula of the nonlocal operator to get sharp pointwise estimates,
and the Calder\'on estimate to control the remainder
in the geometrical decomposition.
These constitute the main technical parts of the present work.
The most delicate part of the proof lies in the derivation of the monotonicity of generalized energy.
Careful treatment has been performed  for the nonlocal operator and localization functions,
in order to decouple the generalized energy into different localized parts with sufficiently high-order errors.
\medskip

\paragraph{\bf Organization of paper}
The rest of this paper is organized as follows.
In Section \ref{Sec-Mod} we derive the geometrical decomposition of solutions
and the estimates of the corresponding modulation equations.
Then, Sections \ref{Sec-Boot}-\ref{Sec-Proof-Boot}
are dedicated to the proof of the crucial bootstrap estimates
of modulation parameters and remainder.
In particular, in Section \ref{Sec-Gen-Energy},
we construct the generalized energy functional adapted to multi-bubble case,
and prove the monotonicity and coercivity controls,
which are the key ingredients towards the derivation of bootstrap estimates.
In Section \ref{Sec-Exist-Multi}, we prove the main results in Theorem \ref{thm-main}.
At last,  the Appendix contains a collection of the preliminary lemmas used in the proof,
including the properties of fractional Laplacian operators
and the linearized operator around the ground state.
\medskip

\paragraph{\bf Notations}
The Fourier transform of a function $f$ is denoted by $\mathcal{F}(f)=\hat{f}$.
The fractional Laplacian operator $D^s=(-\Delta)^{\frac{s}{2}}$
for $s\geq 0$ can be defined via the Fourier transform by
\be\label{def-ds}
\widehat{D^sf}(\xi)=|\xi|^s\hat{f}(\xi).
\ee

We use  $\dot{H}^{s}(\bbr)$ and $H^{s}(\bbr)$
for the homogeneous and nonhomogeneous Sobolev spaces,
respectively, for  $s\in \bbr$.
Let $L^p(\R)$ denote
the space of $p$-integrable complex-valued functions,
$\<v,w\>:=\int_{\bbr} v(x) \ol w(x)dx$ the inner product of the Hilbert space
$L^2(\R)$,
and $\mathscr{S}(\R)$ the Schwartz space.
We use the Japanese bracket $\<x\>:=(1+|x|^2)^{1/2}$. We also denote $\Lambda f(x):= \frac{1}{2}f + xf^\prime$ as the $L^2$-critical scaling operator.

In the following, we sometimes use the multi-variable calculus notation
such as $\nabla f = f^\prime$ and $\Delta f = f^{\prime\prime}$
for functions $f : \R \rightarrow \R$ to improve the readability of certain formulae derived
below.

We use  $f=\calo(g)$ to denote $|f|\leq C|g|$,
where the positive constant $C$  is allowed to depend on universally fixed quantities only.
Throughout the paper, the positive constant $C$ may change from line to line.

\section{Geometrical decomposition and modulation equations}    \label{Sec-Mod}
In this section, we first construct the approximate blow-up profiles.
Then, we derive the geometrical decomposition
and the corresponding estimates for the modulation equations of geometrical parameters.
This is the starting point to carry out the modulation analysis.

\subsection{Approximate blow-up profiles}
Unlike the NLS,
where the pseudo-conformal invariance permits a direct derivation of
an explicit blow-up solution from the ground state,
such a direct approach is unavailable for equation \eqref{equa}.
In \cite{K-L-R}, a high-order approximate blow-up profile $Q_{\mathcal{P}}$ has been constructed
as a substitute for $Q$.
In this subsection, we introduce the approximate blow-up profiles
in the  multi-bubble context,
which serve as the fundamental building blocks for constructing the approximate bubbling solutions.
For convenience,
we identify a complex-valued function $f: \R \rightarrow \C$ with the vector valued function $\mathbf{f}: \R\rightarrow\R^2$ as
\ben
\mathbf{f}=
\left[\begin{array}{c}
  {\rm Re}f \\
  {\rm Im}f
\end{array}\right].
\enn

By applying this notation, we have the existence of approximate blow-up profiles $\textbf{Q}_k=[{\rm Re}Q_k, {\rm Im}Q_k]^T$ around the ground state $\textbf{Q}=[Q,0]^T$ in the following lemma.

\begin{lemma}[Approximate blow-up profiles]\label{lem-app}
For every $1\leq k\leq K$,
for $b_k$ and $v_k$ sufficiently small,
there exists a smooth function $Q_k$ of the vector form
\begin{align}   \label{Qk-approx}
\textbf{Q}_k=\textbf{Q}+b_k\textbf{R}_{1,0}+v_k\textbf{R}_{0,1}+b_kv_k\textbf{R}_{1,1}
+b_k^2\textbf{R}_{2,0}+b_k^3\textbf{R}_{3,0},
\end{align}
such that the following holds:
$Q_k$ solves the equation
\be\label{eq-q}
-\frac{i}{2}b_k^2\partial_{b_k}Q_k-ib_kv_k\partial_{v_k}Q_k+ib_k\Lambda Q_k-iv_k Q^\prime_k-DQ_k-Q_k+|Q_k|^2Q_k=-\Psi_k,
\ee
with the remainder term $\Psi_k$ satisfying that for $m\in\N$ and $\nu=0,1,2$,
\be\label{est-psi}
\|\Psi_k\|_{H^m}\leq C(m)(b_k^4+v_k^2), \quad|\nabla^{\nu}\Psi_k(x)|\leq C(b_k^4+v_k^2)\langle x\rangle^{-2}, \quad  x\in \R .
\ee
Moreover, the vector functions $\{\textbf{R}_{i,j}\}_{1\leq i\leq 3, 0\leq j\leq 1}$ have the symmetry structure
\ben
\mathbf{R_{1,0}}=
\left[\begin{array}{c}
  0 \\
  \mathrm{even}
\end{array}\right], \; \mathbf{R_{0,1}}=
\left[\begin{array}{c}
  0 \\
  \mathrm{odd}
\end{array}\right],\; \mathbf{R_{1,1}}=
\left[\begin{array}{c}
  \mathrm{odd} \\
  0
\end{array}\right],\; \mathbf{R_{2,0}}=
\left[\begin{array}{c}
  \mathrm{even} \\
  0
\end{array}\right],\; \mathbf{R_{3,0}}=
\left[\begin{array}{c}
  0 \\
  \mathrm{even}
\end{array}\right]
\enn
and the following regularity and decay bounds
\ben\ba
& \|\textbf{R}_{i,j}\|_{H^m}+\|\Lambda\textbf{R}_{i,j}\|_{H^m}+\|\Lambda^2\textbf{R}_{i,j}\|_{H^m}\leq C(m),\quad  m\in \N,\\
&|\textbf{R}_{i,j}|+|\Lambda \textbf{R}_{i,j}|+|\Lambda^2 \textbf{R}_{i,j}|\leq C\langle x\rangle^{-2},\quad  x\in \R.
\ea\enn

The proof of Lemma \ref{lem-app} is similar  to that of Proposition 4.1 in \cite{K-L-R},
thus it is omitted here for simplicity.

\end{lemma}
\begin{remark}\label{rmk-1}
The parameters $b_k$ and $v_k$ are assumed to be small.
In fact, as we shall see in the bootstrap estimates in Section 3,
the parameters $b_k, v_k$ depend on $t$
and satisfy the a prior bounds that
$b_k(t)\sim |t|$ and $v_k(t)\sim b_k^2\sim t^2$,
which are sufficiently small for $t$ close to $0$.

Let us mention the difference between the approximate profile $Q_k$ in \eqref{eq-q}
with the one constructed in   \cite[Proposition 4.1]{K-L-R}.
In our case, we expand $Q_k$ to the order $\calo(b_k^3)$,
so that it solves equation \eqref{eq-q} with the remainder $\Psi_k$ of the order  $\calo(b_k^4)$.
In \cite{K-L-R}, the approximate profile $Q_\mathcal{P}$ is expanded to a higher order to solve equation \eqref{eq-q}
with the remainder of the order $\calo(b_k^5)$.
The reason that the current expansion \eqref{Qk-approx} is sufficient for our construction is
that the strong interaction between different bubbles gives rise to the dominant errors of order $\calo(t^4)$.

\end{remark}
\begin{remark}
As we can see from \cite{K-L-R}, the choice of the functions $\{\textbf{R}_{i,j}\}$ can be
\ben
\mathbf{R_{1,0}}=
\left[\begin{array}{c}
  0 \\
  S_1
\end{array}\right], \; \mathbf{R_{0,1}}=
\left[\begin{array}{c}
  0 \\
  G_1
\end{array}\right],\; \mathbf{R_{1,1}}=
\left[\begin{array}{c}
  G_2 \\
  0
\end{array}\right],\; \mathbf{R_{2,0}}=
\left[\begin{array}{c}
  S_2 \\
  0
\end{array}\right],\; \mathbf{R_{3,0}}=
\left[\begin{array}{c}
  0 \\
  S_3
\end{array}\right],
\enn
with $G_1$ and $S_1$ defined as in \eqref{eq-g1} and \eqref{eq-s1}, respectively,
$G_2$ being the unique odd  solution to
\ben
L_+G_2=G_1-\La G_1+\nabla S_1+2S_1G_1Q,
\enn
$S_2$ being the unique even solution to
\ben
L_+S_2=\half S_1-\Lambda S_1+S_1^2Q,
\enn
and $S_3$ being the unique odd solution to
\ben
L_-S_3=-S_2+\Lambda S_2+2S_1S_2Q+S_1^3.
\enn

Thus, for $1\leq k\leq K$, we have the expansion
\be\label{Q-asy}
Q_k(x)=Q(x)+ib_kS_1(x)+iv_kG_1(x)+b_kv_kG_2(x)+b_k^2S_2(x)+b_k^3S_3(x).
\ee
A straightforward computation together with the fact $\<S_1, S_1\> + 2\<Q, S_2\> = 0$
also yields that
\be\label{Q-l2}
\|Q_k\|^2_{L^2}=\|Q\|_{L^2}^2+\calo(b_k^4+v_k^2).
\ee
\end{remark}

\subsection{Geometrical decomposition}
For every $1\leq k\leq K$,
define the modulation parameters by
$\mathcal{P}_k:=(\lbb_k, b_k, v_k, \alpha_k,\gamma_k) \in \bbr^{5}$.
Let
$\calp:= (\calp_1, \cdots, \calp_K) \in \bbr^{5K}$.
Similarly,
let $\lbb := (\lbb_1, \cdots, \lbb_K) \in \bbr^K$
and
$\lbb:=\sum_{k=1}^{K}\lbb_{k}$.
Similar notations are also used
for the remaining parameters.

Let $\omega>0$, $T<0$ and $x_k$ and $\t_k$ be as in Theorem \ref{thm-main}.
Set
\begin{align*}
    U(T,x)
    :=   \sum_{k=1}^{K} U_k(T,x) \quad with\quad
   U_k(T,x) := \lbb_k^{-\frac 12}(T) Q_k\(T,\frac{x-\a_k(T)}{\lbb_k(T)}\) e^{i\g_k(T)},
\end{align*}
where $Q_k$ is defined as in Lemma \ref{lem-app} and \begin{align} \label{PjT}
   \calp_k(T)=(\frac{\omega^2}{4} T^2, -\frac{\omega^2}{2} T, \frac{\omega^2}{4} T^2, x_k,  -\frac{4}{\omega^2 T} + \t_k).
\end{align}

\begin{proposition} (Geometrical decomposition)   \label{Prop-dec-un}
Let $u(t)$ be a solution to \eqref{equa} with $u(T)=U(T)$.
For $T$ sufficiently close to 0,
there exist $T_*<T$
and
unique modulation parameters
$\mathcal{P}
\in C^1([T_*, T]; \bbr^{5K})$,
such that
$u$
can be  decomposed into a main blow-up profile and a remainder
\begin{align} \label{u-dec}
    u(t,x)=U(t,x)+R(t,x),\ \ t\in [T_*, T],
\end{align}
where the main blow-up profile
\begin{align}  \label{U-Uj}
    U(t,x)
    =   \sum_{k=1}^{K} U_k(t,x) \quad with\quad
   U_k(t,x) = \lbb_k^{-\frac 12}(t) Q_k\(t,\frac{x-\a_k(t)}{\lbb_k(t)}\) e^{i\g_k(t)},
\end{align}
and the following orthogonality conditions hold on $[T_*, T] $
for every $1\leq k\leq K$:
\be\ba\label{ortho-cond-Rn-wn}
&{\rm Re}\< \widetilde{S}_k(t), R(t)\>=0,\ \
{\rm Re} \< \wt G_{k}(t) , R(t)\>=0,\\
& {\rm Im}\< \nabla U_{k}(t) , R(t)\>=0,\ \
{\rm Im}\< \Lambda_k U_{k}(t), R(t)\>=0,\ \
{\rm Im}\< \wt\rho_{k}(t), R(t)\>=0.
\ea\ee
Here, $\Lambda_k:=\half+(x-\a_k)\cdot \nabla$ and
\ben
\widetilde{S}_k(t,x)=\lbb_k^{-\frac 12} S_1(t,\frac{x-\a_k}{\lbb_k}) e^{i\g_k}, \;
\widetilde{G}_k(t,x)=\lbb_k^{-\frac 12} G_1(t,\frac{x-\a_k}{\lbb_k}) e^{i\g_k}, \;
\widetilde{\rho}_k(t,x)=\lbb_k^{-\frac 12} \rho_k(t,\frac{x-\a_k}{\lbb_k}) e^{i\g_k}
\enn
where $\rho_k=\rho+i\varrho_k$ with $\rho$ being defined as in \eqref{eq-rho1} and $\varrho_k$ solving
\ben
L_-\varrho_k=2b_kS_1\rho Q+b_k\Lambda\rho-2b_kS_2+2v_kG_1\rho Q+v_k\cdot \nabla\rho+v_kG_2.
\enn
\end{proposition}

\begin{remark}
The geometrical decomposition in Proposition \ref{Prop-dec-un}
is actually a local version,  since the backwards time $T_*$ may depend on $T$.
Nevertheless,
as we shall see later,
by virtue of the bootstrap estimates in Theorem \ref{Thm-u-Boot} below,
for $T$ sufficiently close to $0$,
the geometrical decomposition remains valid
on a time interval $[t_0,T]$,
where $t_0$ is a universal backward time independent of $T$.

The proof of Proposition \ref{Prop-dec-un} is based on a standard fixed point argument.
We refer the reader to \cite{K-L-R}
for the proof of Proposition \ref{Prop-dec-un} in the single-bubble case.
The argument can be extended to the multi-bubble case,
see, e.g., \cite{S-Z}  for the case of NLS.
Thus,  for the sake of simplicity,
we omit the details of the proof here.
\end{remark}

\subsection{Modulation equations}
For every $1\leq k\leq K$,
define the \emph{vector of modulation equations} by
\begin{align*}
   Mod_{k}:= |\dot{\lambda}_{k}+b_{k}|+|\lambda_{k}\dot{b}_{k}+\half b_{k}^2|
   +|\dot{\alpha}_{k}-v_{k}|
                +|\lambda_{k}\dot{v}_{k}+b_{k}v_{k}|
              +|\lambda_{k}\dot{\g}_{k}-1|,
\end{align*}
where we use the notation $\dot{g}:= \frac{d}{dt}g$  for any $C^1$ function $g$.
Let $Mod:=\sum_{k=1}^{K}Mod_{k}$.

In the following proposition, we derive the preliminary estimate for these modulation equations.

\begin{proposition} [Control of modulation equations] \label{Prop-Mod-bdd}
Let $u$ be a solution to \eqref{equa} and
assume that the geometrical decomposition \eqref{u-dec} holds on $[T_*, T]$ with $T$ close to 0.
Assume the following uniform smallness for the parameters and the remainder
on $[T_*, T]$:
\ben
\lbb(t)+b(t)+v(t)+\sum_{k=1}^{K}|\a_k(t)-x_k|+\|R(t)\|_{H^\half}\ll1.
\enn
Then, there exists $C>0$ such that for any $ t\in[T_*, T]$,
\begin{align} \label{Mod-bdd}
Mod(t)\leq
C(\sum_{k=1}^{K}|{\rm Re}\<U_k, R\>|+(\lbb+b^2+v)\|R\|_{L^2}+\|R\|_{L^2}^2+\lbb\|R\|_{H^\half}^3+\lbb^2+b^4+v^2).
\end{align}
\end{proposition}

\begin{remark}
In contrast to the five unstable directions in \eqref{ortho-cond-Rn-wn},
the scalar ${\rm Re}\<U_k, R\>$ on the right-hand side of \eqref{Mod-bdd}
corresponds to the extra unstable direction,
which is not controlled by the geometrical decomposition,
but by the localized mass in Lemma \ref{lem-mass-local} later.
It should be mentioned that, unlike in the single-bubble case in \cite{K-L-R},
the scalar ${\rm Re}\<U_k, R\>$ turns out to dominate the upper bound for the modulation equations,
which also reflects the feature in the multi-bubble case.
\end{remark}

\begin{proof}[Proof of Proposition \ref{Prop-Mod-bdd}]
For the reader's convenience, we present the estimate for the modulation equations $\dot{\lambda}_{j}+b_{j}$ for $1\leq j\leq K$,
as an example to illustrate the main arguments.

By inserting the decomposition \eqref{u-dec} into \eqref{equa}
we obtain the equation for the remainder $R$
\be\label{eq-U-R}
iR_t-DR+2|U_j|^2R+U_j^2\overline{R}+(i\partial_tU_j-DU_j+|U_j|^2U_j)=\sum_{l=1}^{4}H_l,
\ee
where
\ben\ba
&H_1=-\sum_{k\neq j}(i\partial_tU_k-DU_k+|U_k|^2U_k),\quad H_2=-(|U|^2U-\sum_{k=1}^{K}|U_k|^2U_k),\\
&H_3=-(2|U|^2R+U^2\ol{R}-2|U_j|^2R-U_j^2\overline{R}),\quad H_4=-(\ol{U}R^2+2U|R|^2+|R|^2R).
\ea\enn
Moreover, by \eqref{Ground} and \eqref{U-Uj},  for $1\leq k\leq K$, we have the algebraic identity
\be\ba\label{exp-eta}
i\partial_tU_k-DU_k+|U_k|^2U_k=e^{i\g_{k}}\lambda_{k}^{-\frac 32}
&\left(i(\lambda_{k}\dot{v}_{k}+b_kv_k)\partial_{v_k}Q_{k}
  + i(\lambda_{k}\dot{b}_{k}+\half b_k^2) \partial_{b_k}Q_{k}
  -i(\dot{\alpha}_{k}-v_k)Q^\prime_{k}\right.\\
     & \left.-i(\dot{\lambda}_{k}+b_k)\Lambda Q_{k} -(\lambda_{k}\dot{\g}_{k}-1)Q_{k}-\Psi_k\right)(t,\frac{x-\alpha_{k}}{\lambda_{k}}).
\ea\ee
Taking the inner product of \eqref{eq-U-R} with
$\Lambda_j  {U_{j}}$ and then taking the real part
we get
\be\label{eq-U-R2}
-{\rm Im}\<R_t, \Lambda_j  U_{j}\>+{\rm Re}\<-DR+2|U_j|^2R+U_j^2\overline{R}, \Lambda_j  U_{j}\>=\sum_{l=1}^{4}{\rm Re}\<H_l, \Lambda_j  U_{j}\>-{\rm Re}\<i\partial_tU_j-DU_j+|U_j|^2U_j, \Lambda_j  U_{j}\>.
\ee

We first consider the left-hand side of \eqref{eq-U-R2}.
Using the orthogonality conditions \eqref{ortho-cond-Rn-wn},
the identity \eqref{exp-eta} and the renormalized variable
\be\label{ren}
R(t,x)= \lbb_j^{-\frac 12} \epsilon_j(t,\frac{x-\a_j}{\lbb_j}) e^{i\g_j},\quad with \;\epsilon_j=\epsilon_{j,1}+i\epsilon_{j,2},
\ee
 we get
\be\ba\label{mod1}
&-{\rm Im}\<R_t, \Lambda_j  {U_{j}}\>={\rm Im}\<R, \partial_t\Lambda_j  {U_{j}}\>
=-{\rm Im}\<\Lambda R, \partial_t  {U_{j}}\>\\
&=-\lbb_j^{-1}{\rm Re}\<\Lambda \epsilon_j, -Q_j-\frac{i}{2}b_j^2\partial_{b_j}Q_j-ib_jv_j\partial_{v_j}Q_j+ib_j\Lambda Q_j-iv_j Q^\prime_j\>+\calo(\lbb^{-j}Mod\|R\|_{L^2})\\
&=\lbb_j^{-1}\(-{\rm Re}\< \epsilon_j, \Lambda Q_j\>+b_j{\rm Im}\< \epsilon_j, \Lambda^2 Q_j\>+\calo((Mod+b_j^2+v_j)\|R\|_{L^2})\).
\ea\ee
Using again the renormalization \eqref{ren} and the expansion of $Q_j$ in \eqref{Q-asy}
we derive
\be\ba\label{mod2}
&{\rm Re}\<-DR+2|U_j|^2R+U_j^2\overline{R}, \Lambda_j  U_{j}\> \\
=&\lbb_j^{-1}{\rm Re}\<-D\epsilon_j+2|Q_j|^2\epsilon_j+Q_j^2\overline{\epsilon}_j , \Lambda  Q_{j}\>\\
=&\lbb_j^{-1}{\rm Re}\<-D\epsilon_j+2|Q|^2\epsilon_j+2ib_jS_1Q\overline{\epsilon}_j
+Q^2\overline{\epsilon}_j , \Lambda  Q+ib_j\Lambda S_1\>+\calo(\lbb_j^{-1}(b_j^2+v_j)\|R\|_{L^2}).
\ea\ee
Combining \eqref{mod1} and \eqref{mod2} we then get
\be\ba\label{mod3}
&-{\rm Im}\<R_t, \Lambda_j  U_{j}\>+{\rm Re}\<-DR+2|U_j|^2R+U_j^2\overline{R}, \Lambda_j  U_{j}\>\\
=&\lbb_j^{-1}\(-\<L_+\epsilon_{j,1}, \Lambda Q\>-b_j\<L_-\epsilon_{j,2}, \Lambda S_1\>
+2b_j\<\epsilon_{j,2}, Q\Lambda QS_1\>+b_j\<\epsilon_{j,2}, \Lambda^2 Q\>\)\\
&+\calo(\lbb_j^{-1}(Mod+b_j^2+v_j)\|R\|_{L^2}).
\ea\ee
By the commutator formula $[D, \Lambda]=D$ and  $L_-S_1=\Lambda Q$,  the following  identity holds
\ben
L_-\Lambda S_1=-S_1+2Q\Lambda Q S_1+\Lambda Q+\Lambda^2Q.
\enn
This fact along with \eqref{mod3} and the identity $L_+\Lambda Q=-Q$ yields that
\be\ba\label{mod4}
{\rm L.H.S.\ of\ \eqref{eq-U-R2}}
&=\lbb_j^{-1}\(\<\epsilon_{j,1}, Q\>+b_j\<\epsilon_{j,2}, S_1\>-b_j\<\epsilon_{j,2}, \Lambda Q\>\)+\calo(\lbb_j^{-1}(Mod+b_j^2+v_j)\|R\|_{L^2})\\
&=\lbb_j^{-1}({\rm Re}\<\epsilon_j, Q_j\>)+\calo(\lbb_j^{-1}(Mod+b_j^2+v_j)\|R\|_{L^2})\\
&=\lbb_j^{-1}({\rm Re}\<U_j, R\>)+\calo(\lbb_j^{-1}(Mod+b_j^2+v_j)\|R\|_{L^2}),
\ea\ee
where we also used the fact that $b_j\<\epsilon_{j,2}, \Lambda Q\>={\rm Im}\<\epsilon_j, \Lambda Q_j\>+b_j^2\<\epsilon_{j,1}, \Lambda S_1\>=\calo(b_j^2\|R\|_{L^2})$, which follows from the orthogonality conditions \eqref{ortho-cond-Rn-wn} that ${\rm Im}\<\epsilon_j, \Lambda Q_j\>={\rm Im}\<R, \Lambda_j U_j\>=0$.

We then consider the right-hand side of \eqref{eq-U-R2}.
The first and second terms can be  bounded by using Lemma \ref{lem-dep}
that for some constant $C>0$ such that
\begin{align}
&|{\rm Re}\<H_1, \Lambda_j  U_{j}\>|
\leq C\lbb Mod,\nonumber\\
&|{\rm Re}\<H_2, \Lambda_j  U_{j}\>|\leq C\lbb.\label{mod5}
\end{align}
Similarly, the third term which is linear in $R$ can also be  bounded by Lemma \ref{lem-dep}:
\be\label{mod7}
|{\rm Re}\<H_3, \Lambda_j  U_{j}\>|\leq C\lbb\|R\|_{L^2}.
\ee
Moreover, the forth term in which is nonlinear in $R$ can be  bounded by Lemma \ref{Lem-GN}:
\be\label{mod8}
|{\rm Re}\<H_4, \Lambda_j  U_{j}\>|\leq C(\lbb^{-1}\|\epsilon_j\|_{L^2}^2+\lbb^{-1}\|\epsilon_j\|_{L^6}^3)
\leq C(\lbb^{-1}\|R\|_{L^2}^2+\|R\|_{H^\half}^3).
\ee

Concerning the last term on the right-hand side of \eqref{eq-U-R2},
we recall that $L_->0$ on $Q^\bot$,
and so $e_1:=\<S_1, \Lambda Q\>=\<S_1, L_-S_1\> >0$.
This fact along with the identity \eqref{exp-eta}  and \eqref{est-psi} yields that
\be\label{mod9}
-{\rm Re}\<i\partial_tU_j-DU_j+|U_j|^2U_j, \Lambda_j  U_{j}\>
=\lbb_j^{-1}\(e_1(\lambda_{j}\dot{b}_{j}+\half b_j^2)+\calo(b_jMod+b_j^4+v_j^2)\).
\ee

Inserting \eqref{mod4}-\eqref{mod9} into \eqref{eq-U-R2} we get
\ben
|\lambda_{j}\dot{b}_{j}+\half b_j^2|\leq C(|{\rm Re}\<U_j, R\>|+(\lbb+b_j^2+v_j)\|R\|_{L^2}+\|R\|_{L^2}^2+\lbb\|R\|_{H^\half}^3+b_jMod+\lbb^2+b_j^4+v_j^2),
\enn
Thus, by summing over $j$ from $1$ to $K$ we arrive at
\ben
\sum_{j=1}^{K}|\lambda_{j}\dot{b}_{j}+\half b_j^2|\leq C(\sum_{j=1}^{K}|{\rm Re}\<U_j, R\>|+(\lbb+b^2+v)\|R\|_{L^2}+\|R\|_{L^2}^2+\lbb\|R\|_{H^\half}^3+bMod+\lbb^2+b^4+v^2).
\enn

Taking the inner products of equation (\ref{eq-U-R})
with $i\wt{S}_j$, $i\wt{G}_j$, $\nabla {U_j}$, $ \wt{\rho}_j$,
respectively, and then taking the real parts,
the remaining four modulation equations can be estimated similarly.
Summing these estimates together yields \eqref{Mod-bdd},
thereby finishing the proof of Proposition \ref{Prop-Mod-bdd}.
\end{proof}

We end this section with the localization functions for the bubbling construction.

\subsection{Localization functions}  \label{Subsec-Local}
We introduce the localization functions which will be frequently used in the construction.
Let $\{x_k\}_{k=1}^K$ be the $K$ distinct points in Theorem \ref{thm-main} and set
\be\label{def-sig}
\sigma :=\frac{1}{12}\min_{1\leq k \leq K-1}\{(x_{k+1}-x_k)\}> 0.
\ee
Let $\Phi: \R\rightarrow [0,1]$ be a smooth function  such that
$|\Phi^\prime(x)| \leq C \sigma^{-1}$ for some $C>0$,
$\Phi(x)=1$ for $x\leq 4\sigma$
and $\Phi(x)=0$ for $x\geq 8\sigma$.
The localization functions $\Phi_k$, $1\leq k\leq K$, are defined by
\ben
&\Phi_1(x) :=\Phi(x-x_1), \ \ \Phi_K(x) :=1-\Phi(x-x_{K-1}),  \\
&\Phi_k(x) :=\Phi(x-x_{k})-\Phi(x-x_{k-1}),\ \ 2\leq k\leq K-1.
\enn
One has the  partition of unity, that is, $1\equiv\sum_{k=1}^K \Phi_k$.

\section{Bootstrap estimates} \label{Sec-Boot}
The objective of this section is to establish the bootstrap estimates
for the remainder  and parameters in the geometrical decomposition,
which play the crucial role in proving the uniform estimates and completing the construction of  multi-bubble blow-up solutions in  Section 7.

\begin{theorem}[Bootstrap estimates] \label{Thm-u-Boot}
 Let $\delta, \varsigma\in(0, \frac12)$ and $\delta+2\varsigma<1$.
 There exists a uniform backwards time $t_0<0$ such that the following holds.
 Let $u$ be a solution to \eqref{equa} and admit the geometrical decomposition on $[T_*, T]$ as in Proposition \ref{Prop-dec-un}
 with $T_*>t_0$.
 Assume that for all $t\in [T_*, T]$, $1\leq k\leq K$,   the following bounds for the remainder and parameters hold:
\begin{align}
&\|D^\half R(t)\|_{L^2}\leq |t|^{2-\delta},\quad\|R(t)\|_{L^2}\leq |t|^{3-\delta}, \quad \|D^{\half+\varsigma}R\|_{L^2}\leq |t|^{1-\delta-2\varsigma},\label{R-Tt}\\
&|\la_{k}(t) - \frac{\omega^2}{4} t^2 |\leq |t|^{4-2\delta},\quad |b_{k}(t)  + \frac{\omega^2}{2} t |\leq |t|^{3-2\delta},   \label{lbbn-Tt} \\
&|\al_{k}(t)-x_k|\leq |t|^{3-\delta},\quad |v_k(t)-\frac{\omega^2}{4} t^2|\leq |t|^{4-2\delta},  \label{vn-Tt}\\
&|\g_{k}(t) + \frac{4}{\omega^2 t} - \theta_k| \leq |t|^{1-2\delta}.  \label{thetan-Tt}
\end{align}
Then,  these bounds can be improved such that, for all $t\in [T_*, T]$ and $1\leq k\leq K$,
\begin{align}
&\|D^\half R(t)\|_{L^2}\leq \frac{1}{2}|t|^{2-\delta},\quad\|R(t)\|_{L^2}\leq \frac{1}{2}|t|^{3-\delta}, \quad \|D^{\half+\varsigma}R\|_{L^2}\leq \half |t|^{1-\delta-2\varsigma}.  \label{wn-Tt-boot-2} \\
&|\la_{k}(t) - \frac{\omega^2}{4} t^2 |\leq \half|t|^{4-2\delta},\quad   |b_{k}(t)  + \frac{\omega^2}{2} t |\leq \half|t|^{3-2\delta},   \label{lbbn-Tt2} \\
&|\al_{k}(t)-x_k|\leq \half|t|^{3-\delta},\quad |v_k(t)-\frac{\omega^2}{4} t^2|\leq \half|t|^{4-2\delta},  \label{vn-Tt12}\\
&|\g_{k}(t) + \frac{4}{\omega^2 t} - \theta_k| \leq \half|t|^{1-2\delta}. \label{thetan-Tt12}
\end{align}
\end{theorem}

%In order to prove Theorem \ref{Thm-u-Boot}, in the context of this section below,
%we  take $T^*<T_*$, sufficiently close to $T_*$, so that from the  $H^s$ continuity
%of the flow, $u$ still has the geometrical decomposition  on the larger interval  $[T^*, T]$.

According to \eqref{equa},
we have the equation for
the remainder $R$
\begin{align} \label{equa-R}
   i\partial_t R -D R+|u|^2u-|U|^2U=-\eta,
\end{align}
where the error $\eta$ satisfies
\begin{align} \label{etan-Rn}
    \eta = i\partial_t U -D U+|U|^2U.
\end{align}
Define a quantity of the remainder
\begin{align} \label{X-def}
X(t):=\|D^\half R(t)\|_{L^2}^2+t^{-2}\|R(t)\|_{L^2}^2.
\end{align}

As a consequence of \eqref{R-Tt}-\eqref{thetan-Tt},
we  have the following a prior bounds for the modulation parameters,
the modulation equations and the error $\eta$, which will be frequently used in the proof of Theorem \ref{Thm-u-Boot}.

\begin{lemma}   \label{Lem-P-ve-U}
By the assumptions in Theorem \ref{Thm-u-Boot}, there exists a positive constant $C$ independent of $R$
such that  for all $t\in [T_*, T]$,  the following holds:\\
(i) For the modulation parameters, $1\leq k\leq K$,
\begin{align}\label{est-pa}
 C t^2\leq \lbb_k(t)\leq \frac{t^2}{C},\quad  C|t|\leq b_k(t)\leq \frac{|t|}{C},\quad  Ct^2\leq v_k(t) \leq \frac{t^2}{C}.
\end{align}
(ii) For the modulation equations
\begin{align} \label{Mod-w-lbb}
Mod(t) \leq C |t|^{4-\delta}.
\end{align}
(iii) For the error $\eta$,
\begin{align}\label{eta-L2}
\| \nabla^\nu\eta(t)\|_{L^2}
      \leq C|t|^{-2-2\nu}(Mod+t^4),\quad \nu=0, 1, 2.
\end{align}
\end{lemma}
\begin{proof}
It is easy to get \eqref{est-pa}.
Then,
by virtue of Proposition \ref{Prop-Mod-bdd}, \eqref{R-Tt}, \eqref{est-pa}
and Lemma \ref{lem-mass-local} in Section 4
we are able to get \eqref{Mod-w-lbb}.
As for \eqref{eta-L2},  we first expand $\eta$ by using \eqref{etan-Rn}
\ben\ba
    \eta = \sum_{k=1}^{K}(i\partial_t U_k -D U_k+|U_k|^2U_k)+(|U|^2U-\sum_{k=1}^{K}|U_k|^2U_k)
    =I+II.
\ea\enn
Then,  using the  identity \eqref{exp-eta}  we have
\ben
| \nabla^{\nu}I|\leq C\sum_{k=1}^{K}\lbb_k^{-\frac32-\nu}(Mod_k \<\frac{x-\a_k}{\lbb_k}\>^{-2}+|\Psi_k|),
\enn
which along with \eqref{est-psi} and \eqref{est-pa} yields that
\ben
\|\nabla^{\nu}I\|_{L^2}\leq C|t|^{-2-2\nu}(Mod+t^4).
\enn
Moreover, by Lemma \ref{lem-dep} and \eqref{est-pa},
\ben
\|\nabla^{\nu}II\|_{L^2}\leq C|t|^{-2-2\nu}t^4.
\enn
Combining the above estimates we get the  bound \eqref{eta-L2}.
\end{proof}

The proof of Theorem \ref{Thm-u-Boot} occupies Sections \ref{Sec-Local-Mass-Moment},
\ref{Sec-Gen-Energy} and \ref{Sec-Proof-Boot},
and relies on the estimates \eqref{est-pa}-\eqref{eta-L2}.
The strategy of proof proceeds mainly in three steps.
We first control the localized mass and momentum in Section \ref{Sec-Local-Mass-Moment}.
Then, in Section \ref{Sec-Gen-Energy} we construct a generalized energy functional
adapted to the multi-bubble case,
and prove the crucial properties of monotonicity and coercivity.
At last, in Section \ref{Sec-Proof-Boot},
we prove the bootstrap estimates in Theorem \ref{Thm-u-Boot}
by using the backwards integration from singularities.

\section{Localized mass and momentum} \label{Sec-Local-Mass-Moment}
In this section, we assume  the assumptions in Theorem \ref{Thm-u-Boot} to hold. Then we first study the \emph{localized mass} $\int |u |^2\Phi_kdx$,
in order to obtain  the estimates for the unstable directions ${\rm Re}\langle U_k, R\rangle$, $1\leq k\leq K$.
Note that the localized mass are no longer conserved in the multi-bubble case.
\begin{lemma}  [Estimate of ${\rm Re}\langle U_k, R\rangle$]    \label{lem-mass-local}
For all $t\in [T_*, T]$ and $1\leq k\leq K$, we have
\begin{align} \label{est-mass-local}
    2{\rm Re}\int U_{k}\ol Rdx+\int |R(t)|^2\Phi_kdx
  = \calo(|t|^{4-\delta}).
\end{align}
\end{lemma}
\begin{remark}
The unstable directions ${\rm Re}\langle U_k, R\rangle$, $1\leq k\leq K$,
arise in serval important estimates, such as the coercivity of the linearized operator \eqref{coer}
and the estimate for the modulation equations \eqref{Mod-bdd}.
However, unlike in the single-bubble case for \eqref{equa}
or the multi-bubble case for the NLS \eqref{NLS},
due to the non-conservation of the localized mass
and the strong coupling between different profiles,
the analysis towards the derivation of \eqref{est-mass-local} is more delicate.
This also reflects the feature of the bubbling phenomena in the nonlocal problem \eqref{equa}.
\end{remark}

{\bf Proof of Lemma \ref{lem-mass-local}.}
Without loss of generality, we prove \eqref{est-mass-local} for $k=1$.
Using \eqref{u-dec}, Lemmas \ref{lem-dep} and  \ref{Lem-P-ve-U}, we expand the localized mass
\begin{align*}
\int |u |^2\Phi_1dx
= \int |U |^2\Phi_1dx +\int |R |^2\Phi_1dx +2{\rm Re}\int U_1 \ol R dx
   + \calo(|t|^3\|R\|_{L^2}),
\end{align*}
which yields that
\begin{align} \label{uL2-phi}
&|2{\rm Re}\int (U_1 \ol R)(t) dx+\int |R(t)|^2\Phi_1 dx|\nonumber \\
 \leq&  |\int |U(T)|^2\Phi_1 dx-\int |U(t)|^2\Phi_1 dx|+|\int |u(t)|^2\Phi_1 dx-\int |u(T)|^2\Phi_1 dx|+\calo( |t|^{3} \|R\|_{L^2}).
\end{align}
where in the last step we also used the boundary condition $u(T)=U(T)$.

Using \eqref{Q-l2}, Lemmas \ref{lem-dep} and Lemma \ref{Lem-P-ve-U},
we have that for any $t\in [T^*, T]$,
\ben
 \int |U(t)|^2\Phi_1 dx=\int|\sum_{k=1}^{K}U_k(t)|^2\Phi_1 dx=\int |U_1(t)|^2dx+\calo(t^4)=\|Q\|^2_{L^2}+\calo(t^4),
\enn
which yields the estimate  for the first term in \eqref{uL2-phi} that, for some $C>0$,
\be\label{est-loc-1}
|\int |U(T)|^2\Phi_1 dx-\int |U(t)|^2\Phi_1 dx|\leq Ct^4.
\ee

The second term on the right-hand side of \eqref{uL2-phi}, which is the evolution of the localized mass,
can be estimated by using equation \eqref{equa}. In fact, we have
\begin{align} \label{du2-bc}
 &\frac{d}{dt}\int |u|^2\Phi_1dx
 = 2{\rm Im}\int \ol{u} D u\Phi_1 dx\nonumber\\
 &= 2{\rm Im}\int \ol{U} DU\Phi_1 dx+2{\rm Im}\int (\ol{U} DR+\ol{R} DU)\Phi_1 dx+2{\rm Im}\int \ol{R} DR\Phi_1 dx
 =:I+II+III.
 \end{align}
We shall estimate the above three terms separately.

Remark that, due to the nonlocal operator $D$,
the estimates below are more delicate than those in \cite{RSZ21,S-Z}.
Indeed, the integration representation of operator $D$ and the pointwise decay property
of the ground state are used
in order to get right decay orders.
Algebraic cancellations (see e.g. \eqref{Q1Qk-cancel} below)
and Calder\'{o}n's estimate in \cite{C65} are also used below.

{\it $(i)$ Estimate of $I$.}
The first term on the right-hand side of \eqref{du2-bc} can be decomposed as
\be\ba\label{dec-i}
I&=2{\rm Im}\int \ol{U}_1 DU_1\Phi_1 dx+2\sum_{k=2}^{K}{\rm Im}\int (\ol{U}_1 DU_k+\ol{U}_kDU_1)\Phi_1 dx+2\sum_{k,l=2}^{K}{\rm Im}\int \ol{U}_k DU_l\Phi_1 dx\\
&=:I_1+I_2+I_3.
\ea\ee
Since $D$ is self-adjoint, then integrating by parts and using the identity \eqref{id-fra2}
we obtain
\ben\ba
|I_1|&=|{\rm Im}\int U_1 (D(\ol{U}_1\Phi_1)-D\ol{U}_1\Phi_1) dx|\\
&=C|{\rm Im}\int U_1(x)dx\int\frac{\ol{U}_1(x+y)(\Phi_1(x+y)-\Phi_1(x))-\ol{U}_1(x-y)(\Phi_1(x)-\Phi_1(x-y))}
{|y|^2}dy|\\
&\leq C \sum\limits_{i=1}^4
   |{\rm Im} \iint_{\Omega_{1,i}}  U_1(x) \frac{\ol{U}_1(x+y)(\Phi_1(x+y)-\Phi_1(x))-\ol{U}_1(x-y)(\Phi_1(x)-\Phi_1(x-y))}{|y|^2} dxdy| \\
&=:C \sum\limits_{i=1}^4 I_{1,i},
\ea\enn
where $\bbr^2$ is partitioned into four regimes
$\Omega_{1,1}:=\{|x-x_1|\leq 3\sigma, |y|\leq \sigma\}$,
$\Omega_{1,2}:=\{|x-x_1|\geq 3\sigma, |y|\leq \sigma\}$,
$\Omega_{1,3}:=\{|x-x_1|\leq \frac{\sigma}{2}, |y| \geq  \sigma\}$
and $\Omega_{1,4}:=\{|x-x_1| \geq  \frac{\sigma}{2}, |y| \geq  \sigma\}$,
where $\sigma>0$ is defined as in \eqref{def-sig}.
Below we estimate each integral $I_{1,i}$, $1\leq i\leq 4$.

{\it $(i.1)$ Estimate of $I_1$.}
The definition of $\Phi_1$ implies that $I_{1,1}=0$.
Moreover, by Taylor's expansion, there exists some $|\t|\leq 1$ such that
for any differentiable function $f$,
\be\ba\label{Taylor}
&|f(x+y)(\Phi_1(x+y)-\Phi_1(x))-f(x-y)(\Phi_1(x)-\Phi_1(x-y))|\\
&\leq |y|^2(|\nabla f (x+\t y)| \|\nabla\Phi_1\|_{L^\9}+|f(x)|\|\nabla^2\Phi_1\|_{L^\9})
 \leq C|y|^2(|\nabla f(x+\t y)|+|f(x)|).
\ea\ee
Thus, we get
\ben\ba
I_{1,2}&\leq  C\iint_{\Omega_{1,2}}|U_1(x)|(|\nabla U_1(x+\t y)|+|U_1(x)|)dxdy\\
&\leq C\lbb_1^{-\half}\iint_{\Omega_{1,2}}|Q_1(\frac{x-\a_1}{\lbb_1})|\(\lbb_1^{-\frac{3}{2}}|\nabla Q_1(\frac{x+\t y-\a_1}{\lbb_1})|+\lbb_1^{-\half}|Q_1(\frac{x-\a_1}{\lbb_1})|\)dxdy.
\ea\enn
We may take $T_*$ close to 0 such that $|x_1 - \a_1(t)|\leq \sigma$.
Then, for $(x,y) \in \Omega_{1,2}$, we have
\ben
|x+\t y-\a_1(t)|\geq \frac{1}{3}|x-x_1|\quad and \quad |x-\a_1(t)|\geq \frac{2}{3}|x-x_1|.
\enn
So, applying the decay property \eqref{Q-decay} of $Q$ and \eqref{est-pa} we get
\ben
I_{1,2}\leq C \lbb_1^{2}\int_{|x-x_1|\geq 3\sigma}
|x-x_1|^{-4}dx\leq Ct^4.
\enn
We may take $T_*$ even closer to 0 such that $|\a_1(t)-x_1|\leq \frac{\sigma}{4}$.
Then, for $(x,y)\in \Omega_{1,3}$, we have $|x+ y-\a_1(t)|\geq \frac{\sigma}{4}$,
and by \eqref{Q-decay},
\ben\ba
I_{1,3}&\leq C\int_{|x-x_1|\leq \frac{\sigma}{2}} |U_1(x)|dx\int_{|y|\geq \sigma}\frac{|U_1(x+y)|}{|y|^2}dy\\
&\leq C\lbb_1^{-1}\int_{|x-x_1|\leq \frac{\sigma}{2}} |Q_1(\frac{x-\a_1}{\lbb_1})|dx\int_{|y|\geq \sigma}|Q_1(\frac{x+y-\a_1}{\lbb_1})||y|^{-2}dy\\
&\leq C\sigma^{-1}\lbb_1^{-1} \int_{|x-x_1|\leq \frac{\sigma}{2}} |Q_1(\frac{x-\a_1}{\lbb_1})|\(\frac{\sigma}{\lbb_1}\)^{-2}dx\leq Ct^4.
\ea\enn
Concerning the last term $I_{1.4}$,
for $|x-x_1|\geq \frac{\sigma}{2}$ we have $|x-\a_1(t)|\geq \frac{\sigma}{4}$.
Hence, it follows that
\ben\ba
I_{1,4}&\leq C\int_{|x-\a_1|\geq \frac{\sigma}{4}} |U_1(x)|dx\int_{|y|\geq \sigma}|y|^{-2}dy \leq C|t|^4.
\ea\enn
Thus,
we conclude that
\be\label{est-i1}
|I_1|\leq C|t|^3.
\ee

{\it $(i.2)$ Estimate of $I_2$.}
The second term $I_2$ in \eqref{dec-i} can be estimated by using the approximate equation \eqref{eq-q} of $Q_k$.
In fact, combining equation \eqref{eq-q},  Lemmas \ref{lem-app} and \ref{Lem-P-ve-U}
we find that
\be\label{id-qk}
DQ_k(y)=-Q_k(y)+|Q_k(y)|^2Q_k(y)+\calo(|t|\<y\>^{-2}), \quad 1\leq k\leq K.
\ee
Using the renormalization \eqref{U-Uj}
we get that for $1\leq k\leq K$,
\be\ba\label{est-lm-ii}
&{\rm Im}\int (\ol{U}_1 DU_k+\ol{U}_kDU_1)\Phi_1 dx\\
&=(\lbb_1\lbb_k)^{-\frac{1}{2}}{\rm Im}\int \lbb_k^{-1}\ol{Q}_1(\frac{x-\a_1}{\lbb_1}) DQ_k(\frac{x-\a_k}{\lbb_k})e^{i(\gamma_k-\gamma_1)}\Phi_1+\lbb_1^{-1} \ol{Q}_k(\frac{x-\a_k}{\lbb_k})DQ_1(\frac{x-\a_1}{\lbb_1})e^{i(\gamma_1-\gamma_k)}\Phi_1 dx\\
&=(\lbb_1\lbb_k)^{-\frac{1}{2}}\frac{4}{\omega t^2}{\rm Im}\int \ol{Q}_1(\frac{x-\a_1}{\lbb_1}) DQ_k(\frac{x-\a_k}{\lbb_k})e^{i(\gamma_k-\gamma_1)}\Phi_1+ \ol{Q}_k(\frac{x-\a_k}{\lbb_k})DQ_1(\frac{x-\a_1}{\lbb_1})e^{i(\gamma_1-\gamma_k)}\Phi_1 dx\\
&\quad+(\lbb_1\lbb_k)^{-\frac{1}{2}}(\lbb_k^{-1}-\frac{4}{\omega t^2}){\rm Im}\int \ol{Q}_1(\frac{x-\a_1}{\lbb_1}) DQ_k(\frac{x-\a_k}{\lbb_k})e^{i(\gamma_k-\gamma_1)}\Phi_1dx\\
&\quad+(\lbb_1\lbb_k)^{-\frac{1}{2}}(\lbb_1^{-1}-\frac{4}{\omega t^2}){\rm Im} \int\ol{Q}_k(\frac{x-\a_k}{\lbb_k})DQ_1(\frac{x-\a_1}{\lbb_1})e^{i(\gamma_1-\gamma_k)}\Phi_1 dx\\
&=:I_{2,1}+I_{2,2}+I_{2,3}
\ea\ee
The identity \eqref{id-qk} implies that
$$DQ_k(y)=-Q_k(y)+\calo(|t|\<y\>^{-2}+\<y\>^{-6}).$$
Inserting this into $I_{2,1}$ we obtain the algebraic cancellation of the leading order term
\begin{align} \label{Q1Qk-cancel}
{\rm Im}\int \ol{Q}_1(\frac{x-\a_1}{\lbb_1}) Q_k(\frac{x-\a_k}{\lbb_k})e^{i(\gamma_k-\gamma_1)}\Phi_1+ \ol{Q}_k(\frac{x-\a_k}{\lbb_k})Q_1(\frac{x-\a_1}{\lbb_1})e^{i(\gamma_1-\gamma_k)}\Phi_1 dx=0.
\end{align}
 Using this fact and similar calculations as above we get
\ben
|I_{2,1}|\leq C|t|^{3}.
\enn
To estimate $I_{2,2}$ and $I_{2,3}$, we shall apply the a prior bound \eqref{lbbn-Tt}.
Then it follows that
\ben
|I_{2,2}|+|I_{2,3}|\leq C|t|^{4-2\delta}.
\enn
Plugging these into \eqref{est-lm-ii} and summing over $k$ from 2 to $K$ we thus obtain
\be\label{est-i2}
|I_2|\leq C|t|^3.
\ee

{\it $(i.3)$ Estimate of $I_3$.}
Taking $T_*$ possibly closer to $0$
we get from \eqref{id-qk} and Lemma \ref{Lem-P-ve-U} that
\be\label{est-i3}
|I_3|\leq C\sum_{k,l=2}^{K}\lbb_k^{-\half}\lbb_l^{-\frac{3}{2}}\int_{|x-x_k|\geq 4\sigma, |x-x_l|\geq 4\sigma}|Q_k(\frac{x-\a_k}{\lbb_k})||DQ_l(\frac{x-\a_l}{\lbb_l})|dx
\leq Ct^4.
\ee
Thus, inserting estimates \eqref{est-i1}, \eqref{est-i2} and \eqref{est-i3} into \eqref{dec-i}
we arrive at
\be\label{est-i}
|I|\leq C|t|^3.
\ee

{\it $(ii)$ Estimate of $II$.}
Regarding the second term $II$ on the right-hand side of \eqref{du2-bc},
we first apply the  integration by parts formula to get
\begin{align*}
{\rm Im}\int (\ol{U} DR+\ol{R} DU)\Phi_1 dx=&\sum_{k=1}^{K}{\rm Im}\int (\ol{U}_k DR+\ol{R} DU_k)\Phi_1 dx \\
=&\sum_{k=1}^{K}{\rm Im}\int R(D(\ol{U}_k\Phi_1)-D\ol{U}_k\Phi_1) dx
=: \sum\limits_{k=1}^K II_k.
\end{align*}
Then by Lemma \ref{lem-formula}, for $1\leq k\leq K$,
\ben\ba
|II_k|&=C|{\rm Im}\int R(x)dx\int\frac{\ol{U}_k(x+y)(\Phi_1(x+y)-\Phi_1(x))-\ol{U}_k(x-y)(\Phi_1(x)-\Phi_1(x-y))}{|y|^2}dy|\\
&\leq C \sum\limits_{i=1}^3
     |{\rm Im}\iint_{\Omega_{2,i}}  R(x)  \frac{\ol{U}_k(x+y)(\Phi_1(x+y)-\Phi_1(x))-\ol{U}_k(x-y)(\Phi_1(x)-\Phi_1(x-y))}{|y|^2}  dxdy| \\
&=: C \sum\limits_{i=1}^3 II_{k,i},
\ea\enn
where
$\Omega_{2,1} := \{|x-x_1|\leq 3\sigma\bigcup |x-x_k|\leq 3\sigma,|y|\leq \sigma\}$,
$\Omega_{2,2} := \{|x-x_1|\geq 3\sigma\bigcap |x-x_k|\geq 3\sigma,|y|\leq \sigma\}$
and $\Omega_{2,3} := \{|y|\geq \sigma\}$.

It is easy to see that $|x\pm y-x_1|\leq4\sigma$ for $(x,y)\in \Omega_{2,1}$,
and thus $II_{k,1}=0$ due to the definition of $\Phi_1$.
Moreover, by \eqref{Taylor}, we get
\ben\ba
II_{k,2}&\leq \int_{|x-x_k|\geq 3\sigma}|R(x)|dx\int_{|y|\leq \sigma}|\nabla U_k(x+\t y)|+|U_k(x)|dy\\
&\leq\int_{|x-x_k|\geq 3\sigma}|R(x)|dx\int_{|y|\leq \sigma}\lbb_k^{-\frac{3}{2}}|\nabla Q_k(\frac{x+\t y-\a_k}{\lbb_k})|+\lbb_k^{-\half}|Q_k(\frac{x-\a_k}{\lbb_k})|dy.
\ea\enn
We take $T^*$ closer to $0$ such that $|x_k - \a_k(t)|\leq \sigma$
and infer that for $|x-x_k|\geq 3\sigma$ and $|y|\leq \sigma$,
\ben
|x+\t y-\a_k(t)|\geq \frac{1}{3}|x-x_k|\quad and\quad  |x-\a_k(t)|\geq \frac{2}{3}|x-x_k|.
\enn
Hence, applying \eqref{Q-decay} we get
\ben
II_{k,2}\leq C\sigma\lbb_k^{\half}\int_{|x-x_k|\geq3\sigma}|R(x)||x-x_k|^{-2}dx\leq Ct\|R\|_{L^2}.
\enn
The last term $II_{k,3}$ can be estimated by
\ben\ba
II_{k,3}\leq C\int_{|y|\geq \sigma}|y|^{-2}dy\int |R(x)|(|U_k(x+y)|+|U_k(x-y)|) dx
\leq C\sigma^{-1}\|R\|_{L^2}.
\ea\enn
Thus,
we conclude that
\be\label{est-ii}
|II|\leq C\|R\|_{L^2}.
\ee

{\it $(iii)$ Estimate of $III$.}
At last, we consider the third term on the right-hand side of \eqref{du2-bc}.
By the integration by parts formula,
\ben
III=2{\rm Im}\int \ol{R} DR\Phi_1 dx={\rm Im}\int  R(D(\ol{R}\Phi_1)-D\ol{R}\Phi_1) dx.
\enn
In view of the definition of $\Phi_1$,
we apply the Calder\'on estimate in Lemma \ref{lem-l2-pro2} to get
\be\label{est-comm}
\|D(\ol{R}\Phi_1)-D\ol{R}\Phi_1\|_{L^2}\leq C\|R\|_{L^{2}},
\ee
which yields
\be\label{est-iii}
|III|\leq C\|R\|_{L^2}^2.
\ee

Now, combining estimates \eqref{est-i}, \eqref{est-ii} and \eqref{est-iii} altogether we get
\ben
|\frac{d}{dt}\int |u|^2\Phi_1dx|\leq C(|t|^3+\|R\|_{L^2}+\|R\|_{L^2}^2).
\enn
Then, integrating from $t$ to $T$ and using the a prior bound \eqref{R-Tt},
the second part of the right-hand side of \eqref{uL2-phi} can be estimated by
\be\label{est-loc-2}
|\int |u(t)|^2\Phi_j dx-\int |u(T)|^2\Phi_j dx|\leq C|t|^{4-\delta}.
\ee
Finally,  inserting \eqref{est-loc-1} and \eqref{est-loc-2} into \eqref{uL2-phi}
we obtain  \eqref{est-mass-local} and finish the proof of Lemma \ref{lem-mass-local}.
\hfill $\square$

Next we obtain the refined estimate of the parameter $v_k$,
by analyzing the \emph{localized momentum}
defined by  ${\rm Im}\int \na u(t)\ol{u}(t)\Phi_kdx$.

\begin{lemma}[Refined estimate of $v_k$]  \label{lem-lcomom}
For all $t\in [T_*, T]$ and $1\leq k\leq K$, we have
\be \label{vk-lambdak}
|\frac{v_k(t)}{\lbb_k(t)}-1|=\calo(|t|^{2-\delta}).
\ee
\end{lemma}

\begin{proof}
Without loss of generality, we prove \eqref{vk-lambdak} for $k=1$.
By \eqref{u-dec}, we first expand
the localized momentum around the profile $U$ to get
\be\ba\label{est-v-1}
{\rm Im}\int \na u(t)\ol{u}(t)\Phi_1dx
=& {\rm Im}\int \na U(t)\ol{U}(t)\Phi_1dx+{\rm Im}\int \na R(t)\ol{R}(t)\Phi_1dx\\
&+{\rm Im}\int (\na U(t)\ol{R}(t)+\na R(t)\ol{U}(t))\Phi_1dx.
\ea\ee
By  Lemmas \ref{lem-dep} and \ref{Lem-P-ve-U}, we have that for $t\in [T_*, T]$
\be\label{est-v-2}
|{\rm Im}\int \na U(t)\ol{U}(t)\Phi_1dx-{\rm Im}\int \na U_1(t)\ol{U}_1(t)dx|
\leq C|t|^3,
\ee
and
\ben
|{\rm Im}\int (\na U(t)\ol{R}(t)+\na R(t)\ol{U}(t))\Phi_1dx-2{\rm Im}\int \na U_1(t)\ol{R}(t)d|
\leq C|t|\|R\|_{L^2},
\enn
which along with the orthogonality conditions \eqref{ortho-cond-Rn-wn} yields that
\be\label{est-v-3}
|{\rm Im}\int (\na U(t)\ol{R}(t)+\na R(t)\ol{U}(t))\Phi_1dx|
\leq C|t|\|R\|_{L^2}.
\ee
Moreover, it follows from Lemma \ref{lem-inter} that
\be\label{est-v-4}
|{\rm Im}\int \na R(t)\ol{R}(t)\Phi_1dx|\leq C\|R\|_{H^\half}^2.
\ee
Thus, inserting \eqref{est-v-2}-\eqref{est-v-4} into \eqref{est-v-1} we obtain
\be\label{est-v-5}
{\rm Im}\int \na u(t)\ol{u}(t)\Phi_1dx
-{\rm Im}\int \na U_1(t)\ol{U}_1(t)dx=
\calo(|t|\|R\|_{L^2}+\|R\|_{H^\half}^2+|t|^3).
\ee

Next, we treat the L.H.S. of \eqref{est-v-5}.
Note that, on one hand, we have
\be\ba\label{mom-1}
&|{\rm Im}\int \na U_1(t)\ol{U}_1(t)dx-{\rm Im}\int \na U_1(T)\ol{U}_1(T)dx|\\
&\leq |{\rm Im}\int \na u(t)\ol{u}(t)\Phi_1dx-{\rm Im}\int \na u(T)\ol{u}(T)\Phi_1dx|
+|{\rm Im}\int \na u(t)\ol{u}(t)\Phi_1dx-{\rm Im}\int \na U_1(t)\ol{U}_1(t)dx|\\
&\quad+|{\rm Im}\int \na u(T)\ol{u}(T)\Phi_1dx-{\rm Im}\int \na U_1(T)\ol{U}_1(T)dx|\\
&\leq |{\rm Im}\int \na u(t)\ol{u}(t)\Phi_1dx-{\rm Im}\int \na u(T)\ol{u}(T)\Phi_1dx|
+\calo(|t|\|R\|_{L^2}+\|R\|_{H^\half}^2+|t|^3).
\ea\ee
On the other hand, by \eqref{Q-asy}, Lemma \ref{Lem-P-ve-U}
and the fact that $\half\<S_1,S_1\>+\<Q,S_2\>=0$, we have
\be\label{mom-2}
{\rm Im}\int \na U_1(t)\ol{U}_1(t)dx
=\lbb_1^{-1}(t){\rm Im}\int \na Q_1(t)\ol{Q}_1(t)dx
=p_1\frac{v_1(t)}{\lbb_1(t)}+\calo(t^2)
\ee
where $p_1:=2\<L_-G_1,G_1\> >0$.
In particular, by \eqref{PjT}, it follows that
\be\label{mom-3}
{\rm Im}\int \na U_1(T)\ol{U}_1(T)dx=p_1\frac{v_1(T)}{\lbb_1(T)}+\calo(T^2)=p_1+\calo(t^2).
\ee

Thus, combining \eqref{mom-1}, \eqref{mom-2} and \eqref{mom-3} together
we arrive at
\be\label{est-v-6}
p_1|\frac{v_1(t)}{\lbb_1(t)}-1|\leq |{\rm Im}\int \na u(t)\ol{u}(t)\Phi_1dx-{\rm Im}\int \na u(T)\ol{u}(T)\Phi_1dx|+\calo(|t|\|R\|_{L^2}+\|R\|_{H^\half}^2+|t|^2).
\ee

The evolution of the localized momentum on the right-side of \eqref{est-v-6}
can be estimated by using equation \eqref{equa}.
Actually, integrating by parts we get
\be\ba\label{est-pv}
&\frac{d}{dt}{\rm Im}\int \na u(t)\ol{u}(t)\Phi_1dx
=-{\rm Re}\int \na \ol{u}(D(u\Phi_1)-Du\Phi_1)dx
-\half\int  |u|^4\na\Phi_1dx\\
&=-\sum_{k=1}^{K}{\rm Re}\int \na \ol{U}(D(U_k\Phi_1)-DU_k\Phi_1)dx-\sum_{k=1}^{K}{\rm Re}\int \na \ol{R}(D(U_k\Phi_1)-DU_k\Phi_1)dx\\
&\quad-\sum_{k=1}^{K}{\rm Re}\int \na \ol{U}_k(D(R\Phi_1)-DR\Phi_1)dx-{\rm Re}\int \na \ol{R}(D(R\Phi_1)-DR\Phi_1)dx-\half\int  |u|^4\na\Phi_1dx\\
&=: \sum\limits_{i=1}^5 J_i
\ea\ee
We estimate the above five terms separately below.

{\it $(i)$ Estimate of $J_1$.}
Let us first estimate $J_1$. For every $1\leq k\leq K$, we have
\be\ba\label{est-pv-1}
&{\rm Re}\int \na \ol{U}(D(U_k\Phi_1)-DU_k\Phi_1)dx\\
&={\rm Re}\int \na \ol{U}_k(D(U_k\Phi_1)-DU_k\Phi_1)dx
+\sum_{l\neq k}{\rm Re}\int \na \ol{U}_l(D(U_k\Phi_1)-DU_k\Phi_1)dx\\
&=:J_{11}+J_{12}.
\ea\ee
The first term $J_{11}$ can be estimated as in the proof of \eqref{est-i1} in Lemma \ref{lem-mass-local}.
In fact, by formula \eqref{id-fra2},  we get
\ben\ba
 J_{11}&=C|{\rm Re}\int \na U_k(x)dx\int\frac{\ol{U}_k(x+y)(\Phi_1(x+y)-\Phi_1(x))-\ol{U}_k(x-y)(\Phi_1(x)-\Phi_1(x-y))}
{|y|^2}dy|\\
& \leq C \sum\limits_{i=1}^4 |{\rm Re} \iint_{\wt \Omega_{1,i}}
    \na U_k(x) \frac{\ol{U}_k(x+y)(\Phi_1(x+y)-\Phi_1(x))-\ol{U}_k(x-y)(\Phi_1(x)-\Phi_1(x-y))}{|y|^2} dxdy |,
\ea\enn
where we set
$\wt \Omega_{1,1}:= \{|x-x_k|\leq 3\sigma,|y|\leq \sigma\}$,
$\wt \Omega_{1,2}:= \{|x-x_k|\geq 3\sigma,|y|\leq \sigma\}$,
$\wt \Omega_{1,3}:= \{|x-x_k|\leq \frac{\sigma}{2},|y|\geq \sigma\}$,
and $\wt \Omega_{1,4}:= \{|x-x_k|\geq \frac{\sigma}{2},|y|\geq \sigma\}$.
The above four integrals can be estimated by following
the arguments as in the proof of \eqref{est-i1} from line to line.
Note also that, compared with \eqref{est-i1}, the derivative contributes an additional  $|t|^{-2}$ fact.
Hence, we get that
\ben
|J_{11}|\leq C|t|.
\enn
The second term $J_{12}$ can be controlled as in the proof of \eqref{est-i2}.
An application of the integration by parts formula yields that, for $l\neq k$,
\ben
{\rm Re}\int \na \ol{U}_l(D(U_k\Phi_1)-DU_k\Phi_1)dx
=-{\rm Re}\int (D\ol{U}_l\na U_k +\na \ol{U}_lDU_k) \Phi_1dx
-{\rm Re}\int D \ol{U}_lU_k\na\Phi_1dx.
\enn
By \eqref{Q-decay}, we have $\na Q=\calo(\<y\>^{-3})$. Then inserting  the identity \eqref{id-qk} into the above formula,
we get
\be\ba\label{est-mom-re}
&{\rm Re}\int (D\ol{U}_l\na U_k +\na \ol{U}_lDU_k) \Phi_1dx\\
&=(\lbb_l\lbb_k)^{-\frac{3}{2}}e^{i(\gamma_k-\gamma_l)}{\rm Re}\int D\ol{Q}_l(\frac{x-\a_l}{\lbb_l}) \na Q_k(\frac{x-\a_k}{\lbb_k})\Phi_1+ \na \ol{Q}_l(\frac{x-\a_l}{\lbb_l})DQ_k(\frac{x-\a_k}{\lbb_k})\Phi_1 dx\\
&=(\lbb_l\lbb_k)^{-\frac{3}{2}}e^{i(\gamma_k-\gamma_l)}{\rm Re}\int \ol{Q}_l(\frac{x-\a_l}{\lbb_l}) \na Q_k(\frac{x-\a_k}{\lbb_k})\Phi_1+ \na \ol{Q}_l(\frac{x-\a_l}{\lbb_l})Q_k(\frac{x-\a_k}{\lbb_k})\Phi_1 dx\\
&\quad+\calo\(\lbb^{-3}\int |t|\<\frac{x-\a_l}{\lbb_l}\>^{-2} \<\frac{x-\a_k}{\lbb_k}\>^{-3}+\<\frac{x-\a_l}{\lbb_l}\>^{-6} \<\frac{x-\a_k}{\lbb_k}\>^{-3}dx\).
\ea\ee
Thanks to the integration by parts formula, we have
\ben\ba
&{\rm Re}\int \ol{Q}_l(\frac{x-\a_l}{\lbb_l}) \na Q_k(\frac{x-\a_k}{\lbb_k})\Phi_1+ \na \ol{Q}_l(\frac{x-\a_l}{\lbb_l})Q_k(\frac{x-\a_k}{\lbb_k})\Phi_1 dx\\
&=\frac{\lbb_l-\lbb_k}{\lbb_l}\int \na \ol{Q}_l(\frac{x-\a_l}{\lbb_l})Q_k(\frac{x-\a_k}{\lbb_k})\Phi_1 dx-\lbb_k{\rm Re}\int \ol{Q}_l(\frac{x-\a_l}{\lbb_l})  Q_k(\frac{x-\a_k}{\lbb_k})\na\Phi_1dx.
\ea\enn
By using a prior estimate \eqref{lbbn-Tt} and Lemma \ref{lem-dep}, we get
\ben
|{\rm Re}\int \ol{Q}_l(\frac{x-\a_l}{\lbb_l}) \na Q_k(\frac{x-\a_k}{\lbb_k})\Phi_1+ \na \ol{Q}_l(\frac{x-\a_l}{\lbb_l})Q_k(\frac{x-\a_k}{\lbb_k})\Phi_1 dx|\leq C(\lbb^{4-\delta}+\lbb^4).
\enn
Thus the first term in the last equality of \eqref{est-mom-re} can be bounded by $\calo(|t|^{2-2\delta})$.
And by Lemma \ref{lem-dep} the remainder in the last equality of \eqref{est-mom-re} can be bounded by $\calo(|t|)$. Thus by inserting these estimates into \eqref{est-mom-re},  we can  get
\ben
|{\rm Re}\int (D\ol{U}_l\na U_k +\na \ol{U}_lDU_k) \Phi_1dx|\leq C|t|.
\enn
Similarly, by using \eqref{id-qk} and Lemma \ref{lem-dep} again, we have
\ben\ba
|{\rm Re}\int D \ol{U}_lU_k\na\Phi_1dx|\leq \lbb_l^{-\frac32}\lbb_k^{-\frac{1}{2}}\int (1+|t|)\<\frac{x-\a_l}{\lbb_l}\>^{-2} \<\frac{x-\a_k}{\lbb_k}\>^{-2}+\<\frac{x-\a_l}{\lbb_l}\>^{-6} \<\frac{x-\a_k}{\lbb_k}\>^{-2}dx\leq C|t|^2.
\ea\enn
Hence, we can obtain
\ben
|J_{12}|\leq C|t|.
\enn
Thus, returning to \eqref{est-pv-1} and summing over $k$ from $1$ to $K$
we obtain the bound
\be\label{est-pv-2}
|J_1|\leq C|t|.
\ee

{\it $(ii)$ Estimate of $J_2+ J_3$.}
Integrating by parts and using the formula \eqref{id-fra2} again
we derive that
\ben\ba
&|{\rm Re}\int \na \ol{R}(D(U_k\Phi_1)-DU_k\Phi_1)dx+{\rm Re}\int \na \ol{U}_k(D(R\Phi_1)-DR\Phi_1)dx|\\
=&|2{\rm Re}\int  \ol{R}(D(\na U_k\Phi_1)-D\na U_k\Phi_1)dx+{\rm Re}\int  \ol{R}(D(U_k\na\Phi_1)-DU_k\na\Phi_1)dx|.
\ea\enn
The above integrations can be estimated by arguing as in the proof of \eqref{est-ii}.
Similarly, compared with \eqref{est-ii},
an additional  $|t|^{-2}$ is lost here.
Thus, we get
\be\label{est-pv-3}
   |J_2+J_3|\leq \sum_{k=1}^{K}|{\rm Re}\int \na \ol{R}(D(U_k\Phi_1)-DU_k\Phi_1)dx
   +{\rm Re}\int \na \ol{U}_k(D(R\Phi_1)-DR\Phi_1)dx|
   \leq C\frac{\|R\|_{L^2}}{t^2}.
\ee

{\it $(iii)$ Estimate of $J_4$.}
The fourth term $J_4$ can be estimated by applying \eqref{est-comm} and standard interpolation theory. In fact, letting $Tg:=\nabla ((D(g\Phi_1)-Dg\Phi_1))$
and using the integration by parts formula we have
\ben
\ol{\<f, Tg\>}=-\int \na \overline{f}(D(g\Phi_1)-Dg\Phi_1)dx.
\enn
 On one hand, from \eqref{est-comm}
 (or the Calder\'on estimate in Lemma \ref{lem-l2-pro2}), it follows that
\ben
|\<f, Tg\>|\leq C\|f\|_{H^1}\|g\|_{L^2}.
\enn
On the other hand, integrating by parts and using the commutator formula $[D, \na]=0$,
we get
\ben
\overline{\<f, Tg\>}=\int  \ol{f}(D(\na g\Phi_1+g\na\Phi_1)-D\na g\Phi_1-Dg\na\Phi_1)dx.
\enn
Using Lemma \ref{lem-l2-pro2} again,
we estimate
\ben\ba
|\<f, Tg\>|&\leq |\int  \ol{f}(D(\na g\Phi_1)-D\na g\Phi_1)dx|+|\int  \ol{f}(D(g\na\Phi_1)-Dg\na\Phi_1)dx|\\
&\leq C(\|f\|_{L^2}\|\na g\|_{L^{2}}+\|f\|_{L^2}\| g\|_{L^2})\\
&\leq C\|f\|_{L^2}\|g\|_{H^{1}}.
\ea\enn
Hence, it follows from standard interpolation arguments (see, e.g., \cite[Lemma 23.1]{Tar}) that
\ben
|\<f, Tg\>|\leq C\|f\|_{H^\half}\|g\|_{H^{\half}}.
\enn
Taking $f=g=R$, it follows that $J_4={\rm Re}\<f,Tg\>$, thus we  obtain
\be\label{est-pv-4}
|J_4| = |{\rm Re}\<R,TR\>|\leq C\|R\|^2_{H^\half}.
\ee

{\it $(iv)$ Estimate of $J_5$.}
Using  \eqref{Q-decay} and the Sobolev embeddings we get
\be\ba\label{est-pv-5}
|J_5|\leq C(\sum_{k=1}^{K}\int  |U_k|^4\na\Phi_1dx+  \|R\|_{L^4}^4)\leq C(|t|^{12}+  \|R\|_{H^\half}^4).
\ea
\ee

Thus, inserting estimates \eqref{est-pv-2}, \eqref{est-pv-3}, \eqref{est-pv-4} and \eqref{est-pv-5} into \eqref{est-pv}
we get
\ben
|\frac{d}{dt}{\rm Im}\int \na u(t)\ol{u}(t)\Phi_1dx|\leq  C(\frac{\|R(t)\|_{L^2}}{|t|^2}+\|R(t)\|_{H^\half}+|t|),
\enn
which along with the a prior estimate \eqref{R-Tt} implies that
\ben\ba
|{\rm Im}\int \na u(t)\ol{u}(t)\Phi_1dx-{\rm Im}\int \na u(T)\ol{u}(T)\Phi_1dx|
&\leq \int_{t}^{T}|\frac{d}{ds}{\rm Im}\int \na u(s)\ol{u}(s)\Phi_1dx|ds\\
&\leq C\int\frac{\|R(s)\|_{L^2}}{|s|^2}+\|R(s)\|_{H^\half}+|s|ds\\
&\leq C|t|^{2-\delta}.
\ea\enn

Therefore,
plugging this into \eqref{est-v-6} we conclude that
\ben
p_1|\frac{v_1(t)}{\lbb_1(t)}-1|\leq C|t|^{2-\delta},
\enn
which yields the refined estimate \eqref{vk-lambdak} and finishes the proof.
\end{proof}

\section{Generalized energy functional: monotonicity and coercivity}  \label{Sec-Gen-Energy}

This section is devoted to study the important generalized energy \eqref{def-I},
which incorporates an energy part and a localized virial part.
The corresponding monotonicity and coercivity properties are the two key ingredients
to derive the bootstrap estimates  of the remainder \eqref{wn-Tt-boot-2},
which enables to integrate the flow backwards
from the singularity.

It should be mentioned that,
as in the context of NLS \cite{S-Z}, the generalized energy
also incorporates the localization functions $\{\Phi_k\}$
in an appropriate way,
which is different from  the single-bubble case in \cite{K-L-R}.

To be precise, let $\chi(x): \R\rightarrow\R$ be a smooth even function,
satisfying that
$\chi^\prime(x) = x$ if $0\leq x\leq 1$,
$\chi^\prime(x) = 3- e^{-x}$ if $x\geq2$,
and the convexity condition
$\chi^{\prime\prime}(x)\geq0$ for $x\geq0$.
Let $\chi_A(x) :=A^2\chi(\frac{x}{A})$ for $A>0$, and
$f(u):= |u|^{2} u$,
 $F(u):= \frac{1}{4} |u|^{4}$.

We define the generalized energy by
\begin{align} \label{def-I}
\mathfrak{I}(R) := &\frac{1}{2}\int |D^\half R|^2+\sum_{k=1}^K\frac{1}{\lambda_{k}} |R|^2 \Phi_kdx
           -{\rm Re}\int F(u)-F(U)-f(U)\ol{R}dx \nonumber \\
&+\sum_{k=1}^K\frac{b_{k}}{2}{\rm Im} \int (\nabla\chi_A) (\frac{x-\alpha_{k}}{\lambda_{k}})\cdot\nabla R \ol{R}\Phi_kdx.
\end{align}

\subsection{Monotonicity of the generalized energy}  \label{Subsec-Mon-Gen-Energy}
The crucial monotonicity property of the generalized energy
is formulated in Proposition \ref{prop-I-mono} below.

\begin{proposition} [Monotonicity of the generalized energy] \label{prop-I-mono}
Assuming the assumptions in Theorem \ref{Thm-u-Boot} to hold. Then for all $t\in [T_*, T]$, there exist positive constants $C$ and $C(A)$ such that, for $A$ large enough,
\ben\ba
  \frac{d\mathfrak{I}}{dt}
\geq C\frac{\|R(t)\|_{L^2}^2}{|t|^3} +\calo\(C(A)\(\ln(2+\|R\|_{H^\half}^{-1})\)^{\half}X(t)+|t|^{3-\delta}\),
  \ea\enn
where $X(t)$ is the quantity of the remainder given by \eqref{X-def}.
\end{proposition}

In order to prove Proposition \ref{prop-I-mono},
we decompose $\mathfrak{I}= \mathfrak{E}+ \mathfrak{L}$,
where $\mathfrak{E}$ denotes the energy part
\ben\ba
\mathfrak{E}(R):=\frac{1}{2} \int |D^\half R|^2+\sum_{k=1}^K\frac{1}{\lambda_{k}} |R|^2 \Phi_k dx
           -{\rm Re}\int F(u)-F(U)-f(U)\ol{R}dx,
\ea\enn
and $\mathfrak{L}$ denotes the localized virial part
\ben\ba      \label{I2}
\mathfrak{L}(R):=\sum_{k=1}^K\frac{b_{k}}{2}{\rm Im}
            \int (\nabla\chi_A) (\frac{x-\alpha_{k}}{\lambda_k})\cdot\nabla R\ol{R}\Phi_kdx.
\ea\enn

Propositions \ref{prop-I1t} and \ref{prop-I2t} below
contain the main estimates of $\mathfrak{E}$ and  $\mathfrak{L}$, respectively.

\begin{proposition} [Control of the energy part]  \label{prop-I1t}
For all $t\in [T_*, T]$, there exists some constant $C>0$ such that
\be\ba  \label{dt-E}
  \frac{d\mathfrak{E}}{dt}
\geq &\sum_{k=1}^{K}\left(\frac{b_k}{2\lambda_k^2} \|R_k\|_{L^2}^2
     -\frac{b_k}{2\lbb_{k}}{\rm Re}\int 2|U_k|^2|R_k|^2+\ol{U}_k^2R_k^2dx\right.\\
&\quad\quad\left.-b_k{\rm Re} \int \frac{x-\a_k}{\lbb_k}\cdot \nabla \ol{U}_k(\ol{U}_kR_k^2+2U_k|R_k|^2)dx\right)
     -C\(|t|^{3-\delta}+\frac{\|R\|_{L^2}^2}{|t|^{3-\delta}}+X(t)\).
\ea\ee
\end{proposition}

\begin{proof}
In view of equation \eqref{equa-R}, we  derive
\begin{align} \label{equa-I1t}
 \frac{d \mathfrak{E}}{dt}
=& -\sum_{k=1}^K \frac{\dot{\lbb}_{k}}{2\lbb_{k}^{2}}\int|R|^2\Phi_kdx
- \sum_{k=1}^K \frac{1}{\lbb_{k}}{\rm Im} \< R_k, D R\>
  -\sum_{k=1}^K \frac{1}{\lbb_{k}}{\rm Im} \<2|U|^2R+U^2\ol{R}, R_k\>  \nonumber\\
  &  -{\rm Re} \< 2U|R|^2+\ol{U}R^2+|R|^2R,  \partial_t {U} \>
     -\sum_{k=1}^K \frac{1}{\lbb_{k}} {\rm Im} \< 2U|R|^2+\ol{U}R^2+|R|^2R, R_k\>\nonumber\\
&+{\rm Im} \<D R +\sum_{k=1}^K  \frac{1}{\lbb_k}R_k -(|u|^2u-|U|^2U), \eta\>
=:  \sum\limits_{j=1}^6 \mathfrak{E}_{j}.
\end{align}
Below we estimate $\mathfrak{E}_{j}$, $1\leq j\leq 6$.

{\it $(i)$ Estimate of $\mathfrak{E}_{1}$.}
Using Lemma \ref{Lem-P-ve-U} we get
$|\frac{\dot{\lbb}_k +b_k}{\lbb_k^2}| \leq C \frac{Mod}{\lbb_k^2} \leq C |t|^{-\delta}$,
and thus
\begin{align} \label{I1t1-esti}
   \mathfrak{E}_{1}
  =\sum_{k=1}^{K} (\frac{b_k}{2\lbb_k^2}\int |R|^2\Phi_kdx-\frac{ \dot{\lbb}_k +b_k}{2\lbb_k^2}\int |R|^2\Phi_kdx)
  \geq\sum_{j=1}^{K} \frac{b_k}{2\lbb_k^2} \|R_k\|_{L^2}^2
        - C |t|^{2-\delta}X(t).
\end{align}

{\it $(ii)$ Estimate of $\mathfrak{E}_{2}$.}
Using the fact that $\sum_{k=1}^{K} R_k=R$ and ${\rm Im} \<R, DR\>=0$
we compute
\begin{align*}
  |\mathfrak{E}_{2}|
  =|\sum\limits_{k=1}^K
     (\frac{1}{\lbb_k} -\frac{1}{\omega t^2})
     {\rm Im} \<R_k, DR\>|
        \leq  \sum\limits_{k=1}^K
        \frac{|\lbb_k-\omega t^2|}{\lbb_k \omega t^2} |{\rm Im} \<R_k, DR\>|.
\end{align*}
Applying Lemma \ref{Lem-P-ve-U} and using the a prior bound \eqref{lbbn-Tt} we get
\be\label{est-e2}
|\mathfrak{E}_{2}|\leq C|t|^{-2\delta}\|R\|_{L^2}^2\leq CX(t).
\ee

{\it $(iii)$ Estimate of $\mathfrak{E}_{3}+\mathfrak{E}_{4}$.}
We first treat the quadratic terms of $R$ in $\mathfrak{E}_{3}$ and $\mathfrak{E}_{4}$.
From \eqref{Q-decay} we  have
\begin{align}
\sum_{k=1}^K \frac{1}{\lbb_{k}}{\rm Im} \<2|U|^2R+U^2\ol{R}, R_k\>
&=\sum_{k=1}^K \frac{1}{\lbb_{k}}{\rm Im} \<2|U_k|^2R_k+U_k^2\ol{R}_k, R_k\>+\calo(\|R\|^2_{L^2})\nonumber\\
&=\sum_{k=1}^K \frac{1}{\lbb_{k}}{\rm Im} \int U_k^2\ol{R}_k^2dx+\calo(\|R\|^2_{L^2}),\nonumber
\end{align}
and
\begin{align}
{\rm Re} \< 2U|R|^2+\ol{U}R^2,  \partial_t U \>
&=\sum_{k=1}^{K}{\rm Re} \< 2U_k|R_k|^2+\ol{U}_kR_k^2,  \partial_t U_k \>+\calo(\|R\|^2_{L^2}).\label{est-e3-1}
\end{align}
By \eqref{exp-eta}, Lemmas \ref{lem-app} and \ref{Lem-P-ve-U},
a straightforward computation shows that, $1\leq k\leq K$,
\ben
\partial_tU_k(x)=\frac{i}{\lbb_k}U_k+\frac{b_k}{2\lbb_k}U_k+b_k\frac{x-\a_k}{\lbb_k}\cdot \nabla U_k-v_k\cdot \nabla U_k+\calo(\lbb_k^{-\half}f(\frac{x-\a_k}{\lbb_k})),
\enn
for some function $f$ satisfying $|f(y)|\leq C\<y\>^{-2}$.
Plugging this into \eqref{est-e3-1} and
applying  Lemma \ref{Lem-P-ve-U} we get
\be\ba\label{est-e3-2}
 {\rm Re} \< 2U|R|^2+\ol{U}R^2,  \partial_t U \>
&=\sum_{k=1}^{K}\frac{1}{\lbb_{k}}{\rm Im} \int \ol{U}_k^2R_k^2dx
+\sum_{k=1}^{K}\frac{b_k}{2\lbb_{k}}{\rm Re}\int 2|U_k|^2|R_k|^2+\ol{U}_k^2R_k^2dx\\
&\quad+\sum_{k=1}^{K}b_k{\rm Re} \int \frac{x-\a_k}{\lbb_k}\cdot \nabla \ol{U}_k(\ol{U}_kR_k^2+2U_k|R_k|^2)dx
+\calo(t^{-2}\|R\|^2_{L^2}).
\ea\ee

Regarding the remaining cubic term in $\mathfrak{E}_{4}$,
using equation \eqref{etan-Rn} we get
\ben
|\< |R|^2R,  \partial_t {U} \>|\leq \int|R|^3(|DU|+|U|^3+|\eta|)dx \leq C\sum_{k=1}^{K}\int|R|^3(|DU_k|+|U_k|^3)dx+\int|R|^3|\eta|dx
\enn
From the Sobolev embeddings, the a prior bound \eqref{R-Tt} and Lemma \ref{Lem-P-ve-U}, it follows that
\be\ba\label{est-e3-3}
|\< |R|^2R,  \partial_t {U} \>|
&\leq C\sum_{k=1}^{K}(\|R^3\|_{\dot{H}^{\frac{1}{4}}}\|U_k\|_{\dot{H}^{\frac{3}{4}}}+\|R\|_{L^6}^3
\|U_k\|_{L^6}^3)+C\|\eta\|_{L^2}\|R\|_{L^6}^3\\
&\leq C\sum_{k=1}^{K}(\|R\|_{L^2}^{\frac{1}{2}}\|R\|_{\dot{H}^{\frac{1}{2}}}^{\frac{5}{2}}
\|U_k\|_{L^2}^{\frac{1}{4}}\|U_k\|_{\dot{H}^{1}}^{\frac{3}{4}}
+\|R\|_{L^2}\|R\|_{\dot{H}^{\frac{1}{2}}}^{2}
\|U_k\|_{L^6}^{3})+C\|\eta\|_{L^2}\|R\|_{L^2}\|R\|^2_{\dot{H}^\half}\\
&\leq C\sum_{k=1}^{K}(\lbb_k^{-\frac{3}{4}}\|R\|^\half_{L^2}\|R\|^\half_{\dot{H}^{\frac 12}}+\lbb_k^{-1}\|R\|_{L^2})
\|R\|_{\dot{H}^{\frac{1}{2}}}^{2}+CX(t)\leq CX(t).
\ea\ee

Thus, collecting \eqref{est-e3-1}, \eqref{est-e3-2} and \eqref{est-e3-3} together
we obtain
\be\ba\label{est-e3}
\mathfrak{E}_{3}+\mathfrak{E}_{4}
=-\sum_{k=1}^{K}\(\frac{b_k}{2\lbb_{k}}{\rm Re}\int 2|U_k|^2|R_k|^2+\ol{U}_k^2R_k^2dx
+b_k{\rm Re} \int \frac{x-\a_k}{\lbb_k}\cdot \nabla \ol{U}_k(\ol{U}_kR_k^2+2U_k|R_k|^2)dx\)
+\calo(X(t)).
\ea\ee

{\it $(ii)$ Estimate of $\mathfrak{E}_{5}$.}
The terms in $\mathfrak{E}_{5}$ have the cubic and higher orders in $R$,
and can also be bounded by using the Sobolev embeddings, \eqref{R-Tt} and Lemma \ref{Lem-P-ve-U}:
\be\ba\label{est-e5}
|\mathfrak{E}_{5}|&=|\sum_{k=1}^{K} \lbb_{k}^{-1} {\rm Im} \< 2U|R|^2+\ol{U}R^2+|R|^2R, R_k\>|\\
&\leq C\sum_{k=1}^{K}\lbb_k^{-1}(\int |U||R|^3 dx+\int |R|^4dx)\leq C\sum_{k=1}^{K}\lbb^{-1}_k(\|U\|_{L^2}\|R\|_{L^6}^3+\|R\|_{L^4}^4)\\
&\leq C\sum_{k=1}^{K}\lbb^{-1}_k(\|R\|_{L^2}\|R\|_{\dot{H}^\half}^2+\|R\|^2_{L^2}\|R\|_{\dot{H}^\half}^2)
\leq CX(t).
\ea\ee

{\it $(ii)$ Estimate of $\mathfrak{E}_{6}$.}
We expand the nonlinearity and use the integration by parts to get
\be\ba\label{est-e6-1}
 \mathfrak{E}_{6}
 &=-{\rm Im} \<2U|R|^2+\ol{U}R^2+|R|^2R, \eta\>
   + {\rm Im} \<R, D \eta +\sum_{k=1}^K  \frac{1}{\lbb_k}\eta\Phi_k -2|U|^2\eta+U^2\ol{\eta}\>  \\
 &=:  \mathfrak{E}_{6.1} +  \mathfrak{E}_{6.2}.
\ea\ee

Let us first estimate the first term $\mathfrak{E}_{6.1}$ on the right-hand side above,
which have the quadratic and higher orders of $R$.
Using the Sobolev embeddings, \eqref{R-Tt} and Lemma \ref{Lem-P-ve-U} we have
\be\ba\label{est-e6-6}
 |\mathfrak{E}_{6.1}|
&\leq \|\eta\|_{L^2}(\|U\|_{L^6}\|R\|_{L^6}^2+\|R\|_{L^6}^3)\\
&\leq C\|\eta\|_{L^2}(|t|^{-\frac{2}{3}}\|R\|_{L^2}^{\frac{2}{3}}\|R\|_{\dot{H}^\half}^\frac{4}{3}
+\|R\|_{L^2}\|R\|_{\dot{H}^\half}^2)\\
&\leq C\|\eta\|_{L^2}X(t)\leq CX(t).
\ea\ee

Next, we treat the more delicate second part $\mathfrak{E}_{6.2}$
on the right-hand side of \eqref{est-e6-1},
which contains the linear terms of $R$.
In view of \eqref{exp-eta} and \eqref{etan-Rn},
we write
\ben
\eta=\sum_{k=1}^{K}\eta_k+\wt\eta
\enn
with
\ben\ba
\eta_k=e^{i\g_{k}}\lambda_{k}^{-\frac 32}
&\left(-(\lambda_{k}\dot{v}_{k}+b_kv_k)G_1
  -(\lambda_{k}\dot{b}_{k}+\half b_k^2) S_1
  -i(\dot{\alpha}_{k}-v_k)\cdot \nabla Q\right.\\
     & \left.-i(\dot{\lambda}_{k}+b_k)\Lambda Q -(\lambda_{k}\dot{\g}_{k}-1)Q\right)(t,\frac{x-\alpha_{k}}{\lambda_{k}}).
\ea\enn
In view of \eqref{Q-decay},  Lemmas \ref{lem-app} and \ref{Lem-P-ve-U},
we have for $\nu=0,1$
\ben
|\nabla^\nu\wt\eta(x)|\leq C\sum_{k=1}^{K}\lbb_k^{-\frac{3}{2}-\nu}\<\frac{x-\a_k}{\lbb_k}\>^{-2}(|t|Mod_k+|t|^4)
+|\nabla^\nu(|U|^2U-\sum_{k=1}^{K}|U_k|^2U_k)|,
\enn
and thus
\be\label{est-wteta}
\|\nabla^\nu\wt\eta\|_{L^2}\leq C|t|^{-2-2\nu}(|t|Mod+|t|^4).
\ee
Moreover, using $|\eta_k(x)|\leq C\lbb_k^{-\frac{3}{2}}\<\frac{x-\a_k}{\lbb_k}\>^{-2}Mod_k$, Lemmas \ref{lem-dep} and \ref{Lem-P-ve-U}
we obtain
\begin{align*}
  |\<R,-2|U|^2\eta_k+U^2\ol{\eta}_k\>-\<R,-2|U_k|^2\eta_k+U_k^2\ol{\eta}_k\>|
 \leq C Mod_k\|R\|_{L^2}\leq C|t|^{6-2\delta},
\end{align*}
and
\begin{align*}
 |\<R,  {\lbb_k}^{-1} \eta\Phi_k\>-\<R,  {\lbb_k}^{-1} \eta_k\>|
\leq \lbb_k^{-1}\|R\|_{L^2}(\|\eta_k(1-\Phi_k)\|_{L^2}+\sum_{l\neq k}\|\eta_l\Phi_k\|_{L^2})
 \leq C\lbb_k^{-\half}Mod_k\|R\|_{L^2}\leq C|t|^{6-2\delta}.
\end{align*}
Thus, the second term on the right-hand side of \eqref{est-e6-1}
has the expansion
\be\ba\label{est-e6-3}
 \mathfrak{E}_{6.2}&=\sum_{k=1}^K{\rm Im} \<R, D \eta_k +  \frac{1}{\lbb_k}\eta_k -2|U_k|^2\eta_k+U_k^2\ol{\eta}_k\>+
{\rm Im} \<R, D \wt\eta +\sum_{k=1}^K  \frac{1}{\lbb_k}\wt\eta\Phi_k -2|U|^2\wt\eta+U^2\ol{\wt\eta}\>+\calo(|t|^{6-2\delta}).
\ea\ee

The contribution of $\wt\eta$ can be estimated by using \eqref{est-wteta},
\be\ba\label{est-e6-4}
&|{\rm Im} \<R, D \wt\eta +\sum_{k=1}^K  \frac{1}{\lbb_k}\wt\eta\Phi_k -2|U|^2\wt\eta+U^2\ol{\wt\eta}\>|\\
\leq& C(\|\wt\eta\|_{H^1}+|t|^{-2}\|\wt\eta\|_{L^2})\|R\|_{L^2}
\leq C(|t|^{3-\delta}+\frac{\|R\|_{L^2}^2}{|t|^{3-\delta}}).
\ea\ee

In order to estimate the remaining terms involving $\eta_k$ in \eqref{est-e6-3},
we use the renormalization \eqref{ren}
together with expansion $Q_k(y)=Q(y)+\calo(|t|\<y\>^{-2})$
to get
\ben\ba
&{\rm Im}\<R, D \eta_k +  \frac{1}{\lbb_k}\eta_k -2|U_k|^2\eta_k+U_k^2\ol{\eta}_k\>\\
&=-\lbb_k^{-2}\left((\lambda_{k}\dot{v}_{k}+b_kv_k)\<\epsilon_{k,2},L_-G_1\>
  +(\lambda_{k}\dot{b}_{k}+\half b_k^2) \<\epsilon_{k,2},L_-S_1\>
  -(\dot{\alpha}_{k}-v_k)\cdot \<\epsilon_{k,1},L_+\na Q\>\right.\\
     &\qquad \left.\quad\quad-(\dot{\lambda}_{k}+b_k)\<\epsilon_{k,1},L_+\Lambda Q\>+(\lambda_{k}\dot{\g}_{k}-1)\<\epsilon_{k,2},L_-Q\>+\calo(|t|Mod_k\|R\|_{L^2})
     \right).
\ea\enn
Then, by virtue of the algebraic property of the linearized operator \eqref{Q-kernel}
and the orthogonality conditions in \eqref{ortho-cond-Rn-wn} and Lemma \ref{lem-mass-local},
an additional $|t|$ factor can be gained:
\be\ba\label{est-e6-5}
|{\rm Im}\<R, D \eta_k +  \frac{1}{\lbb_k}\eta_k -2|U_k|^2\eta_k+U_k^2\ol{\eta}_k\>|
&\leq C\lbb_k^{-2}Mod_k(|\<\epsilon_{k,2},\na Q\>|+|\<\epsilon_{k,2},\Lambda Q\>|+|\<\epsilon_{k,1}, Q\>|+|t|\|R\|_{L^2})\\
&\leq C(\frac{\|R\|_{L^2}^2}{t^2}+|t|^{4-2\delta}).
\ea\ee

Thus, plugging \eqref{est-e6-4} and \eqref{est-e6-5} into \eqref{est-e6-3}
we get
\ben
| \mathfrak{E}_{6.2}|\leq C(|t|^{3-\delta}+\frac{\|R\|_{L^2}^2}{|t|^{3-\delta}}),
\enn
which along with \eqref{est-e6-6} yields that
\be\label{est-e6}
|\mathfrak{E}_{6}|
\leq |\mathfrak{E}_{6,1}| + |\mathfrak{E}_{6,2}|
\leq  C(|t|^{3-\delta}+\frac{\|R\|_{L^2}^2}{|t|^{3-\delta}}+X(t)).
\ee

Therefore, plugging estimates \eqref{I1t1-esti}, \eqref{est-e2}, \eqref{est-e3}, \eqref{est-e5} and \eqref{est-e6} into \eqref{equa-I1t}
we obtain \eqref{dt-E} and finish the proof of Proposition \ref{prop-I1t}.
\end{proof}

Proposition \ref{prop-I2t} below contains the control of the more delicate
localized virial functional.

\begin{proposition} [Control of the  virial part] \label{prop-I2t}
For all $t\in [T_*, T]$, there exists a constant $C(A)>0$ such that
\ben\ba
\frac{d \mathfrak{L}}{dt}=&\sum_{k=1}^{K}\left(\frac{b_{k}}{2\lbb_k}   {\rm Re}
       \< \Delta\chi_A(\frac{x-\alpha_{k}}{\lambda_{k}})R_k, DR_k \>
       +b_{k}   {\rm Re}
       \< \na\chi_A(\frac{x-\alpha_{k}}{\lambda_{k}})\nabla R_k, DR_k\>\right.\\
&\quad\quad\left.+b_{k}{\rm Re}\< \na\chi_A(\frac{x-\alpha_{k}}{\lambda_{k}})\na U_k,2U_k|R_k|^2+\ol{U}_kR_k^2 \>\right)\\
& +\calo\(C(A)\(\ln(2+\|R\|_{H^\half}^{-1})\)^{\half}X(t)
+\frac{\|R\|_{L^2}^2}{|t|^{3-\delta}}+|t|^{3-\delta}\).
\ea\enn
\end{proposition}

\begin{proof}
Using the integration by parts,  we get
\ben\ba
   \frac{d \mathfrak{L}}{dt}
   =&\sum_{k=1}^K \frac{\dot{b}_{k}}{2}
           {\rm Im} \< \nabla\chi_A(\frac{x-\alpha_{k}}{\lambda_{k}})\cdot\nabla R, R_k\>
           +\sum_{k=1}^K\frac{b_{k}}{2}{\rm Im}
       \< \partial_t (\nabla\chi_A(\frac{x-\alpha_{k}}{\lambda_{k}}))\cdot\nabla R, {R}_k\>\\
    &+\sum_{k=1}^K \frac{b_{k}}{2\lbb_k}   {\rm Im}
       \< \Delta\chi_A(\frac{x-\alpha_{k}}{\lambda_{k}})R_k, \pa_t {R} \>
   +   \sum_{k=1}^K \frac{b_{k}}{2}    {\rm Im}
       \< \nabla\chi_A(\frac{x-\alpha_{k}}{\lambda_{k}})\cdot ( \nabla R_k + \na R \Phi_k), \partial_t  R \>  \\
    =:& \sum\limits_{j=1}^4 \mathfrak{L}_j.
         \ea\enn

$(i)$ {\it Estimates of $ \mathfrak{L}_1$ and $\mathfrak{L}_2$.}
Straightforward computations show that for $1\leq k\leq K$,
\ben\ba
\frac{\dot{b}_{k}}{2}
           \< \nabla\chi_A(\frac{x-\alpha_{k}}{\lambda_{k}})\cdot\nabla R, R_k\>= \frac{\lbb_k\dot{b}_{k}+\half b_k^2}{2\lbb_k}
           \< \nabla\chi_A(\frac{x-\alpha_{k}}{\lambda_{k}})\cdot\nabla R, R_k\>
- \frac{b_{k}^2}{4\lbb_k}
            \< \nabla\chi_A(\frac{x-\alpha_{k}}{\lambda_{k}})\cdot\nabla R, R_k\>,
\ea\enn
and
\ben
\partial_t(\nabla\chi_A(\frac{x-\alpha_k}{\lambda_k}))
= -\nabla^2\chi_A(\frac{x-\alpha_k}{\lambda_k})\cdot
\big(\frac{x-\a_k}{\lbb_k}\cdot \frac{\dot{\lbb}_k +b_k}{\lbb_k} -\frac{x-\a_k}{\lbb_k}\cdot \frac{b_k}{\lbb_k}
  + \frac{ \dot{\a}_k - v_k}{\lbb_k}
  + \frac{v_k}{\lbb_k} \big).
\enn
By Lemmas \ref{lem-inter} and \ref{Lem-P-ve-U}, there exists a constant $C(A)>0$ such that
\begin{align*}
|\mathfrak{L}_1|&\leq C(1+\frac{Mod}{t^2})\sum_{k=1}^{K}|\int \na\chi_A(\frac{x-\alpha_{k}}{\lambda_{k}})\Phi_k\nabla R \ol{R}dx|\nonumber\\
&\leq C\sum_{k=1}^{K}(\|\na \chi_A\|_{L^\9}\|D^{\half}R\|_{L^2}^2+\lbb_k^{-1}\|\na^2 \chi_A\|_{L^\9}\|R\|_{L^2}^2)
\leq C(A)X(t),
\end{align*}
and
\begin{align*}
 |\mathfrak{L}_2|&\leq C(|t|+Mod)\sum_{k=1}^{K}\frac{b_k}{\lbb_k}|\int (1+\frac{x-\a_k}{\lbb_k})\na^2\chi_A(\frac{x-\alpha_{k}}{\lambda_{k}})\Phi_k\nabla R \ol{R}dx|\nonumber\\
 &\leq C\sum_{k=1}^{K}(\|(1+y)\na^2 \chi_A\|_{L^\9}\|D^{\half}R\|_{L^2}^2+\lbb_k^{-1}\|(1+y)\na^3 \chi_A\|_{L^\9}\|R\|_{L^2}^2)
 \leq C(A)X(t).
\end{align*}

$(ii)$ {\it Estimates of $ \mathfrak{L}_3$ and $ \mathfrak{L}_4$.} Using equation \eqref{equa-R},
we first get that for $1\leq k\leq K$,
\be\ba\label{est-l3l4}
 &\frac{b_{k}}{2\lbb_k}   {\rm Im}
       \< \Delta\chi_A(\frac{x-\alpha_{k}}{\lambda_{k}})R_k, \pa_t {R} \>+\frac{b_{k}}{2}   {\rm Im}
       \< \nabla\chi_A(\frac{x-\alpha_{k}}{\lambda_{k}})( \nabla R_k + \na R \Phi_k), \pa_t {R} \>\\
       =&\frac{b_{k}}{2\lbb_k}   {\rm Re}
       \< \Delta\chi_A(\frac{x-\alpha_{k}}{\lambda_{k}})R_k, DR\>
       +\frac{b_{k}}{2}   {\rm Re}
       \<\na\chi_A(\frac{x-\alpha_{k}}{\lambda_{k}})( \nabla R_k + \na R\Phi_k),DR\>\\
       &-  \frac{b_{k}}{2\lbb_k} {\rm Re}
       \<\Delta\chi_A(\frac{x-\alpha_{k}}{\lambda_{k}})R_k, f(u)-f(U)\>
       -  \frac{b_{k}}{2} {\rm Re}
       \< \na\chi_A(\frac{x-\alpha_{k}}{\lambda_{k}})( \nabla R_k + \na R\Phi_k),f(u)-f(U) \>\\
       & - \frac{b_{k}}{2\lbb_k} {\rm Re}
       \<\Delta\chi_A(\frac{x-\alpha_{k}}{\lambda_{k}})R_k, \eta \>-  \frac{b_{k}}{2} {\rm Re}
       \< \na\chi_A(\frac{x-\alpha_{k}}{\lambda_{k}})( \nabla R_k + \na R\Phi_k), \eta \>  \\
       =&: \sum\limits_{j=1}^6 \mathfrak{L}_{34,j}.
\ea\ee

Let us first estimate the quadratic terms of $R$ in \eqref{est-l3l4},
in order to decouple the interaction between the remainders.
We claim that there exists $C(A)>0$ such that, for $1\leq k \leq K$,
\begin{align}
& \mathfrak{L}_{34,1} = \frac{b_{k}}{2\lbb_k}   {\rm Re}
       \< \Delta\chi_A(\frac{x-\alpha_{k}}{\lambda_{k}})R_k, DR_k \>
       + \calo_A(X(t)),\label{clm-loc1}\\
& \mathfrak{L}_{34,2}
     = b_{k}   {\rm Re}
       \< \na\chi_A(\frac{x-\alpha_{k}}{\lambda_{k}})\nabla R_k, DR_k\>
     + \calo_A(X(t)),  \label{clm-loc2}
\end{align}
where $\calo_A(\cdot)$ means up to constants depending on $A$.

Once  \eqref{clm-loc1} and \eqref{clm-loc2} are proved,
we have that for $1\leq k\leq K$,
\ben\ba
 \mathfrak{L}_{34,1} + \mathfrak{L}_{34,2}
 =\frac{b_{k}}{2\lbb_k}   {\rm Re}
       \< \Delta\chi_A(\frac{x-\alpha_{k}}{\lambda_{k}})R_k, DR_k \>+
       b_{k}   {\rm Re}
       \< \na\chi_A(\frac{x-\alpha_{k}}{\lambda_{k}})\nabla R_k, DR_k\>+\calo_A(X(t)).
\ea\enn
It should be mentioned that Lemma \ref{lem-dep} is not applicable
to the decoupling of the interaction between the remainders $R$ and $R_k$ in \eqref{clm-loc1} and \eqref{clm-loc2}.
Here we take the advantage of the integration by parts
and the decay of the derivatives of cut-off function $\na \chi$
to decouple different remainder bubbles $R_k$ and $R_l$, $k\not =l$.

{\it Estimate of $\eqref{clm-loc1}$.}
First we prove \eqref{clm-loc1}. By using $R=\sum_{k=1}^{K}R_k$, we have
\ben\ba
   \mathfrak{L}_{34,1} -\frac{b_{k}}{2\lbb_k}   {\rm Re}
       \< \Delta\chi_A(\frac{x-\alpha_{k}}{\lambda_{k}})R_k, DR_k \>
       =\sum_{l\neq k}\frac{b_{k}}{2\lbb_k} \< \Delta\chi_A(\frac{x-\alpha_{k}}{\lambda_{k}})R_k, DR_l \>.
\ea\enn
To estimate the above remainder terms, we find that for $l\neq k$,
\ben\ba
 \frac{b_{k}}{2\lbb_k}   {\rm Re}
       \< \Delta\chi_A(\frac{x-\alpha_{k}}{\lambda_{k}})R_k, DR_l\>
        =\frac{b_{k}}{2\lbb_k}   {\rm Re}
       \< \Delta\chi_A(\frac{x-\alpha_{k}}{\lambda_{k}})R_k, DR\Phi_l \>
       +\frac{b_{k}}{2\lbb_k}   {\rm Re}
       \< \Delta\chi_A(\frac{x-\alpha_{k}}{\lambda_{k}})R_k, DR_l-DR\Phi_l \>.
\ea\enn
Note that,
by the Calder\'on estimate in Lemma \ref{lem-l2-pro2},
\be\label{est-l3-1}
|\frac{b_{k}}{2\lbb_k}   {\rm Re}
       \< \Delta\chi_A(\frac{x-\alpha_{k}}{\lambda_{k}})R_k, DR_l-DR\Phi_l \>|
       \leq C(A)|t|^{-1}\|R\|_{L^2}^2.
\ee
Set $\Delta\chi_{A,k}(x):=\Delta\chi_A(\frac{x-\alpha_{k}}{\lambda_{k}})\Phi_k(x)$.
Splitting $D=D^\half D^\half$ with $D^\half$ being  self-adjoint
and performing the integrating by parts formula
we then get
\be\ba  \label{esti.2}
&\frac{b_{k}}{2\lbb_k}   {\rm Re}
       \< \Delta\chi_A(\frac{x-\alpha_{k}}{\lambda_{k}})R_k, DR\Phi_l \>\\
&=\frac{b_{k}}{2\lbb_k}   {\rm Re}
       \< \Delta\chi_{A,k}\Phi_l, |D^\half R|^2 \>+
\frac{b_{k}}{2\lbb_k}   {\rm Re}
       \< D^\half(\Delta\chi_{A,k}\Phi_lR)
       -D^\half R\Delta\chi_{A,k}\Phi_l, D^\half R \>
\ea\ee

Note that $\Delta \chi(y)$ decays exponentially fast for $|y|$ large,
and so
\be\ba\label{est-l3-2}
|\frac{b_{k}}{2\lbb_k}   {\rm Re}
       \< \Delta\chi_{A,k}\Phi_l, |D^\half R|^2 \>|
       &\leq C|t|^{-1}\int_{|x-x_k|\geq 4\sigma}|\Delta\chi_A(\frac{x-\alpha_{k}}{\lambda_{k}})||D^\half R|^2dx\\
       &\leq C(A)|t|^{-1}e^{-\frac{C}{|t|}}\|D^\half R\|_{L^2}^2
       \leq C(A)X(t).
\ea\ee

Moreover, by  Lemma \ref{lem-l2-pro}, we have
\be\ba   \label{esti.1}
 |\frac{b_{k}}{2\lbb_k}   {\rm Re}
       \< D^\half(\Delta\chi_{A,k}\Phi_lR)
       -D^\half R\Delta\chi_{A,k}\Phi_l, D^\half R \>|
&\leq C|t|^{-1}\|D^\half R\|_{L^2}\|D^\half(\Delta\chi_{A,k}\Phi_lR)
       -D^\half R\Delta\chi_{A,k}\Phi_l\|_{L^2}\\
&\leq C|t|^{-1}\|D^\half R\|_{L^2}\|R\|_{L^2}\|\mathcal{F}(D^\half(\Delta\chi_{A,k}\Phi_l))\|_{L^1}.
\ea\ee

We claim that
\begin{align}  \label{D12Dchi-L1}
  \|\mathcal{F}(D^\half(\Delta\chi_{A,k}\Phi_l))\|_{L^1}\leq C(A).
\end{align}
Actually, using the exponential decay of $\na^\nu\chi(y)$ when $\nu\geq 2$,
\begin{align} \label{F-Deltachi.1}
  |\mathcal{F}(\Delta\chi_{A,k}\Phi_l)(\xi)|\leq C\lbb_k^{-1}\int_{|x-x_k|\geq 4\sigma} |\Delta\chi(\frac{x-\alpha_{k}}{A\lambda_{k}})| dx\leq C.
\end{align}
Since ${\rm supp}(\Phi_l) \subseteq \{|x-x_k|\geq 4\sigma\}$,
we have that for $|\xi|>1$,
\begin{align}  \label{F-Deltachi.2}
|\mathcal{F}(\Delta\chi_{A,k}\Phi_l)(\xi)|&\leq \frac{C}{|\xi|^2}|\mathcal{F}(\na^2(\Delta\chi_{A,k}\Phi_l))(\xi)|
\leq \frac{C}{|\xi|^2}\int|\na^2(\Delta\chi_{A,k}\Phi_l)|dx  \nonumber \\
&\leq \frac{C}{|\xi|^2}\int_{|x-x_k|\geq 4\sigma}\lbb_k^{-2}|\na^4\chi_{A}(\frac{x-\alpha_{k}}{\lambda_{k}})|
   +\lbb_k^{-1} |\na^3\chi_{A}(\frac{x-\alpha_{k}}{\lambda_{k}})|
+|\Delta\chi_{A}(\frac{x-\alpha_{k}}{\lambda_{k}})|dx  \nonumber  \\
&\leq C(A)\frac{e^{-\frac{C}{\lbb_k}}}{|\xi|^2\lbb_k^2}\leq C(A)|\xi|^{-2}
\end{align}
for some $C(A)>0$.
By \eqref{F-Deltachi.1} and \eqref{F-Deltachi.2},
\ben
\|\mathcal{F}(D^\half(\Delta\chi_{A,k}\Phi_l))\|_{L^1}=\int  |\xi|^\half|\mathcal{F}(\Delta\chi_{A,k}\Phi_l)(\xi)|d \xi\leq C(A).
\enn
Thus, we obtain \eqref{D12Dchi-L1}, as claimed.

Hence, taking into account \eqref{esti.1} we get
\be\label{est-l3-3}
|\frac{b_{k}}{2\lbb_k}   {\rm Re}
       \< D^\half(\Delta\chi_{A,k}\Phi_lR)
       -D^\half R\Delta\chi_{A,k}\Phi_l, D^\half R \>|
       \leq C(A) |t|^{-1}\|D^\half R\|_{L^2}\|R\|_{L^2}\leq C(A)X(t).
\ee
Combining estimates \eqref{est-l3-1}, \eqref{esti.2}, \eqref{est-l3-2} and \eqref{est-l3-3} altogether we obtain
\be\label{est-l-1}
  \mathfrak{L}_{34,1} =\frac{b_{k}}{2\lbb_k}   {\rm Re}
       \< \Delta\chi_A(\frac{x-\alpha_{k}}{\lambda_{k}})R_k, DR_k \>+\calo_A(X(t)),
\ee
which proves the desired estimate \eqref{clm-loc1}.

{\it Estimate of \eqref{clm-loc2}.}
Next we prove the more difficult estimate \eqref{clm-loc2}.
First, we expand by the Leibniz rule,
\be\ba\label{est-l4-1}
  \mathfrak{L}_{34,2} &=b_{k}   {\rm Re}
       \< \na\chi_A(\frac{x-\alpha_{k}}{\lambda_{k}})\nabla R\Phi_k, DR\>
       +\frac{b_{k}}{2}   {\rm Re}
       \< \na\chi_A(\frac{x-\alpha_{k}}{\lambda_{k}})R\na\Phi_k, DR\>.
\ea\ee
An application of the integration by parts  yields
\ben\ba
   \frac{b_{k}}{2}   {\rm Re}
       \< \na\chi_A(\frac{x-\alpha_{k}}{\lambda_{k}})R\na\Phi_k, DR\>
       =&\frac{b_{k}}{2}   {\rm Re}
       \< \na\chi_A (\frac{x-\alpha_{k}}{\lambda_{k}}) \na\Phi_k, |D^\half R|^2\> \\
       &+\frac{b_{k}}{2}   {\rm Re}
       \< D^\half(\na\chi_A (\frac{x-\alpha_{k}}{\lambda_{k}}) \na\Phi_kR)
       -D^\half R\na\chi_A(\frac{x-\alpha_{k}}{\lambda_{k}})\na\Phi_k, D^\half R\>.
\ea\enn
We also see that there exists $C(A)>0$ such that
\ben
|\frac{b_{k}}{2}   {\rm Re}
       \< \na\chi_A(\frac{x-\alpha_{k}}{\lambda_{k}})\na\Phi_k, |D^\half R|^2\>|\leq C(A)|t|\|D^\half R\|_{L^2}^2,
\enn
and by Lemma \ref{lem-l2-pro},
\ben\ba
  & |\frac{b_{k}}{2}   {\rm Re}
       \< D^\half(\na\chi_A(\frac{x-\alpha_{k}}{\lambda_{k}})\na\Phi_kR)
         -D^\half R\na\chi_A(\frac{x-\alpha_{k}}{\lambda_{k}})\na\Phi_k, D^\half R\>| \\
&\leq C|t|\|D^\half R\|_{L^2}\|R\|_{L^2}
\|\mathcal{F}(D^\half(\na\chi_A(\frac{\cdot-\alpha_{k}}{\lambda_{k}})\na\Phi_k))\|_{L^1}.
\ea\enn
Similar to the calculation of \eqref{F-Deltachi.1}, \eqref{F-Deltachi.2} and using that $\na\Phi\in \mathscr{S}$, we have
\ben\ba
&|\mathcal{F}(\na\chi_A(\frac{\cdot-\alpha_{k}}{\lambda_{k}})\na\Phi_k)(\xi)|\leq C\int |\na\chi(\frac{x-\alpha_{k}}{A\lambda_{k}})\na\Phi_k| dx\leq C,\\
&|\mathcal{F}(\na\chi_A(\frac{\cdot-\alpha_{k}}{\lambda_{k}})\na\Phi_k)(\xi)|
\leq \frac{C}{|\xi|^2}\int|\na^2(\na\chi(\frac{x-\alpha_{k}}{A\lambda_{k}})\na\Phi_k)|dx\leq C(A)|\xi|^{-2}.
\ea\enn
So we get $\|\mathcal{F}(D^\half(\na\chi_A(\frac{\cdot-\alpha_{k}}{\lambda_{k}})\na\Phi_k))\|_{L^1}\leq C(A)$ and thus
 \be\label{est-l4-2}
 |\frac{b_{k}}{2}   {\rm Re}
       \< \na\chi_A(\frac{x-\alpha_{k}}{\lambda_{k}})R\na\Phi_k, DR\>|=\calo_A(X(t)).
 \ee
Thus, plugging this into \eqref{est-l4-1} we obtain
\begin{align}   \label{esti.3}
\frac{b_{k}}{2}   {\rm Re}
       \< \na\chi_A(\frac{x-\alpha_{k}}{\lambda_{k}})( \nabla R_k + \na R\Phi_k), DR\>
       =b_{k}   {\rm Re}
       \< \na\chi_A(\frac{x-\alpha_{k}}{\lambda_{k}})\nabla R\Phi_k, DR\>+\calo_A(X(t)).
\end{align}

Let $\na\chi_{A,k}(x):=\na\chi_A(\frac{x-\alpha_{k}}{\lambda_{k}})\Phi_k(x)$.
Using the commutator formula  $[D, \nabla]=0$, \eqref{est-l4-2} and
integrating by parts twice  we get
\ben\ba
&b_{k}   {\rm Re}
       \< \na\chi_A(\frac{x-\alpha_{k}}{\lambda_{k}})\nabla R\Phi_k, DR\>\\
&=-\frac{b_{k}}{\lbb_k}   {\rm Re}
       \< \Delta\chi_A(\frac{x-\alpha_{k}}{\lambda_{k}})R\Phi_k, DR\>
-b_{k}   {\rm Re}
       \< \na\chi_A(\frac{x-\alpha_{k}}{\lambda_{k}})R\nabla\Phi_k, DR\>
-b_{k}   {\rm Re}
       \< \na\chi_A(\frac{x-\alpha_{k}}{\lambda_{k}})R\Phi_k,\nabla DR\>\\
&=-\frac{b_{k}}{\lbb_k}   {\rm Re}
       \< \Delta\chi_A(\frac{x-\alpha_{k}}{\lambda_{k}})R\Phi_k, DR\>
-b_{k}   {\rm Re}
       \< \na\chi_A(\frac{x-\alpha_{k}}{\lambda_{k}})R\Phi_k,D\nabla R\>+\calo_A(X(t))\\
&=-\frac{b_{k}}{\lbb_k}   {\rm Re}
       \< \Delta\chi_{A,k}R, DR\>
-b_{k}   {\rm Re}
       \< \na\chi_{A,k}D R, \nabla R\>
 -b_{k}   {\rm Re}
       \< D(\na\chi_{A,k}R)-\na\chi_{A,k}D R, \na R\>+\calo_A(X(t)),
\ea\enn
which implies that
\be\ba\label{est-l4-4}
&b_{k}   {\rm Re}
       \< \na\chi_A(\frac{x-\alpha_{k}}{\lambda_{k}})\nabla R\Phi_k, DR\>\\
=&-\frac{b_{k}}{2\lbb_k}   {\rm Re}
       \< \Delta\chi_{A,k}R, DR\> -\frac{b_{k}}{2}   {\rm Re}
       \< D(\na\chi_{A,k}R)-\na\chi_{A,k}D R,\nabla R\>+\calo_A(X(t)).
\ea\ee
Plugging \eqref{est-l4-4} into \eqref{esti.3} we get that
\be\ba\label{est-l-2}
&\frac{b_{k}}{2}   {\rm Re}
       \< \na\chi_A(\frac{x-\alpha_{k}}{\lambda_{k}})( \nabla R_k + \na R\Phi_k), DR\>\\
=&-\frac{b_{k}}{2\lbb_k}   {\rm Re}
       \< \Delta\chi_{A,k}R, DR\> -\frac{b_{k}}{2}   {\rm Re}
       \< D(\na\chi_{A,k}R)-\na\chi_{A,k}D R,\nabla R\>+\calo_A(X(t)).
\ea\ee

Moreover, a similar calculation as in the proof of \eqref{est-l4-4} shows that
\be\ba\label{est-l-3}
&b_{k}   {\rm Re}
       \< \na\chi_A(\frac{x-\alpha_{k}}{\lambda_{k}})\nabla R_k, DR_k\>\\
=&-\frac{b_{k}}{2\lbb_k}   {\rm Re}
       \< \Delta\chi_{A,k}R, DR_k\> -\frac{b_{k}}{2}   {\rm Re}
       \< D(\na\chi_{A,k}R)-\na\chi_{A,k}D R,\nabla R_k\>+\calo_A(X(t)).
\ea\ee
Then, comparing  \eqref{est-l-2} with \eqref{est-l-3}
we find that
\be\ba\label{est-l4-5}
&|\frac{b_{k}}{2}   {\rm Re}
       \< \na\chi_A(\frac{x-\alpha_{k}}{\lambda_{k}})( \nabla R_k + \na R\Phi_k), DR\>
-b_{k}   {\rm Re}
       \< \na\chi_A(\frac{x-\alpha_{k}}{\lambda_{k}})\nabla R_k, DR_k\>|\\
\leq& C\frac{b_{k}}{\lbb_k}  |{\rm Re}
       \< \Delta\chi_{A,k}R, DR-DR_k\>|
       +C b_{k}|   {\rm Re}
       \< D(\na\chi_{A,k}R)-\na\chi_{A,k}D R,\na R-\nabla R_k\>|+\calo_A(X(t)).
\ea\ee

Thus, the proof of \eqref{clm-loc2} is now reduced to proving that the R.H.S.
of in \eqref{est-l4-5} is of order $\calo_A(X(t))$.

To this end, it is easy to see from  \eqref{clm-loc1} that
\be\label{est-l4-11}
\frac{b_{k}}{\lbb_k}  |{\rm Re}
       \< \Delta\chi_{A,k}R, DR-DR_k\>|=\calo_A(X(t)).
\ee
Moreover, one has
\be\ba \label{est-l4-6}
& b_{k}   {\rm Re}
       \< D(\na\chi_{A,k}R)-\na\chi_{A,k}D R,\na R-\nabla R_k\> \\
=&b_{k}   \sum_{l\neq k}{\rm Re}
       \< D(\na\chi_{A,k}R)-\na\chi_{A,k}D R,\na R_l\>\\
=&b_{k}   \sum_{l\neq k}{\rm Re}
       \< D(\na\chi_{A,k}R)-\na\chi_{A,k}D R,R\na \Phi_l\>
+ b_{k}   \sum_{l\neq k}{\rm Re}
       \< D(\na\chi_{A,k}R)-\na\chi_{A,k}D R,\na R\Phi_l\>.
\ea\ee

We first estimate the first term on the right-hand side of \eqref{est-l4-6}
by using H\"older's inequality,
\be \label{est-l4-6.1}
|b_{k}   {\rm Re}
       \< D(\na\chi_{A,k}R)-\na\chi_{A,k}D R,R\na \Phi_l\>|
       \leq Cb_k\|(D(\na\chi_{A,k}R)-\na\chi_{A,k}D R)\na \Phi_l\|_{L^2}\|R\|_{L^2}.
\ee
Applying Lemma \ref{lem-l2-pro} we have for $l\neq k$,
\be\ba\label{est-l4-7}
&\|(D(\na\chi_{A,k}R)-\na\chi_{A,k}D R)\na \Phi_l\|_{L^2}\\
\leq& \|D(\na\chi_{A,k}\na\Phi_lR)-\na\chi_{A,k}\na \Phi_lD R\|_{L^2}
+\|D(\na\chi_{A,k}\na\Phi_lR)-D(\na\chi_{A,k}R)\na\Phi_l\|_{L^2}\\
\leq& C\|\mathcal{F}(D(\na\chi_{A,k}\na \Phi_l))\|_{L^1}\|R\|_{L^2}+\|\mathcal{F}(D\na\Phi_l)\|_{L^1}\|\na\chi_{A,k}R\|_{L^2}.
\ea\ee
Note that, there exists $C(A)>0$ such that
\ben
|\mathcal{F}(\na\chi_{A,k}\na \Phi_l)(\xi)|
\leq
C\int |\na\chi_A(\frac{x-\alpha_{k}}{\lambda_{k}})
\Phi_k(x)\na \Phi_l(x)|dx\leq C(A),
\enn
and  the exponential decay property of $\na^\nu\chi(y)$, $2\leq \nu\leq 4$,
and the boundedness of $\na\chi(y)$ yield that, for $|\xi|>1$,
\ben\ba
|\mathcal{F}(\na\chi_{A,k}\na \Phi_l)(\xi)|
\leq& \frac{C}{|\xi|^3}\int\na^3(\na\chi_{A,k}\na \Phi_l)dx\\
\leq& \frac{C}{|\xi|^3}
\int \sum\limits_{\nu_1+\nu_2=3} \lbb_k^{-N_1}
|\na^{1+\nu_1} \chi_{A,k}(x)  \na^{1+\nu_2} \Phi_l(x)| dx \\
\leq& \frac{C}{|\xi|^3}(\int_{|x-x_k|\geq 4\sigma}\lbb_k^{-3}|\na^4\chi_{A}(\frac{x-\alpha_{k}}{\lambda_{k}})|
dx+1)\\
\leq& C(A)|\xi|^{-3}(\lbb_k^{-2}e^{-\frac{C}{\lbb_k}}+1)\leq C(A)|\xi|^{-3}.
\ea\enn
So it follows that
\be\ba\label{est-l4-8}
\|\mathcal{F}(D(\na\chi_{A,k}\na \Phi_l))\|_{L^1}
= \int |\xi||\mathcal{F}(\na\chi_{A,k}\na \Phi_l)|(\xi) d\xi
\leq C(A).
\ea\ee
We also have $\|\mathcal{F}(D\na\Phi_l)\|_{L^1}\leq C$, since $\na\Phi(x)\in \mathscr{S}$.
Plugging this fact and \eqref{est-l4-8} into \eqref{est-l4-7} we come to
\ben
\|(D(\na\chi_{A,k}R)-\na\chi_{A,k}D R)\na \Phi_l\|_{L^2}\leq C(A)\|R\|_{L^2},
\enn
which, via \eqref{est-l4-6.1}, yields that
\be\label{est-l4-9}
|b_{k}   {\rm Re}
       \< D(\na\chi_{A,k}R)-\na\chi_{A,k}D R,R\na \Phi_l\>|
       \leq C(A)b_k\|R\|_{L^2}^2.
\ee

It remains to estimate the second term on the right-hand side of \eqref{est-l4-6}. We claim that
\be\label{est-l-com}
|b_{k}   {\rm Re}
       \< D(\na\chi_{A,k}R)-\na\chi_{A,k}D R,\na R\Phi_l\>|
\leq C(A)b_k \|R\|_{H^\half}^2.
\ee
To show this, we let $Tg:=\na ((D(\na\chi_{A,k}g)-\na\chi_{A,k}D g)\Phi_l)$. On one hand,
using integrating by parts we have
\ben
\<f, Tg\>=-\<\na f\Phi_l, D(\na\chi_{A,k}g)-\na\chi_{A,k}D g)\>.
\enn
Then it follows from Lemma \ref{lem-l2-pro2} and the exponential decay of $\na^\nu\chi(y)$ with $\nu\geq 2$ that
\ben\ba
|\<f, Tg\>|
&\leq \|\na f\|_{L^2}\|(D(\na\chi_{A,k}g)-\na\chi_{A,k}D g)\Phi_l\|_{L^2}\\
&\leq \|\na f\|_{L^2}\(\|D(\na\chi_{A,k}\Phi_l g)-\na\chi_{A,k}\Phi_lD g\|_{L^2}
+\|D(\na\chi_{A,k}g\Phi_l)-D(\na\chi_{A,k} g)\Phi_l\|_{L^2}\)\\
&\leq C\|\na f\|_{L^2}\(\|\na(\na\chi_{A,k} \Phi_l)\|_{L^\9}\|g\|_{L^2}+\|\na \Phi_l\|_{L^\9}\|\na\chi_{A,k}g\|_{L^2}\)\\
&\leq C(A)\|f\|_{H^1}\| g\|_{L^2}.
\ea\enn
On the other hand,  we get
\ben\ba
\<f, Tg\>
=\<f\na\Phi_l, D(\na\chi_{A,k}g)-\na\chi_{A,k}D g\>+\<f\Phi_l, D\na(\na\chi_{A,k}g)-\na(\na\chi_{A,k}D g)\>,
\ea\enn
which along with Lemma \ref{lem-l2-pro2} and the exponential decay of $\na^\nu\chi(y)$  with $\nu\geq 2$ yields that
\ben\ba
|\<f, Tg\>|&\leq \|f\|_{L^2}\left(\|\(D(\na\chi_{A,k}g)-\na\chi_{A,k}D g\)\na\Phi_l\|_{L^2}
+\|\(D(\na\chi_{A,k}\na g)-\na\chi_{A,k}D \na g\)\Phi_l\|_{L^2}\right.\\
&\left.\quad+\|D(\na(\na\chi_{A,k})g\Phi_l)-D(\na(\na\chi_{A,k}) g)\Phi_l\|_{L^2}+\|D(\na(\na\chi_{A,k})g\Phi_l)-\na(\na\chi_{A,k})\Phi_lDg\|_{L^2}\right)\\
&\leq C\|f\|_{L^2} \big(\|\na(\na\chi_{A,k})\|_{L^\9}\|g\|_{L^2}
        + \|\na(\na\chi_{A,k})\|_{L^\9} \|\na g\|_{L^2} \|\Phi_l\|_{L^2} \\
&\qquad  \qquad\ +\|\na \Phi_l\|_{L^\9}\|\na(\na\chi_{A,k}) g\|_{L^2}+\|\na(\na(\na\chi_{A,k})\Phi_l)\|_{L^\9}\|g\|_{L^2} \big)\\
&\leq C(A) \|f\|_{L^2} (\lbb_k^{-1} \|g\|_{L^2} + \|\na g\|_{L^{2}}).
\ea\enn
By taking $f=g=R$ and using interpolation theory
(see the proof of \cite[Lemma F.1]{K-L-R}) we obtain
\ben
|\< D(\na\chi_{A,k}R)-\na\chi_{A,k}D R,\na R\Phi_l\>|=|\<f, Tg\>|\leq C(A)(\|R\|^2_{\dot H^\half}+\lbb_k^{-1}\|R\|_{L^2}^2),
\enn
which yields \eqref{est-l-com}, as claimed.

Thus, combining \eqref{est-l4-6}, \eqref{est-l4-9} and \eqref{est-l-com}
we arrive at
\be\label{est-l4-10}
|b_{k}   {\rm Re}
       \< D(\na\chi_{A,k}R)-\na\chi_{A,k}D R,\na R-\nabla R_k\>|
       \leq C(A)b_k(\|R\|^2_{\dot H^\half}+\lbb_k^{-1}\|R\|_{L^2}^2).
\ee
Then, inserting \eqref{est-l4-11} and \eqref{est-l4-10} into \eqref{est-l4-5}, we conclude
\ben\ba
|\frac{b_{k}}{2}   {\rm Re}
       \< \na\chi_A(\frac{x-\alpha_{k}}{\lambda_{k}})( \nabla R_k + \na R\Phi_k), DR\>
-b_{k}   {\rm Re}
       \< \na\chi_A(\frac{x-\alpha_{k}}{\lambda_{k}})\nabla R_k, DR_k\>|=\calo_A(X(t)),
\ea\enn
and thus  \eqref{clm-loc2}  is proved, as claimed.

{\it Estimate of $\mathfrak{L}_{34.3}$ and $ \mathfrak{L}_{34.4}$.}
We now decouple the interaction between the remainder $R$ and the nonlinearity $f(u)$ in \eqref{est-l3l4}. We expand the nonlinearity $f(u)$ to get
\be\ba\label{est-l-9}
 -(\mathfrak{L}_{34.3} + \mathfrak{L}_{34.4})
=&     \frac{b_{k}}{2\lbb_k}{\rm Re}
       \<\Delta\chi_A(\frac{x-\alpha_{k}}{\lambda_{k}})R_k, 2|U|^2R+U^2\ol{R}\>
      +\frac{b_{k}}{2} {\rm Re}
       \< \na\chi_A(\frac{x-\alpha_{k}}{\lambda_{k}})( \nabla R_k + \na R\Phi_k),2|U|^2R+U^2\ol{R} \>\\
& +     {\rm Re}
       \<\frac{b_{k}}{2\lbb_k}\Delta\chi_A(\frac{x-\alpha_{k}}{\lambda_{k}})R_k+\frac{b_{k}}{2} \na\chi_A(\frac{x-\alpha_{k}}{\lambda_{k}})( \nabla R_k + \na R\Phi_k), 2U|R|^2+\ol{U}R^2+|R|^2R \>
\ea\ee

Again let us first estimate the quadratic terms of $R$ on the right-hand side of \eqref{est-l-9}.
Integration by parts and using Lemma \ref{Lem-P-ve-U}
we derive
\ben\ba
&\frac{b_{k}}{2}{\rm Re}\< \na\chi_A(\frac{x-\alpha_{k}}{\lambda_{k}})( \nabla R_k + \na R\Phi_k),2|U|^2R+U^2\ol{R} \>\\
&=b_{k}{\rm Re}\< \na\chi_A(\frac{x-\alpha_{k}}{\lambda_{k}})\na R_k,2|U|^2R+U^2\ol{R} \>
+\calo(|t|^{-1}\|R\|_{L^2}^2)\\
&=-\frac{b_{k}}{2\lbb_k}{\rm Re}
       \<\Delta\chi_A(\frac{x-\alpha_{k}}{\lambda_{k}})R_k, 2|U|^2R+U^2\ol{R}\>
-b_{k}{\rm Re}\< \na\chi_A(\frac{x-\alpha_{k}}{\lambda_{k}})\na U\Phi_k,2U|R|^2+\ol{U}R^2 \>
+\calo(|t|^{-1}\|R\|_{L^2}^2).
\ea\enn
This yields that
\be\ba\label{est-l-4}
&\frac{b_{k}}{2\lbb_k}{\rm Re}
       \<\Delta\chi_A(\frac{x-\alpha_{k}}{\lambda_{k}})R_k, 2|U|^2R+U^2\ol{R}\>
+\frac{b_{k}}{2}{\rm Re}\< \na\chi_A(\frac{x-\alpha_{k}}{\lambda_{k}})( \nabla R_k + \na R\Phi_k),2|U|^2R+U^2\ol{R} \>\\
&=-b_{k}{\rm Re}\< \na\chi_A(\frac{x-\alpha_{k}}{\lambda_{k}})\na U\Phi_k,2U|R|^2+\ol{U}R^2 \>
+\calo(|t|^{-1}\|R\|_{L^2}^2)\\
&=-b_{k}{\rm Re}\< \na\chi_A(\frac{x-\alpha_{k}}{\lambda_{k}})\na U_k,2U_k|R_k|^2+\ol{U}_kR_k^2 \>
+\calo(|t|X(t)).
\ea\ee
where in the last step we also applied Lemma \ref{lem-dep}.

Regarding the higher order terms on the right-hand side of \eqref{est-l-9},
using the Sobolev embeddings, \eqref{R-Tt}  and Lemma \ref{Lem-P-ve-U} we get
\be\ba\label{est-l-6}
&|\frac{b_{k}}{2\lbb_k}{\rm Re}
       \<\Delta\chi_A(\frac{x-\alpha_{k}}{\lambda_{k}})R_k, 2U|R|^2+\ol{U}R^2+|R|^2R\>|\leq C(A)\frac{b_{k}}{\lbb_k}(\|R\|_{L^4}^3\|U\|_{L^4}+\|R\|_{L^4}^4)\\
&\leq C(A)\frac{b_{k}}{\lbb_k}(\|R\|_{L^2}^{\frac{3}{2}}\|D^\half R\|_{L^2}^{\frac32}\|U\|_{L^2}^\half\|D^\half U\|_{L^2}^\half+\|R\|_{L^2}^2\|D^\half R\|_{L^2}^2)\\
&\leq C(A)(\lbb^{-\frac34}\|R\|_{L^2}^{\frac{3}{2}}\|D^\half R\|_{L^2}^{\frac32}+\lbb^{-\half}\|R\|_{L^2}^2\|D^\half R\|_{L^2}^2)\leq C(A)X(t),
\ea\ee
and
\be\label{est-l-7}
|\frac{b_{k}}{2} {\rm Re}\< \na\chi_A(\frac{x-\alpha_{k}}{\lambda_{k}})R\na\Phi_k, 2U|R|^2+\ol{U}R^2+|R|^2R \>|\leq C(A)|b_k|(\|U\|_{L^4}\|R\|_{L^4}^3+\|R\|_{L^4}^4)\leq C(A)X(t).
\ee
Moreover, the a priori estimate \eqref{R-Tt} implies that $\|D^{\half+\varsigma}R\|_{L^2}\leq C$.
Combining this fact with the fractional chain rule  and \eqref{est-hs}
we derive
\be\ba\label{est-l-8}
&|\frac{b_k}{2}{\rm Re}\< \na\chi_A(\frac{x-\alpha_{k}}{\lambda_{k}})\na R\Phi_k, 2U|R|^2+\ol{U}R^2+|R|^2R \>|\\
&\leq C(A)|b_k|\|R\|_{\dot{H}^\half}(
\|\na\chi_A(\frac{x-\alpha_{k}}{\lambda_{k}})\Phi_k(2U|R|^2+\ol{U}R^2)\|_{\dot{H}^\half}
+\|\na\chi_A(\frac{x-\alpha_{k}}{\lambda_{k}})\Phi_k|R|^2R\|_{\dot{H}^\half})\\
&\leq C(A)|b_k|\|R\|_{\dot{H}^\half}(\|\na\chi_{A,k}\|_{L^\9}\|U\|_{L^\9}\|R\|_{L^\9}
\|R\|_{\dot{H}^\half}
+\|\na\chi_{A,k}\|_{L^\9}\|R\|_{L^\9}^2\|U\|_{\dot{H}^\half}
+\|U\|_{L^\9}\|R\|_{L^\9}^2\|\na\chi_{A,k}\|_{\dot{H}^\half}\\
&\quad+\|R\|_{L^\9}^3\|\na\chi_{A,k}\|_{\dot{H}^\half}
+\|\na\chi_{A,k}\|_{L^\9}\|R\|_{L^\9}^2\|R\|_{\dot{H}^\half})\\
&\leq C(A)|t|^{3-\kappa}\(|t|^{-1}\(\ln(2+\|R\|_{H^\half}^{-1})\)^{\half}
+|t|^{-2}\ln(2+\|R\|_{H^\half}^{-1})\)
\|R\|_{H^\half}^2\\
&\leq C(A)\(\ln(2+\|R\|_{H^\half}^{-1})\)^{\half}X(t).
\ea\ee
Thus, it follows from \eqref{est-l-6}, \eqref{est-l-7} and \eqref{est-l-8} that
the higher order terms can be bounded by
\be\ba\label{est-l-5}
&|{\rm Re}
       \<\frac{b_{k}}{2\lbb_k}\Delta\chi_A(\frac{x-\alpha_{k}}{\lambda_{k}})R_k+\frac{b_{k}}{2} \na\chi_A(\frac{x-\alpha_{k}}{\lambda_{k}})( \nabla R_k + \na R\Phi_k), 2U|R|^2+\ol{U}R^2+|R|^2R \>|\\
&\leq C(A)\(\ln(2+\|R\|_{H^\half}^{-1})\)^{\half}X(t).
\ea\ee

Consequently,
inserting \eqref{est-l-4} and \eqref{est-l-5} into \eqref{est-l-9}
we arrive at
\be\ba\label{est-l-10}
-(\mathfrak{L}_{34.3} + \mathfrak{L}_{34.4})
=-b_{k}{\rm Re}\< \na\chi_A(\frac{x-\alpha_{k}}{\lambda_{k}})\na U_k,2U_k|R_k|^2+\ol{U}_kR_k^2 \>
+\calo_A\(\(\ln(2+\|R\|_{H^\half}^{-1})\)^{\half}X(t)\).
\ea\ee

{\it Estimate of $\mathfrak{L}_{34.5} + \mathfrak{L}_{34.6}$.}
We estimate the last two terms in \eqref{est-l3l4} involving the error $\eta$ .
Integration by parts yields that
\ben\ba
 -(\mathfrak{L}_{34.5} + \mathfrak{L}_{34.6})
 =-\frac{b_{k}}{2\lbb_k}{\rm Re}\<\Delta\chi_A(\frac{x-\alpha_{k}}{\lambda_{k}})R_k, \eta \>
-b_k{\rm Re}
       \< \na\chi_A(\frac{x-\alpha_{k}}{\lambda_{k}})R_k , \na\eta \>
       -\frac{b_{k}}{2}{\rm Re}
       \< \na\chi_A(\frac{x-\alpha_{k}}{\lambda_{k}})R\na\Phi_k, \eta \>,
\ea\enn
which, via the Cauchy-Schwartz inequality and Lemma \ref{Lem-P-ve-U},
can be bounded by
\be\ba\label{est-l-11}
&C(A)b_k(\lbb_k^{-1}\|R\|_{L^2} \|\eta\|_{L^2}+\|R\|_{L^2}\|\na \eta\|_{L^2}+\|R\|_{L^2}\|\eta\|_{L^2})\\
\leq& C(A)(\frac{\|R\|_{L^2}^2}{t^2}+\frac{\|R\|_{L^2}^2}{|t|^{3-\delta}}+|t|^{3-\delta})
\leq C(A)(\frac{\|R\|_{L^2}^2}{|t|^{3-\delta}}+|t|^{3-\delta}).
\ea\ee

Finally, plugging \eqref{clm-loc1}, \eqref{clm-loc2}, \eqref{est-l-10} and \eqref{est-l-11} into \eqref{est-l3l4}
we complete the proof of Proposition \ref{prop-I2t}.
\end{proof}

By virtue of Propositions \ref{prop-I1t} and \ref{prop-I2t},
we are now able to prove Proposition \ref{prop-I-mono}.

{\bf Proof of Proposition \ref{prop-I-mono}.}
Applying Propositions \ref{prop-I1t} and \ref{prop-I2t},
we are now able to decouple  the time derivative of $\mathfrak{I}$ into $K$  parts and derive
\be\ba\label{est-pf-mo-1}
\frac{d\mathfrak{I}}{dt}
\geq&\sum_{k=1}^{K}\left(\frac{b_k}{2\lambda_k^2} \|R_k\|_{L^2}^2
-\frac{b_k}{2\lbb_{k}}{\rm Re}\int 2|U_k|^2|R_k|^2+\ol{U}_k^2R_k^2dx\right.\\
&\quad\left.+\frac{b_{k}}{2\lbb_k}   {\rm Re}
       \< \Delta\chi_A(\frac{x-\alpha_{k}}{\lambda_{k}})R_k, DR_k \>+
       b_{k}   {\rm Re}
       \< \na\chi_A(\frac{x-\alpha_{k}}{\lambda_{k}})\nabla R_k, DR_k\>\right.\\
    & \quad\left.+b_{k}{\rm Re}\< (\na\chi_A(\frac{x-\alpha_{k}}{\lambda_{k}})
    -\frac{x-\a_k}{\lbb_k})\cdot\na U_k,2U_k|R_k|^2+\ol{U}_kR_k^2 \>
   \right)\\
&+\calo(C(A)\(\ln(2+\|R\|_{H^\half}^{-1})\)^{\half}X(t)
+\frac{\|R\|_{L^2}^2}{|t|^{3-\delta}}+|t|^{3-\delta}).
\ea\ee
Then each part can be treated separately by using similar
arguments as in  the single-bubble case \cite{K-L-R}.

To control the terms in the second line of \eqref{est-pf-mo-1},  we follow the functional calculus argument as in the proof of  \cite[Lemma 6.1]{K-L-R}. In fact, for $1\leq k\leq K$, we can get that
\be\ba\label{est-l3-l4}
&\frac{b_{k}}{2\lbb_k}   {\rm Re}
       \< \Delta\chi_A(\frac{x-\alpha_{k}}{\lambda_{k}})R_k, DR_k \>+
       b_{k}   {\rm Re}
       \< \na\chi_A(\frac{x-\alpha_{k}}{\lambda_{k}})\nabla R_k, DR_k\>\\
&=\frac{b_k}{2\lbb_k}\int_{s=0}^{+\9}\sqrt{s}\int\Delta \chi_A(\frac{x-\a_k}{\lbb_k})|\na R_{k,s}|^2dxds-\frac{b_k}{8\lbb_k^2}\int_{s=0}^{+\9}\sqrt{s}\int\Delta^2 \chi_A(\frac{x-\a_k}{\lbb_k})| R_{k,s}|^2dxds.
\ea\ee
Here
\ben
R_{k,s}(t,x):=\frac{1}{\sqrt{2\pi s}}\int e^{-\sqrt{s}|x-y|}R_k(t,y)dy
\enn
 solves the equation
\ben
-\Delta R_{k,s}+sR_{k,s}=\sqrt{\frac{2}{\pi}}R_k, \quad for \;s>0.
\enn

Then, inserting \eqref{est-l3-l4} into \eqref{est-pf-mo-1} we get
\be\ba\label{est-pf-mo-2}
  \frac{d\mathfrak{I}}{dt}
\geq &\sum_{k=1}^{K}\left(\frac{b_k}{2\lambda_k^2} \|R_k\|_{L^2}^2+\frac{b_k}{2\lbb_k}\int_{s=0}^{+\9}\sqrt{s}\int\Delta \chi(\frac{x-\a_k}{\lbb_k})|\na R_{k,s}|^2dxds\right.\\
&\quad\left.-\frac{b_k}{8\lbb_k^2}\int_{s=0}^{+\9}\sqrt{s}\int\Delta^2 \chi_A(\frac{x-\a_k}{\lbb_k})| R_{k,s}|^2dxds -\frac{b_k}{2\lbb_{k}}{\rm Re}\int 2|U_k|^2|R_k|^2+\ol{U}_k^2R_k^2dx\right.\\
&\quad\left.+b_k{\rm Re} \int (\na\chi_A(\frac{x-\alpha_{k}}{\lambda_{k}})-\frac{x-\a_k}{\lbb_k})\cdot \nabla \ol{U}_k(\ol{U}_kR_k^2+2U_k|R_k|^2)dx\right)\\
 &  +\calo(C(A)\(\ln(2+\|R\|_{H^\half}^{-1})\)^{\half}X(t)
+\frac{\|R\|_{L^2}^2}{|t|^{3-\delta}}+|t|^{3-\delta}).
  \ea\ee
Set
\be\label{ren2}
R_k :=\lbb_k^{-\half}\varepsilon_k(t, \frac{x-\a_k}{\lbb_k})e^{i\g_k}, \quad with \;\varepsilon_k=\varepsilon_{k,1}+i\varepsilon_{k,2},\; 1\leq k\leq K.
\ee
Recall the fact that  $Q_k(y)=Q(y)+\calo(|t|\<y\>^{-2})$.
Plugging the renormalization \eqref{ren2} into \eqref{est-pf-mo-2}, we  obtain
\be\ba\label{est-pf-mo-3}
  \frac{d\mathfrak{I}}{dt}
\geq &\sum_{k=1}^{K}\frac{b_k}{2\lambda_k^2} \( \|\varepsilon_k\|_{L^2}^2+\int_{s=0}^{+\9}\sqrt{s}\int
\chi^{\prime\prime}(\frac yA)|\na \varepsilon_{k,s}|^2dyds
-\int 3Q^2\varepsilon_{k,1}^2+Q^2\varepsilon_{k,2}^2dy\)\\
&+\sum_{k=1}^{K}\frac{b_k}{\lbb^2_k}({\rm Re} \int (A\chi^\prime(\frac yA)-y)\cdot \nabla \ol{Q}_k(\ol{Q}_k\varepsilon_k^2+2Q_k|\varepsilon_k|^2)dx-\frac{1}{8A^2}
\int_{s=0}^{+\9}\sqrt{s}\int \chi^{\prime\prime\prime\prime}(\frac yA)| \varepsilon_{k,s}|^2dyds )\\
 & +\calo(C(A)\(\ln(2+\|R(t)\|_{H^\half}^{-1})\)^{\half}X(t)
+\frac{\|R\|_{L^2}^2}{|t|^{3-\delta}}+|t|^{3-\delta}).
  \ea\ee

Note that, by the orthogonality conditions \eqref{ortho-cond-Rn-wn}, Lemmas \ref{lem-dep},
\ref{Lem-P-ve-U} and  \ref{lem-mass-local},
the unstable directions of $\varepsilon_k$, $1\leq k\leq K$, can be controlled by
\ben
Scal(\varepsilon_k)=\calo(t^2\|R\|^2_{L^2}+|t|^{8-2\delta}),
\enn
where $Scal(\varepsilon_k)$ is defined as in \eqref{Scal-def} below.

We also recall from \cite[Proposition B.1]{K-L-R} the coercivity estimate
that, there exits $C>0$ such that for $A$ large enough, $1\leq k\leq K$,
\begin{align}\label{est-pf-mo-4}
&\|\varepsilon_k\|_{L^2}^2+\int_{s=0}^{+\9}\sqrt{s}\int
\chi^{\prime\prime}(\frac yA)|\na \varepsilon_{k,s}|^2dyds
-\int 3Q^2\varepsilon_{k,1}^2+Q^2\varepsilon_{k,2}^2dy  \nonumber \\
\geq& C \|\ve_k\|_{L^2}^2-\frac{1}{C}Scal(\ve_k)
\geq  C\|R_k\|_{L^2}^2-\frac{1}{C}|t|^{8-2\delta}.
\end{align}

Moreover, by the definition of the cut-off function $\chi$,
$A\chi^\prime(\frac yA)-y=0$ for $|y|\leq A$.
Then, it follows from the decay property $Q_k(y)\leq C\<y\>^{-2}$  that
\be\label{est-pf-mo-5}
|\int (A\chi^\prime(\frac yA)-y)\cdot \nabla \ol{Q}_k(\ol{Q}_k\varepsilon_k^2+2Q_k|\varepsilon_k|^2)dy|
\leq C \sup\limits_{y\geq A} \frac{A+|y|}{|y|^2}\|\varepsilon_k\|_{L^2}^2\leq \frac{C}{A}\|R\|_{L^2}^2.
\ee
In view of \cite[Lemma B.3]{K-L-R}, one also has
\be\label{est-pf-mo-6}
|\frac{1}{8A^2}
\int_{s=0}^{+\9}\sqrt{s}\int_{\R} \chi^{\prime\prime\prime\prime}(\frac yA)| \varepsilon_{k,s}|^2dyds|\leq \frac{C}{A}\|R\|_{L^2}^2.
\ee

Therefore, inserting \eqref{est-pf-mo-4}, \eqref{est-pf-mo-5} and \eqref{est-pf-mo-6} into \eqref{est-pf-mo-3}
we obtain that there exists $C>0$ such that for $A$ large enough,
\ben\ba
  \frac{d\mathfrak{I}}{dt}
\geq C\frac{\|R\|_{L^2}^2}{|t|^3} +\calo\(C(A)\(\ln(2+\|R(t)\|_{H^\half}^{-1})\)^{\half}X(t)+|t|^{3-\delta}\),
  \ea\enn
thereby finishing the proof of Proposition \ref{prop-I-mono}.
\hfill $\square$

\subsection{Coercivity of generalized energy functional} \label{Subsec-Coer-Gen-Energy}
In this subsection, we further prove the coercivity type control of the generalized energy functional,
which is the content of Proposition \ref{lem-est-ge} below.

\begin{proposition} [Coercivity type control of the generalized energy]  \label{lem-est-ge}
For all $t\in [T_*, T]$, there exist $C_1, C_2>0$ such that for $A$ sufficiently large,
\begin{align}   \label{esti-X-I}
C_1X(t)-\frac{1}{C_1}t^{6-2\delta}\leq \mathfrak{I}(t)\leq C_2X(t).
\end{align}
\end{proposition}

\begin{proof}
Using Lemmas \ref{lem-inter} and \ref{lem-dep} and the Sobolev embeddings
we expand
\be\ba\label{est-en}
\mathfrak{I}(t)&=\frac{1}{2}{\rm Re}\int |D^\half R|^2+\sum_{k=1}^K\frac{1}{\lambda_{k}} |R|^2 \Phi_k
-2|U|^2|R|^2-U^2\ol{R}^2dx+ o(X(t))\\
&=\frac{1}{2}{\rm Re}\int |D^\half R|^2+\sum_{k=1}^K\frac{1}{\lambda_{k}} |R|^2 \Phi_k
-\sum_{k=1}^{K}(2|U_k|^2|R|^2-U_k^2\ol{R}^2)dx+ o(X(t)).
\ea\ee
Using the fact that $\|U(t)\|_{L^\9}\leq Ct^{-2}$,
we immediately obtain the upper bound that for some $C_2>0$,
\ben
\mathfrak{I}(t)\leq C_2X(t).
\enn

It remains to prove the lower bound
that is, the coercivity estimate in \eqref{esti-X-I}.
Using the partition of unity $\sum_{k=1}^{K}\Phi_k=1$
we rewrite the second order terms in \eqref{est-en} as
\be\ba\label{est-en1}
&\sum_{k=1}^{K}{\rm Re}\int (|D^\half R|^2+\frac{1}{\lambda_{k}} |R|^2) \Phi_k
-2|U_k|^2|R|^2-U_k^2\ol{R}^2dx\\
=&\sum_{k=1}^{K}{\rm Re}\int (|D^\half R|^2+\frac{1}{\lambda_{k}} |R|^2) \phi_{A,k}
-2|U_k|^2|R|^2-U_k^2\ol{R}^2dx\\
&+\sum_{k=1}^{K}\int (|D^\half R|^2+\frac{1}{\lambda_{k}} |R|^2) (\Phi_k-\phi_{A,k})dx
=:K_1+K_2,
\ea\ee
where $\phi_{A,k}(x):=\phi_A(\frac{x-\a_k}{\lbb_k})$
with the cut-off function $\phi_A$ defined as in  Lemma \ref{Lem-coer-local}.

On one hand,  using $Q_k(y)=Q(y)+\calo(|t|\langle y\rangle^{-2})$, \eqref{Qk-approx}
and the renormalization \eqref{ren}
we derive
\ben\ba
K_1&=\sum_{k=1}^{K}\frac{1}{\lbb_k}{\rm Re}\int (|D^\half \epsilon_k|^2+|\epsilon_k|^2) \phi_{A}
-2|Q_k|^2|\epsilon_{k}|^2-Q_k^2\ol{\epsilon}_{k}^2dy\\
&=\sum_{k=1}^{K}\frac{1}{\lbb_k}\int (|D^\half \epsilon_k|^2+|\epsilon_k|^2) \phi_{A}
-3Q^2\epsilon_{k,1}^2-Q^2\epsilon_{k,2}^2dy+\calo(|t|^{-1}\|R\|^2_{L^2}).
\ea\enn
Moreover,  the orthogonality conditions \eqref{ortho-cond-Rn-wn},
Lemmas \ref{Lem-P-ve-U} and \ref{lem-mass-local} imply that
the unstable directions of $\epsilon_k$ can be controlled by
\ben
Scal(\epsilon_k)=\calo(t^2\|R\|^2_{L^2}+|t|^{8-2\delta}).
\enn
Thus, applying Lemmas \ref{Lem-coer-local} and  \ref{Lem-P-ve-U}
we infer that there exists $C_1>0$ such that for $A$ sufficiently large,
\be\ba\label{est-en-i}
K_1 &\geq \sum_{k=1}^{K}\frac{C_1}{\lbb_k}\int (|D^\half \epsilon_k|^2+|\epsilon_k|^2) \phi_{A}dy-\frac{1}{C_1}|t|^{6-2\delta}+o(X(t))\\
&\geq C_1\sum_{k=1}^{K}\int (|D^\half R|^2+\frac{1}{\lambda_{k}} |R|^2) \phi_{A,k}dx-\frac{1}{C_1}|t|^{6-2\delta}+o(X(t)).
\ea\ee

On the other hand, let $\wt{C}:=\min\{1, C_1\}$.
For $t$ close to 0 such that $|\a_k(t)-x_k|\leq \sigma$,
we derive that, after the renormalization,
\be\ba\label{est-en-ii}
K_2=&\sum_{k=1}^{K}\frac{1}{\lbb_k}\int (|D^\half \epsilon_k|^2+|\epsilon_k|^2) (\Phi_k(\lbb_k y+\a_k)-\phi_{A}(y))dy\\
\geq& \sum_{k=1}^{K}\frac{\wt C}{\lbb_k}\int_{|y|\leq \frac{\sigma}{\lbb_k}} (|D^\half \epsilon_k|^2+|\epsilon_k|^2) (\Phi_k(\lbb_k y+\a_k)-\phi_{A}(y))dy\\
&+\sum_{k=1}^{K}\frac{1}{\lbb_k}\int_{|y|\geq \frac{\sigma}{\lbb_k}} (|D^\half \epsilon_k|^2+|\epsilon_k|^2) (\Phi_k(\lbb_k y+\a_k)-\phi_{A}(y))dy,
\ea\ee
where we also used the fact that $\Phi_k(\lbb_k y+\a_k)-\phi_{A}(y)\geq 0$ for $|y|\leq \frac{\sigma}{\lbb_k}$, $1\leq k\leq K$.
Note that
\ben\ba
&\sum_{k=1}^{K}\frac{1}{\lbb_k}\int_{|y|\geq \frac{\sigma}{\lbb_k}} (|D^\half \epsilon_k|^2+|\epsilon_k|^2) (\Phi_k(\lbb_k y+\a_k)-\phi_{A}(y))dy\\
&\geq\sum_{k=1}^{K}\frac{\wt C}{\lbb_k}\int_{|y|\geq \frac{\sigma}{\lbb_k}} (|D^\half \epsilon_k|^2+|\epsilon_k|^2) (\Phi_k(\lbb_k y+\a_k)-\phi_{A}(y))dy
 -\sum_{k=1}^{K}\frac{1-\wt C}{\lbb_k}\int_{|y|\geq \frac{\sigma}{\lbb_k}} (|D^\half \epsilon_k|^2+|\epsilon_k|^2) \phi_{A}(y)dy.
\ea\enn
Since by the definition of $\phi_A$,
$|\phi_{A}(y)|\leq C\lbb_k^{a}$ for $|y|\geq \frac{\sigma}{\lbb_k}$ with $0<a<1$,
it follows that
\be\ba\label{est-en-ii1}
&\sum_{k=1}^{K}\frac{1}{\lbb_k}\int_{|y|\geq \frac{\sigma}{\lbb_k}} (|D^\half \epsilon_k|^2+|\epsilon_k|^2) (\Phi_k(\lbb_k y+\a_k)-\phi_{A}(y))dy\\
&\geq\sum_{k=1}^{K}\frac{\wt C}{\lbb_k}\int_{|y|\geq \frac{\sigma}{\lbb_k}} (|D^\half \epsilon_k|^2+|\epsilon_k|^2) (\Phi_k(\lbb_k y+\a_k)-\phi_{A}(y))dy- C\sum_{k=1}^{K}\frac{\lbb_k^a}{\lbb_k}\int (|D^\half \epsilon_k|^2+|\epsilon_k|^2) dy\\
&\geq\sum_{k=1}^{K}\frac{\wt C}{\lbb_k}\int_{|y|\geq \frac{\sigma}{\lbb_k}} (|D^\half \epsilon_k|^2+|\epsilon_k|^2) (\Phi_k(\lbb_k y+\a_k)-\phi_{A}(y))dy
+o(X(t)).
\ea\ee
Thus,
inserting \eqref{est-en-ii1} into \eqref{est-en-ii} we obtain
\be\label{est-en-ii2}
K_2\geq \sum_{k=1}^{K}\wt C\int (|D^\half R|^2+\frac{1}{\lbb_k}|R|^2) (\Phi_k-\phi_{A,k})dx
+o(X(t)).
\ee

Therefore, plugging \eqref{est-en-i} and \eqref{est-en-ii2} into \eqref{est-en1}
we obtain the lower bound
\ben
\mathfrak{I}(t)\geq C\int |D^\half R|^2+\sum_{k=1}^{K}\frac{1}{\lambda_{k}} |R|^2 \Phi_kdx- \frac{1}{C}|t|^{6-2\delta}+o(X(t))
\geq CX(t)-\frac{1}{C}|t|^{6-2\delta},
\enn
and thus finish the proof.
\end{proof}

\section{Proof of bootstrap estimates} \label{Sec-Proof-Boot}
Now we are ready to prove Theorem \ref{Thm-u-Boot}.

\begin{proof}[Proof of Theorem \ref{Thm-u-Boot}]
{\it $(i)$ Bootstrap estimate of $X(t)$.}
Applying Propositions \ref{prop-I-mono} and \ref{lem-est-ge},
and using $\mathfrak{I}(T)=0$
we obtain that for all $t\in[T_*,T]$,
\be\label{est-pf-b-1}
C_1X(t)-\frac{1}{C_1}|t|^{6-2\delta}\leq \mathfrak{I}(T)-\int_{t}^{T}\frac{d\mathfrak{I}}{ds}ds
\leq \int_{t}^{T}C(A)\(\ln(2+\|R(s)\|_{H^\half}^{-1})\)^{\half}X(s)+|s|^{3-\delta}ds.
\ee
For  $s$ close to 0, we have $\|R\|_{H^\half}\geq \half\lbb X^\half$ and $\ln X^{-1}(s)\leq X^{-p}(s)$ for  $p$ satisfying $(1-p)(4-2\delta)=3-\delta$. Thus by  \eqref{R-Tt} and Lemma \ref{Lem-P-ve-U}, we have
\be\ba\label{est-pf-b-2}
C(A)\(\ln(2+\|R(s)\|_{H^\half}^{-1})\)^{\half}X(s)&\leq C(A)\(\ln \lbb^{-1}(s)+\ln X^{-1}(s)\)X(s)\\
&\leq C(A)\(s^{-\half} X(s)+X^{1-p}(s)\)\leq \frac{C_1(4-2\delta)}{8}|s|^{3-2\delta}.
\ea\ee
Inserting \eqref{est-pf-b-2} into \eqref{est-pf-b-1}
we then obtain there exists $t_0$ such that for all $t\in[T_*,T]$ with $T_*\geq t_0$,
\ben
X(t)\leq \frac{1}{8}|t|^{4-2\delta}+\frac{1}{C_1(4-\delta)}|t|^{4-\delta}\leq \frac{1}{4}|t|^{4-2\delta},
\enn
which implies that for all $t\in[T_*,T]$,
\ben
\|D^\half R(t)\|_{L^2}\leq \frac{1}{2}|t|^{2-\delta},\quad\|R(t)\|_{L^2}\leq \frac{1}{2}|t|^{3-\delta},
\enn
and so verifies the estimate of $X(t)$ in \eqref{wn-Tt-boot-2}.

{\it $(ii)$ Bootstrap estimate of $\|D^{\half+\varsigma}R\|_{L^2}$.}
First, we write the equation of $R$
\be\label{equa-F}
i\partial_tR=DR-|R|^2R-\eta-F,
\ee
with $\eta$ defined as in \eqref{etan-Rn} and
\ben
F : =|U+R|^2(U+R)-|U|^2U-|R|^2R.
\enn
Set $Y(t):=\|D^{\half+\varsigma}R\|_{L^2}^2$.
Using equation \eqref{equa-F} and the integration by parts formula
we obtain the equation
\ben
Y^\prime(t)=-2{\rm Im}\<D^{\half+\varsigma}(|R|^2R+\eta+F),D^{\half+\varsigma}R\>.
\enn
Using $U=\sum_{k=1}^{K}U_k$
we can adapt the estimates in \cite[Appdendix E]{K-L-R} from line to line to
the current multi-bubble case and obtain
\ben
|Y^\prime(t)|\leq C|t|^{2-2\delta-4\varsigma}
+C|t|^{-1}Y(t)\|R(t)\|_{H^\half}\ln^\half\(\frac{1+Y^\half(t)}{\|R(t)\|_{H^\half}}\),
\enn
where  the a prior bound  $X(t)\leq |t|^{4-2\delta}$ and Lemma \ref{Lem-P-ve-U} are applied.
Integrating from $t$ to $T$,
using the boundary condition  $Y(T)=0$, the a prior bounds $Y(t)\leq |t|^{2-2\delta-4\varsigma}$ and $\|R(t)\|_{H^\half}\leq |t|^{2-\delta}$,
we get that for $t$ close to 0,
\ben
Y(t)\leq \int_{t}^{0}|Y^\prime(s)|ds\leq C\int_{t}^{0}|s|^{2-2\delta-4\varsigma}+s^{-1}Y(s)\|R(s)\|^\half_{H^\half}ds\leq C|t|^{3-2\delta-4\varsigma}\leq \frac{1}{4}|t|^{2-2\delta-4\varsigma}.
\enn
This yields that
\ben
\|D^{\half+\varsigma}R\|_{L^2}\leq \half|t|^{1-\delta-2\varsigma},
\enn
and thus verifies the estimate of $\|D^{\half+\varsigma}R\|_{L^2}$ in \eqref{wn-Tt-boot-2}.

{\it $(iii)$ Bootstrap estimates of $\lambda_k$ and $b_k$.}
From Lemma \ref{Lem-P-ve-U} it follows that for all $t\in[T_*,T]$,
\begin{align*}
  |\frac{d}{dt}(\lambda^{-\half}_{k}b_k )|
=\lbb_k^{-\frac32}\(\lbb_k\dot{b}_k+\half b_k^2-\half(b_k\dot{\lbb}_k+b_k^2)b_k\)
\leq C|t|^{-3}Mod_k,
\end{align*}
which along with \eqref{Mod-w-lbb} and the fact  $(\lambda^{-\half}_{k}b_k)(T)=\omega$ yields that
\begin{align}  \label{gamlbb-1}
 | (\lambda^{-\half}_{k}b_k)(t)-\omega |
\leq\int_{t}^{T}   |\frac{d}{ds}(\lambda^{-\half}_{k}b_k ) |ds
\leq C\int_{t}^{0}|s|^{-3}Mod_k(s)ds
\leq C|t|^{2-\delta}.
\end{align}
Moreover, the straightforward calculation shows that
\begin{align*}
|\frac{d}{dt}(\lambda_{k}^\half(t) +\frac{\omega}{2} t)|
=\half|\lbb^{-\half}_k(\dot{\lbb}_k+b_k)-(\lbb_k^{-\half} b_k-\omega)|
\leq C|t|^{2-\delta}.
\end{align*}
Taking into account  $\lambda_{k}^\half(T)=-\frac{\omega}{2} T$,
due to \eqref{PjT},
we get
\ben
|\lambda_{k}^\half(t) +\frac{\omega}{2} t|
\leq\int_{t}^{T}  |\frac{d}{ds}(\lambda_{k}^\half(s) +\frac{\omega}{2} s)|dr
\leq C\int_{t}^{0}|s|^{2-\delta}ds\leq C|t|^{3-\delta}.
\enn
This leads to
\be\label{lbb-Tt*}
|\lbb_k(t)-\frac{\omega^2}{4}t^2|\leq |\lambda_{k}^\half(t) +\frac{\omega}{2} t||\lambda_{k}^\half(t) -\frac{\omega}{2} t|\leq C|t|^{4-\delta}\leq \half |t|^{4-2\delta},
\ee
thereby verifying the estimate of $\lbb_k$ in \eqref{lbbn-Tt2}.

Similarly,
by \eqref{Mod-w-lbb} and \eqref{gamlbb-1},
\begin{align*}
   |\frac{d}{dt}(b_k(t)+\frac{\omega^2}{2}t)|
   =  |\lbb_k^{-1}(\lbb_k\dot{b}_k+\frac{1}{2}b_k^2)-\half(\lbb_k^{-1}b_k^2-\omega^2)|
   \leq  C|t|^{2-\delta},
\end{align*}
which along with $b_k(T) = -\frac12\omega^2T$ yields that for $t$ close to 0,
\begin{align*}
   |b_k(t)+\frac{\omega^2}{2}t|
   \leq \int_t^{T}  |\frac{d}{ds} (b_k(s)+\frac{\omega^2}{2}s)| ds
   \leq C\int_{t}^{0}|s|^{2-\delta}ds\leq C|t|^{3-\delta}\leq \half|t|^{3-2\delta}.
\end{align*}
Thus the estimate of $b_k$ in \eqref{lbbn-Tt2} is verified.

{\it $(iv)$ Bootstrap estimates of $v_k$ and $\alpha_k$.}
From \eqref{lbb-Tt*}, Lemma \ref{Lem-P-ve-U} and the estimate of $v_k$ in Lemma \ref{lem-lcomom}
it follows that for $t$ close to 0,
\begin{align}\label{v-Tt*}
|v_k-\frac{\omega^2}{4}t^2|
\leq \lbb_k|\frac{v_k}{\lbb_k}-1|+|\lbb_k-\frac{\omega^2}{4}t^2|
\leq C|t|^{4-\delta}\leq \frac{1}{2}|t|^{4-2\delta},
\end{align}
which, via $\alpha_k(T)=x_k$, gives the estimate of $v_k$ in \eqref{vn-Tt12}.

Moreover, using \eqref{Mod-w-lbb} and \eqref{v-Tt*} and the fact that   $\alpha_k(T)=x_k$
we get
\ben
|\dot{\alpha}_{k}|=
|\dot{\a}_k-v_k+v_k|
\leq Mod_k(t)+|v_k|\leq Mod_k(t)+Ct^2,
\enn
which yields that for $t$ close to 0,
\ben
|\alpha_{k}(t)-x_k|\leq\int_{t}^{T}|\dot{\alpha}_{k}(s)|ds
\leq \half |t|^{3-\delta},
\enn
thereby verifying the estimate of $\a_k$ in \eqref{vn-Tt12}.

{\it $(iv)$ Bootstrap estimate of $\g_k$.}
Using \eqref{lbb-Tt*} and Lemma \ref{Lem-P-ve-U}
we compute
\begin{align*}
   |\frac{d}{dt}(\g_{k}(t) + \frac{4}{\omega^2 t} - \theta_k)|
= |\lbb_k^{-1}(\lambda_{k}\dot{\g}_{k}-1)+\frac{\omega^2 t^2-4\lbb_k}{\omega^2 t^2\lbb_k}|
\leq C|t|^{-\delta},
\end{align*}
which along with $\g_{k}(T) = -\frac{4}{\omega^2 T} + \theta_k$ yields that for $t$  close to 0,
\ben
 |\g_{k}(t) + \frac{4}{\omega^2 t} - \theta_k|
 \leq\int_{t}^{T} | \frac{d}{ds}(\g_{k}(t) + \frac{4}{\omega^2 t} - \theta_k) |ds
 \leq C\int_{t}^{0}|s|^{-\delta}ds
 \leq C|t|^{1-\delta}
\leq \frac{1}{2} |t|^{1-2\delta},
\enn
thereby we verify the estimate of $\g_k$ in  \eqref{thetan-Tt12}.
Therefore, the proof of Theorem \ref{Thm-u-Boot} is complete.
\end{proof}

\section{Existence of multi-bubble solutions} \label{Sec-Exist-Multi}

Let  $\{t_n\}$ be an increasing sequence of times converging to 0.
For every $1\leq k\leq K$, let
\begin{align} \label{PjTn}
   (\lbb_{n,k}, b_{n,k}, v_{n,k}, \a_{n,k}, \g_{n,k})(t_n)
   :=(\frac{\omega^2}{4} t_n^2, -\frac{\omega^2}{2} t_n, \frac{\omega^2}{4} t_n^2, x_k,  -\frac{4}{\omega^2 t_n} + \t_k),
\end{align}
and $Q_{n,k}(t_n,x)$ be the approximate profiles defined  in Lemma \ref{lem-app}
with $b_{n,k}(t_n), v_{n,k}(t_n)$ replacing $b_k, v_k$, respectively.
Consider the approximate solutions $u_n$ solving the following  equation
\be    \label{equa-u-t}
\left\{ \begin{aligned}
 &i\partial_t u_n=D u_n-|u_n|^{2}u_n,   \\
 &u_n(t_n)=\sum_{j=1}^{K}\lbb_{n,k}^{-\frac 12}(t_n) Q_{n,k}\(t_n,\frac{x-\a_{n,k}(t_n)}{\lbb_{n,k}(t_n)}\) e^{i\g_{n,k}(t_n)}.
\end{aligned}\right.
\ee

We have the following uniform estimates as a consequence of the bootstrap estimates in Theorem \ref{Thm-u-Boot}.

\begin{theorem}[Uniform estimates]\label{thm-ue}
Let $\delta, \varsigma\in(0, \frac12)$ and $\delta+2\varsigma<1$.
There exists a uniform backwards time $t_0<0$ such that, for n large enough,
$u_n$ admits the geometrical decomposition $u_n = U_n + R_n$ on $[t_0, t_n]$ as in  Proposition \ref{Prop-dec-un}
with the main blow-up profile given by
\begin{align*}
    U_n(t,x)
    =   \sum_{k=1}^{K} U_{n,k}(t,x), \quad \quad
   U_{n,k}(t,x) = \lbb_{n,k}^{-\frac 12}(t) Q_{n,k}\(t,\frac{x-\a_{n,k}(t)}{\lbb_{n,k}(t)}\) e^{i\g_{n,k}(t)}.
\end{align*}
Moreover, the reminder  and the modulation parameters satisfy that for $t\in[t_0, t_n]$,
\begin{align}
&\|D^\half R_n(t)\|_{L^2}\leq |t|^{2-\delta},\quad\|R_n(t)\|_{L^2}\leq |t|^{3-\delta}, \quad \|D^{\half+\varsigma}R_n(t)\|_{L^2}\leq |t|^{1-\delta-2\varsigma}\label{R-Ttn}\\
&|\la_{n,k}(t) - \frac{\omega^2}{4} t^2 |\leq |t|^{4-2\delta},\quad |b_{n,k}(t)  + \frac{\omega^2}{2} t |\leq |t|^{3-2\delta}, \quad|\al_{n,k}(t)-x_k|\leq |t|^{3-\delta},\label{lba-Ttn}\\
& |v_{n,k}(t)-\frac{\omega^2}{4} t^2|\leq |t|^{4-2\delta},  \quad
|\g_{n,k}(t) + \frac{4}{\omega^2 t} - \t_k| \leq |t|^{1-2\delta}, \quad 1\le k\leq K. \label{thetan-Ttn}
\end{align}.
\end{theorem}

\begin{proof}
In view of Theorem \ref{Thm-u-Boot} and the boundary condition \eqref{PjTn},
estimates \eqref{R-Ttn}-\eqref{thetan-Ttn} follow from a standard continuity argument.
The relevant arguments can be found, for instance, in \cite{S-Z}.
For the sake of simplicity, the details are omitted here.
\end{proof}

Now, we are in a position to prove the main result in Theorem \ref{thm-main},
that is, to construct multi-bubble blow-up solutions to the half-wave equation.

{\bf Proof of Theorem \ref{thm-main}.}
By virtue of Theorem \ref{thm-ue},
$\{u_n(t_0)\}$ are uniformly bounded in $H^{\half+\varsigma}(\R)$,
and thus there exists   $u_0 \in H^{\half+\varsigma}(\R)$ such that for any $s\in[0,\half+\varsigma]$,
 \begin{align} \label{unt0-u0-Hs}
    u_n(t_0) \rightharpoonup u_0,\quad weakly\; in\; H^s(\R),\quad as\ n \to \9.
\end{align}

We claim that the convergence holds strongly in the $L^2$ space, i.e.,
\begin{align} \label{unt0-u0-L2}
    u_n(t_0) \to u_0,\quad strongly\;  in\; L^2(\R),\quad as\ n \to \9.
\end{align}
To this end,
let $\varphi: \R\rightarrow\R$ be a smooth nonnegative cut-off function
such that $\varphi(x)=0$ for $|x|\leq 1$ and $\varphi(x)\equiv1$ for $|x|\geq 2$.
Set $\varphi_A(x):=\varphi(\frac{x}{A})$ for $A>0$.
On one hand, the boundary condition $
u_n(t_n)=\sum_{k=1}^{K}U_{n,k}(t_n)$ and Lemma \ref{lem-dep} imply that
for $A\geq 2\max_{1\leq k\leq K}|x_k|$,
\begin{align}\label{est-sec6}
|\int |u_n(t_n)|^2\varphi_A dx|\leq \frac{C}{A^3}t_n^6.
\end{align}
On the other hand, it follows from equation \eqref{equa-u-t} and
the integration by parts formula that
\ben
\frac{d}{dt}\int|u_n|^2\varphi_Adx=2{\rm Im}\int \ol{u}_n D u_n\varphi_A dx
={\rm Im}\int u_n(D(\ol{u}_n\varphi_A)-D \ol{u}_n\varphi_A)dx.
\enn
By the Calder\'on estimate in Lemma \ref{lem-l2-pro2},
the conservation law of mass and the fact that $\|\nabla \varphi_A\|_{L^\9}=\frac{1}{A}\|\nabla \varphi\|_{L^\9}$,
\ben
|\frac{d}{dt}\int|u_n|^2\varphi_Adx|\leq C\|\nabla \varphi_A\|_{L^\9}\|u_n\|_{L^2}^2
\leq \frac{C}{A}\|\nabla \varphi\|_{L^\9}\|u_n\|_{L^2}^2\leq \frac{C}{A}.
\enn
Hence, integrating from $t_0$ to $t_n$ and using \eqref{est-sec6}
we infer that there exists $C>0$ independent of $n$ such that
\ben
|\int |u_n(t_0)|^2\varphi_A dx|\leq \frac{C}{A},
\enn
which yields the uniform integrability
\be\label{est-main-1}
 \sup\limits_{n\geq 1} \|u_n(t_0)\|_{L^2(|x|\geq A)}\rightarrow0,\quad as\ A\rightarrow\9.
\ee
Thus, the weak convergence in $L^2$ along with  \eqref{est-main-1} implies \eqref{unt0-u0-L2}, as claimed.

Combining \eqref{unt0-u0-Hs} and \eqref{unt0-u0-L2} together
and using interpolations
we also have that for any $s\in[0,\half+\varsigma)$,
 \begin{align} \label{unt0-u0t0-Hs}
    u_n(t_0) \rightarrow u_0,\quad strongly\; in\; H^s(\R),\quad as\ n \to \9.
\end{align}

By virtue of \eqref{unt0-u0t0-Hs}
and
the  local well-poseness theory,
we then obtain a unique $H^\half$-solution $u$ to \eqref{equa-u-t} on $[t_0, 0)$
satisfying that $u(t_0)=u_0$,
and
\begin{align}\label{un-u-0-L2}
\lim_{n\rightarrow \infty}\|u_n-u\|_{C([t_0,t];H^\half(\R))}=0,\ \ t\in [t_0, 0).
\end{align}

Moreover, for each $t_0\leq t<0$ fixed,
the modulation estimate \eqref{Mod-bdd} and Theorem \ref{thm-ue}
yield that the derivatives of  parameters $\dot \calp_n$ are uniformly bounded
on $[t_0,t]$,
and thus $\calp_n$ are equi-continuous on $[t_0,t]$, $n\geq 1$.
Then, by  the Arzel\`{a}-Ascoli Theorem,
$\calp_n$ converges uniformly on $[t_0,t]$ up to some subsequence
(which may depend on $t$).
But, using the diagonal arguments
one may extract a universal subsequence (still denoted by $\{n\}$)
such that
for some $\calp:= (\calp_1, \cdots, \calp_K)$,
where
$\mathcal{P}_k:=(\lbb_k, b_k, v_k, \a_k, \g_k) \in C([t_0,t]; \R^{5})$,
$1\leq k\leq K$,
and for every $t\in [t_0,0)$,
one has
\begin{align} \label{Pn-P}
   \mathcal{P}_n \rightarrow\mathcal{P}\ \ in\ C([t_0,t]; \R^{5K}).
\end{align}
Then, taking into account the uniform estimates \eqref{lba-Ttn}-\eqref{thetan-Ttn}
we obtain that
for all $t\in[t_0,0)$,
\be\ba\label{est-par-sol}
&|\la_{k}(t) - \frac{\omega^2}{4} t^2 |\leq |t|^{4-2\delta},\quad |b_{k}(t)  + \frac{\omega^2}{2} t |\leq |t|^{3-2\delta}, \quad|\al_{k}(t)-x_k|\leq |t|^{3-\delta},\\
& |v_{k}(t)-\frac{\omega^2}{4} t^2|\leq |t|^{4-2\delta},  \quad
|\g_{k}(t) + \frac{4}{\omega^2 t} - \t_k| \leq |t|^{1-2\delta}, \quad 1\leq k\leq K.
\ea\ee

Let
\begin{align*}
    U(t,x)
    :=   \sum_{k=1}^{K} U_{k}(t,x), \quad with\quad
   U_{k}(t,x) := \lbb_{k}^{-\frac 12}(t) Q_{k}\(t,\frac{x-\a_{k}(t)}{\lbb_{k}(t)}\) e^{i\g_{k}(t)}.
\end{align*}
It follows from the inequality
\ben
\|u(t)-U(t)\|_{H^\half}\leq \|u(t)-u_n(t)\|_{H^\half}+\|U(t)-U_n(t)\|_{H^\half}+\|R_n(t)\|_{H^\half},
\enn
the uniform estimates \eqref{R-Ttn} and
the convergence in \eqref{un-u-0-L2} and \eqref{Pn-P} that
\ben
\lim_{t\rightarrow0^{-}}\|u(t)-U(t)\|_{H^\half}=0.
\enn
Taking into account the approximation of  $Q_k$ to $Q$ in \eqref{Q-asy},
the estimates of parameters in \eqref{est-par-sol} and Lemma \ref{lem-dep}
we conclude that $u(t)$ blows up at $t=0$
and satisfies
\ben\ba
&u(t)-\sum_{k=1}^{K}\lbb_{k}^{-\frac 12}(t) Q\(t,\frac{x-\a_{k}(t)}{\lbb_{k}(t)}\)e^{i\g_k(t)}\rightarrow 0\quad in \quad L^2(\R),\quad as \quad t\rightarrow 0^-.
\ea\enn
In particular, one has
\ben\ba
&\|u(t)\|_{L^2}^2=\lim_{t\rightarrow 0^-}\|U(t)\|_{L^2}^2=K\|Q\|_{L^2}^2\quad and\quad u^2(t)\rightharpoonup\sum_{k=1}^{K}\|Q\|^2_{L^2}\delta_{x=x_k},\quad as \quad t\rightarrow 0^-.
\ea\enn

Therefore,
the proof of Theorem \ref{thm-main} is complete.

%\begin{appendices}
\appendix

\section{Fractional Laplacian operators}

In this appendix we will recall some fundamental properties of fractional Laplacian operators.
Let us start with the following lemma
which provides the pointwise characterization of the fractional Laplacian operator defined as in \eqref{def-ds}.

\begin{lemma}[\cite{D-P-V}]\label{lem-fra}
Let $s\in(0,1)$.
Then, for any $f\in \mathscr{S}(\R)$,
\be\ba
(-\Delta)^sf(x)&=C(s) P.V.\int_{\R}\frac{f(x+y)-f(x)}{|y|^{1+2s}}dy\\
&=-\frac{1}{2}C(s)\int_{\R}\frac{f(x+y)+f(x-y)-2f(x)}{|y|^{1+2s}}dy,\label{id-fra}
\ea\ee
where the normalization constant is given by
\be\label{constant}
C(s)=\left(\int_{\R}\frac{1-cosx}{|x|^{1+2s}}dx\right)^{-1}.
\ee
\end{lemma}

The following lemma provides useful formulas for the fractional
Laplacian of the product of two functions.
\begin{lemma}[\cite{BWZ}]\label{lem-formula}
Let $s\in(0,1)$ and $f, g\in \mathscr{S}(\R)$.
Then, we have
\be\label{id-fra2}
(-\Delta)^{s}(fg)-(-\Delta)^sfg
=C(s)\int\frac{f(x+y)(g(x+y)-g(x))-f(x-y)(g(x)-g(x-y))}{|y|^{1+2s}}dy,
\ee
and
\ben
(-\Delta)^{s}(fg)-(-\Delta)^sfg-f(-\Delta)^sg
=-C(s)\int\frac{(f(x+y)-f(x))(g(x+y)-g(x))}{|y|^{1+2s}}dy,
\enn
where $C(s)$ is defined as in \eqref{constant}.
\end{lemma}

Lemma \ref{Lem-GN} below contains the standard Sobolev embeddings and interpolations
that will be used to control the high order terms in the expansion of nonlinearity.

\begin{lemma} \label{Lem-GN}
(i) For $0<s<\half$,
\ben
\dot{H}^{s}(\R)\hookrightarrow L^{\frac{2}{1-2s}}(\R).
\enn
(ii) For $s=\half$,
\ben
H^{\half}(\R)\hookrightarrow L^p(\R),\quad \forall p\in[2,+\infty).
\enn
(iii) For $s>\half$,
\ben
H^{s}(\R)\hookrightarrow L^p(\R),\quad \forall p\in[2,+\infty].
\enn
(iv) For $s_1\leq s\leq s_2$, the following interpolation estimate holds
\ben
\|u\|_{\dot{H}^s(\R)}\leq \|u\|^{1-\t}_{\dot{H}^{s_1}(\R)}\|u\|^\t_{\dot{H}^{s_2}(\R)}\quad with\quad s=(1-\t)s_1+\t s_2.
\enn
(v) For $s>\half$ and $f\in H^s(\R)$,  there exists $C>0$ depending on $s$ such that
\be\label{est-hs}
\|f\|_{L^\9}\leq C\|f\|_{H^\half}\(\ln\(2+\frac{\|f\|_{H^s}}{\|f\|_{H^\half}}\)\)^{\half}.
\ee
\end{lemma}

\begin{proof}
The Sobolev embeddings and interpolation estimates $(i)-(iv)$ can be found, e.g., in \cite{BCD}.
See \cite[Appendix D]{K-L-R} for the proof of $(v)$.
\end{proof}

We also use the following fractional chain rule.
\begin{lemma}(\cite{CM91})  \label{Lem-Pro-rule}
Let $0<s<1$ and $1< p, p_1, p_2<\infty$ with $\frac{1}{p}=\frac{1}{p_1}+\frac{1}{p_2}$.
Then, for any $f\in C^{1}(\C, \C)$, the exists $C>0$ such that
for any $u\in \mathscr{S}(\R)$,
\ben
\|D^sf(u)\|_{L^p}\leq C\|f^{\prime}(u)\|_{L^{p_1}}\|D^su\|_{L^{p_2}}.
\enn
\end{lemma}

We recall the following commutator estimates.
\begin{lemma}(\cite[Appendiex E]{K-L-R})\label{lem-l2-pro}
For $0\leq s\leq 1$ and $f, g\in \mathscr{S}(\R)$,
we have
\ben
\|D^s(fg)-fD^sg\|_{L^2}\leq C\min\{\|D^sf\|_{L^2}\|\hat{g}\|_{L^1},\|\widehat{D^sf}\|_{L^1}\|g\|_{L^2}\}.
\enn
\end{lemma}
\begin{lemma}(\cite[Theorem 2]{C65})\label{lem-l2-pro2}
For $\nabla f\in L^\9(\R)$ and $g\in \mathscr{S}(\R)$,
we have
\ben
\|D(fg)-fDg\|_{L^2}\leq C\|\nabla f\|_{L^\9}\|g\|_{L^2}.
\enn
\end{lemma}

\begin{lemma}(\cite[Lemma F.1]{K-L-R})\label{lem-inter}
Suppose $d\geq 1$ and let $\phi: \R^d\rightarrow\R$ be such that $\na \phi$ and $\Delta \phi$ belong to $L^\9(\R^d)$.
Then, there exists some $C>0$ such that
\ben
|\<\na \phi\cdot \na u, u\>|\leq C(\|\na \phi\|_{L^\9}\|u\|^2_{\dot{H}^\half}+\|\Delta \phi\|_{L^\9}\|u\|^2_{L^2}).
\enn
\end{lemma}

The following decoupling estimate permits to control
the interactions between different bubbles,
as well as the interplay between bubbles and the localization functions
with supports away from the bubble center.

\begin{lemma} [Decoupling estimates] \label{lem-dep}
Assume that $f, g\in C^\9(\R)$ have  the decay estimate $|f(x)|+|g(x)|\leq \frac{C_1}{1+|x|^2}$ for some constant $C_1>0$.
Assume that $h\in C^\9(\R)$ with $0\leq h\leq1$, and $h(x)=0$ for $|x|\leq \sigma$ where $\sigma$ is a positive constant.
Then, there exists $C>0$ such that for  any small $\epsilon>0$,
\begin{align*}
   \int \big|f(x)g\(x+\frac{1}{\epsilon}\)\big|dx\leq C\epsilon^2,
\end{align*}
and
\begin{align*}
  \int \big|f\(\frac{x}{\epsilon}\)\big|^2h(x)dx\leq C\epsilon^4.
\end{align*}
\end{lemma}

\begin{proof}
By the assumption on $f$, $g$ and $h$, it follows that for some $C>0$,
\ben\ba
\int \big|f(x)g\(x+\frac{1}{\epsilon}\)\big|dx
&\leq C\int_{|x|\leq \frac{1}{2\epsilon}} \big|f(x)g\(x+\frac{1}{\epsilon}\)\big|dx
+C\int_{|x|\geq \frac{1}{2\epsilon}} \big|f(x)g\(x+\frac{1}{\epsilon}\)\big|dx\\
&\leq C\epsilon^2\int|f(x)|+|g(x)|dx\leq C\epsilon^2.
\ea\enn
Moreover, one has
\ben\ba
\int \big|f\(\frac{x}{\epsilon}\)\big|^2h(x)dx
\leq C\int_{|x|\geq \sigma} \big|f\(\frac{x}{\epsilon}\)\big|^2dx
\leq C\epsilon^4\int_{|x|\geq \sigma} |x|^{-4}dx
\leq C\epsilon^4.
\ea\enn
Thus, the proof is complete.
\end{proof}

\section{Linearized operators around the ground state} \label{Subsec-Coer}
In this appendix we will collect  some coercivity properties of the linearized operators around the ground state.

We first recall from \cite{F-L,KMR} that the ground state is a smooth function satisfying the decay estimate
\be\label{Q-decay}
 |Q(x)|+|\Lambda Q(x)|+|\Lambda^2 Q(x)|\leq \frac{C}{1+|x|^2},
\ee
for some $C>0$.
Let $L=(L_+,L_-)$ be the linearized operator around the ground state,
defined by
\begin{align*}
     L_{+}:= D + I -3Q^2, \ \
    L_{-}:= D +I -Q^2.
\end{align*}
The generalized null space of  $L$ is
spanned by $\{Q, G_1, S_1, \na Q, \Lambda Q, \rho\}$,
where
$G_1$ is the unique odd solution to the equation
\be\label{eq-g1}
L_-G_1=-\na Q,
\ee
$S_1$ is the unique even solution to the equation
\be\label{eq-s1}
L_-S_1=\Lambda Q,
\ee
and
$\rho$ is the unique even solution to the equation
\be\label{eq-rho1}
L_{+}\rho= S_1.
\ee
Thus, we have the following algebraic identities:
\be \ba \label{Q-kernel}
&L_+ \na Q =0,\ \ L_+ \Lambda Q = -2 Q,\ \ L_+ \rho = S_1, \\
&L_{-} Q =0,\ \ L_{-} G_1 = - \na Q,\ \ L_{-} S_1=  \Lambda Q.
\ea\ee

For any complex-valued $H^\half$ function $f = f_1 + i f_2$
in terms of the real and imaginary parts,
we set
\ben
(Lf,f) :=\int f_1L_+f_1dx+\int f_2L_-f_2dx,
\enn
and define the scalar products along the unstable directions
in the null space
\begin{align} \label{Scal-def}
Scal(f):=\<f_1,Q\>^2+\<f_1,G_1\>^2+\<f_1,S_1\>^2+\<f_2,\nabla Q\>^2+\<f_2,\Lambda Q\>^2+\<f_2,\rho_1\>^2.
\end{align}

One has the following coercivity estimate
\begin{lemma} [\cite{K-L-R}, Coercivity estimate]   \label{Lem-coer}
There exists a positive constant $C>0$,  such that
\begin{align} \label{coer}
(Lf,f)\geq&  C\|f\|_{H^\half}^2
             -\frac{1}{C}Scal(f),\ \ \forall f \in H^\half(\R).
\end{align}
\end{lemma}
The proof is based on the non-degeneracy of the linearized operator (\cite{F-L}) and a variational argument,
see \cite[Lemma B.4]{K-L-R} for more details.

In the multi-bubble case, it is important to derive
the following localized version of the coercivity estimate
in the construction of multi-bubble blow-up solutions.

\begin{lemma} [Localized coercivity estimate]   \label{Lem-coer-local}
Let $0<a<1$ and $\phi: \R\rightarrow\R$ be a positive smooth radially decreasing function,
such that
$\phi(x) = 1$ for $|x|\leq 1$,
$\phi(x) = |x|^{-a}$ for $|x|\geq 2$,
$0<\phi \leq 1$.
Set $\phi_A(x) :=\phi(\frac{x}{A})$, $A>0$.
Then, there exists $C>0$ such that
for $A$ large enough,
\begin{align*}
\int (|f|^2+|D^\half f|^2)\phi_A -3Q^2f_1^2-Q^2f_2^2dx
 \geq C\int(|f|^2+|D^\half f|^2)\phi_A dx-\frac{1}{C}Scal(f), \ \ \forall f \in H^\half(\R).
\end{align*}
\end{lemma}

\begin{proof}
Let $\tilde{f}:=f\phi_A^{\frac12}$ with $\tilde{f}=\tilde{f}_1+i\tilde{f}_2$.
We have
\be\label{g}
D^\half f\phi_A^\half=D^\half\tilde{f}+\left(D^\half(\tilde{f}\phi_A^{-\half})-D^\half \tilde{f}\phi_A^{-\half}\right)\phi_A^\half=:D^\half\tilde{f}+g.
\ee
It follows that
\begin{align}\label{m}
&\int (|D^\half f|^2 +|f|^2)\phi_A -3Q^{2}f_1^2-Q^{2}f_2^2dx \nonumber \\
=&\int|D^\half \tilde {f}|^2+|\tilde {f}|^2-3Q^{2} \tilde {f}_1^2-Q^{2} \tilde {f}_2^2dx + \int(1-\phi^{-1}_A)(3Q^{2}\tilde{f}_1^2+Q^{2}\tilde{f}_2^2)dx
 +\|g\|_{L^2}^2
+2{\rm Re} \< D^\half\tilde{f}, g\>  \nonumber \\
=&: K_1+K_2+K_3+K_4.
\end{align}
In the sequel,
let us estimate each term on the right-hand side above.

{\it $(i)$ Estimate of $K_1$.} We claim that there exists $C(A)>0$ with $\lim_{A\rightarrow+\infty}C(A)=0$ such that
\be\label{est-scal}
|Scal(\tilde{f})- Scal(f)|\leq C(A)\|\tilde{f}\|_{L^2}^2.
\ee
This along with Lemma \ref{Lem-coer} and \eqref{est-scal}
yields that there exists $C>0$ such that for $A$ sufficiently large,
\be\label{k1}
K_1\geq C\|\tilde{f}\|^2_{H^\half}-\frac{1}{C}Scal(f).
\ee

In order to prove \eqref{est-scal}, we rewrite
\ben
\langle\tilde{f}_1, Q\rangle=\langle f_1, Q\rangle+\int \tilde{f}_1Q(\phi_A^\half-1)\phi_A^{-\half}dx.
\enn
Note that
\ben
|\int \tilde{f}_1Q(\phi_A^\half-1)\phi_A^{-\half}dx|\leq C\int_{|x|\geq A} |\tilde{f}_1Q\phi_A^{-\half}|dx
\leq C\|\tilde{f}\|_{L^2}\left(\int_{|x|\geq A}Q^2\phi_A^{-1}dx\right)^\half.
\enn
By the decay property that $Q(y)\sim\<y\>^{-2}$,
\ben
\int_{|x|\geq A}Q^2\phi_A^{-1}dx=\frac{1}{A}\int_{|y|\geq 1}Q^2(Ay)\phi^{-1}(y)dy
\leq C\frac{1}{A^5}\int_{|y|\geq 1}|y|^{a-4}dy\leq \frac{C}{A^5}.
\enn
Thus, we get
\ben
|\langle\tilde{f}_1, Q\rangle-\langle f_1, Q\rangle|\leq CA^{-\frac{5}{2}}\|\tilde{f}\|_{L^2}.
\enn
Similar arguments also apply to the remaining five unstable directions,
and thus we obtain \eqref{est-scal}, as claimed.

{\it $(ii)$ Estimate of $K_2$.}
Using the decay property of $Q$ again we see that
\be\label{k2}
|K_2|\leq C\int_{|x|\geq A}\phi^{-1}_AQ^{2}|\tilde{f}|^2dx\leq C\|\phi^{-1}_AQ^2\|_{L^\infty(|x|\geq A)}\|\tilde{f}\|_{L^2}^2\leq A^{-4}\|\tilde{f}\|_{L^2}^2.
\ee

{\it $(iii)$ Estimate of $K_3$ and $K_4$.} We claim that there exists $C(A)>0$ with $\lim_{A\rightarrow+\infty}C(A)=0$ such that
\be\label{est-g}
\|g\|_{L^2}\leq C(A)\|\tilde{f}\|_{L^2}.
\ee
This yields that
\be\label{k34}
|K_3|+|K_4|\leq C(A)\|\tilde{f}\|^2_{H^{\frac 12}}.
\ee

In order to prove \eqref{est-g},
by Lemma \ref{lem-formula} and the Minkowski inequality,
\ben\ba
\|g\|_{L^2}&=C\left(\int \phi_A(x)\left|\int \frac{\tilde{f}(x+y)(\phi_A^{-\frac12}(x+y)-\phi_A^{-\frac12}(x))
-\tilde{f}(x-y)(\phi_A^{-\frac12}(x)-\phi_A^{-\frac12}(x-y))}
{|y|^{\frac{3}{2}}}dy\right|^2dx\right)
^{\half}\\
& \leq C\int|y|^{-\frac32}\left(\int \phi_A(x)\left(\tilde{f}(x+y)(\phi_A^{-\frac12}(x+y)-\phi_A^{-\frac12}(x))\right)^2dx\right)
^{\frac12}dy\\
& \leq C\sum_{j=1}^{4}\int_{\Omega_{j,1}}|y|^{-\frac32}
\left(\int_{\Omega_{j,2}} \phi_A(x)\left(\tilde{f}(x+y)(\phi_A^{-\frac12}(x+y)-\phi_A^{-\frac12}(x))\right)^2dx\right)
^{\frac12}dy\\
&=:G_1+G_2+G_3+G_4,
\ea\enn
where $\Omega_1: =\Omega_{1,1} \cup \Omega_{1,2}=\{|x|\leq \frac{A}{2}, |y|\leq \frac{A}{4}\}$.
For the sake of simplicity,
we write $\Omega_j := \Omega_{j,1} \cup \Omega_{j,2}$ for the remaining three regimes, $2\leq j\leq 4$,
and take
$\Omega_2:=\{|x|\geq \frac{A}{2}, |y|\leq \frac{A}{4}\}$, $\Omega_3:=\{|x|\leq \frac{|y|}{2}, |y|\geq \frac{A}{4}\}$ and $\Omega_4:=\{|x|\leq \frac{|y|}{2}, |y|\geq \frac{A}{4}\}$.
Then, the proof of \eqref{est-g} is reduced to estimating $G_k$, $1\leq k\leq 4$.

First note that $\phi_A^{-\frac12}(x+y)-\phi_A^{-\frac12}(x)=0$ for $(x, y)\in \Omega_1$,
and so  $G_1=0$.

Moreover, using the mean valued theorem we get that for some $0\leq \t\leq 1$,
\ben
\phi_A(x)|\phi_A^{-\frac12}(x+y)-\phi_A^{-\frac12}(x)|^2=\frac{1}{4}\frac{\phi_A(x)}{\phi^3_A(x+\t y)}|\phi^\prime_A(x+\t y)|^2|y|^2.
\enn
For  $\frac{A}{2}\leq |x|\leq 3A$ and $|y|\leq \frac{A}{4}$, we have $\frac{A}{4}\leq |x+\t  y|\leq 4A$,
and thus
\ben
\phi_A(x)|\phi_A^{-\frac12}(x+y)-\phi_A^{-\frac12}(x)|^2=\frac{1}{4A^2}\frac{\phi_A(x)}{\phi^3_A(x+\t y)}|\phi^\prime(\frac{x+\t y}{A})|^2|y|^2\leq \frac{C}{A^2}|y|^2.
\enn
While for $|x|\geq 3A$ and $|y|\leq \frac{A}{4}$,
we have $\frac{|x|}{2}\leq |x+\t y|\leq \frac{3|x|}{2}$,
and thus
\ben
\phi_A(x)|\phi_A^{-\frac12}(x+y)-\phi_A^{-\frac12}(x)|^2
\leq C\frac{\phi_A(\frac x2)}{\phi^3_A(\frac {3x}{2})}|\phi^\prime_A(x+\theta y)|^2|y|^2
\leq \frac{C}{A^2}|y|^2.
\enn
Hence, we obtain
\be\label{g2}
G_2\leq \frac{C}{A}\int_{0}^{\frac{A}{4}}y^{-\half}dy\|\tilde{f}\|_{L^2}\leq CA^{-\half}\|\tilde{f}\|_{L^2}.
\ee

Regarding $G_3$, we shall use the fact that
\be\label{1}
\phi_A(x)|\phi_A^{-\frac12}(x+y)-\phi_A^{-\frac12}(x)|^2\leq C\phi_A(x)(\phi_A^{-1}(x+y)+\phi_A^{-1}(x)).
\ee
For $(x, y)\in \Omega_3$,
we have $\frac{|y|}{2}\leq |x+y|\leq \frac{3|y|}{2}$, so
\ben
\phi_A(x)\phi_A^{-1}(x+y)\leq C\phi_A^{-1}(\frac{3|y|}{2})\leq C|y|^a,
\enn
which implies
\ben
\phi_A(x)|\phi_A^{-\frac12}(x+y)-\phi_A^{-\frac12}(x)|^2\leq C(1+|y|^a),
\enn
and thus
\be\label{g3}
G_3\leq C\int_{\frac{A}{4}}^{+\infty}y^{-\frac{3}{2}}(1+y^{\frac{a}{2}})dy\|\tilde{f}\|_{L^2}
\leq CA^{\frac{a-1}{2}}\|\tilde{f}\|_{L^2}.
\ee

The proof for the estimate of $G_4$ also relies on \eqref{1}.
In fact, for $(x, y)\in \Omega_4$, we have $|x+y|\leq 3|x|$,
and thus
\ben
\phi_A(x)\phi_A^{-1}(x+y)\leq C\phi_A(x)\phi_A^{-1}(3x)\leq C,
\enn
and, by \eqref{1},
\ben
\phi_A(x)|\phi_A^{-\frac12}(x+y)-\phi_A^{-\frac12}(x)|^2\leq C.
\enn
It follows that
\be\label{g4}
G_4\leq C\int_{\frac{A}{4}}^{+\infty}y^{-\frac{3}{2}}dy\|\tilde{f}\|_{L^2}\leq CA^{-\half}\|\tilde{f}\|_{L^2}.
\ee
Thus, combing\eqref{g2}, \eqref{g3} and \eqref{g4} altogether we obtain \eqref{est-g}, as claimed.

Now, plugging \eqref{k1}, \eqref{k2} and \eqref{k34} into \eqref{m}
we get that for $A$ large enough,
\ben
\int (|D^\half f|^2 +|f|^2)\phi_A -3Q^{2}f_1^2-Q^{2}f_2^2dx
\geq C_1\|\tilde{f}\|_{H^\half}^2-C_2Scal(f).
\enn
Hence, in order to finish the proof,
it remains to show that there exists $C>0$ such that for $A$ sufficiently large
\be \label{wtf-H12-esti}
\|\tilde{f}\|_{H^\half}^2\geq C\int(|D^\half f|^2+|f|^2)\phi_Adx.
\ee
Actually, since $\|\tilde{f}\|^2_{L^2}=\int|f|^2\phi_Adx$ and
\ben\ba
\|D^\half \tilde{f}\|^2_{L^2}&=\int|D^\half f\phi_A^\half+(D^\half\tilde{f}-D^\half(\tilde{f}\phi_A^{-\half})\phi_A^\half)|^2dx
=\int|D^\half f\phi_A^\half-g|^2dx\\
&=\int|D^\half f|^2\phi_Adx+\int|g|^2dx-2{\rm Re}\int D^\half f\phi_A^\half\ol{g}dx,
\ea\enn
where $g$ is given by \eqref{g},
we apply \eqref{est-g} to derive that for $A$ large enough
\ben
\|D^\half \tilde{f}\|^2_{L^2}+\|\tilde{ f}\|^2_{L^2}\geq C\int(|D^\half f|^2+|f|^2)\phi_Adx.
\enn
This yields \eqref{wtf-H12-esti} and finishes the proof of Lemma \ref{Lem-coer-local}.
\end{proof}

\medskip

\noindent{\bf Acknowledgements.}
D. Cao is supported by NSFC(Grant No. 12371212). Y. Su is supported by NSFC (No. 12371122).
D. Zhang  is grateful for the support by NSFC (No. 12271352, 12322108)
and Shanghai Rising-Star Program 21QA1404500.

\end{document}